\documentclass[reqno,12pt,letterpaper]{amsart}
\usepackage{amsmath,amssymb,amsthm,graphicx,mathrsfs,url}
\usepackage[usenames,dvipsnames]{color}
\usepackage[colorlinks=true,linkcolor=Red,citecolor=Green]{hyperref}
\usepackage{floatrow}
\usepackage[usenames,dvipsnames]{xcolor}
\usepackage{tikz} 
\usepackage{slashed}
\usepackage{graphicx}
\usepackage[colorlinks=true]{hyperref}
\usepackage{color}
\usepackage{amsmath,amsfonts}
\usepackage{calrsfs}
\usepackage{amsmath,environ}
\usepackage{floatrow}
\DeclareMathAlphabet{\pazocal}{OMS}{zplm}{m}{n}
\usepackage{autonum}

\numberwithin{equation}{section}
\usepackage{amsmath, amsthm}
    \usepackage{dsfont}
    \usepackage{fancybox}
    \usepackage{amssymb}

\newcommand{\Aa}{{\mathbb{A}}}
\newcommand{\bsu}{\boldsymbol{\sigma_1}}

\newcommand{\bsd}{\boldsymbol{\sigma_2}}
\newcommand{\bst}{\boldsymbol{\sigma_3}}
\newcommand{\bsi}{\boldsymbol{\sigma_i}}
\newcommand{\bsj}{\boldsymbol{\sigma_j}}

\newcommand{\bmu}{\boldsymbol{m_1}}

\newcommand{\bmd}{\boldsymbol{m_2}}
\newcommand{\bmt}{\boldsymbol{m_3}}

\newcommand{\bmj}{\boldsymbol{m_j}}
\newcommand{\bmk}{\boldsymbol{m_k}}

\newcommand{\Di}{\slashed{D}}

\newcommand{\tvar}{\var_\sharp}

\newcommand{\tmu}{{\mu_\sharp}}
\newcommand{\mub}{{\mu_\flat}}

\newcommand{\QQ}{{\pazocal{Q}}}
\newcommand{\QQQ}{{\mathscr{Q}}}
\newcommand{\VVV}{{\mathscr{V}}}
\newcommand{\KKK}{{\mathscr{K}}}
\newcommand{\UUU}{{\mathscr{U}}}
\newcommand{\WWW}{{\mathscr{W}}}
\newcommand{\R}{\mathbb{R}}

\newcommand{\p}{{\partial}}

\newcommand{\w}{\omega}

\newcommand{\matrice}[1]{\left[ \begin{matrix}
#1
\end{matrix} \right]}
\newcommand{\Det}{{\operatorname{Det}}}
\newcommand{\PP}{{\mathscr{P}}}
\newcommand{\Pp}{{\mathbb{P}}}
\newcommand{\PPP}{{\mathscr{P}}}

\newcommand{\loc}{{\text{loc}}}
\newcommand{\epsi}{\varepsilon}
\newcommand{\sgn}{{\operatorname{sgn}}}

\newcommand{\te}{\theta}
\newcommand{\lr}[1]{\langle #1 \rangle}
\newcommand{\blr}[1]{\left\langle #1 \right\rangle}
\newcommand{\bblr}[1]{\big\langle #1 \big\rangle}

\newcommand{\tW}{\Ww}

\newcommand{\WW}{\pazocal{W}}

\newcommand{\tPP}{\Pp}
\newcommand{\Z}{\mathbb{Z}}
\newcommand{\Dd}{\mathbb{D}}
\newcommand{\Ll}{\mathbb{L}}

\newcommand{\Ww}{\mathbb{W}}
\newcommand{\Ss}{\mathbb{S}}
\newcommand{\Id}{{\operatorname{Id}}}

\newcommand{\EE}{\pazocal{E}}
\newcommand{\CCC}{\pazocal{C}}

\newcommand{\II}{\pazocal{I}}

\newcommand{\UU}{{\pazocal{U}}}
\newcommand{\KK}{{\pazocal{K}}}

\newcommand{\HH}{\pazocal{H}}

\newcommand{\vp}{{\varphi}}

\newcommand{\ove}[1]{{\overline{#1}}}
\newcommand{\systeme}[1]{\left\{ \begin{matrix} #1 \end{matrix} \right.}

\newcommand{\C}{\mathbb{C}}

\newcommand{\var}{{\vartheta}}

\newcommand{\az}{\alpha}

\newcommand{\OO}{{\mathscr{O}}}

\newcommand{\dist}{{\operatorname{dist}}}

\renewcommand{\Re}{\operatorname{Re}}

\newcommand{\pp}{{\operatorname{pp}}}
\newcommand{\ess}{{\operatorname{ess}}}
\renewcommand{\Im}{\operatorname{Im}}

\newcommand{\tDi}{\slashed{\Dd}}

\newcommand{\otau}{\overline{\tau}}

\newcommand{\lamdba}{{\lambda}}
\newcommand{\1}{\mathds{1}}
\newcommand{\RR}{\pazocal{R}}
\newcommand{\de}{ \ \mathrel{\stackrel{\makebox[0pt]{\mbox{\normalfont\tiny def}}}{=}} \ }

\title[Edge states in honeycomb structures]{Characterization of edge states in perturbed honeycomb structures}
\author{Alexis Drouot}

\NewEnviron{equations}{%
\begin{equation}\begin{gathered}
  \BODY
\end{gathered}\end{equation}
}

\newtheorem{thm}{Theorem}
\newtheorem{cor}{Corollary}

\newtheorem{defi}{Definition}
\newtheorem{lem}{Lemma}[section]

\newtheorem{conj}{Conjecture}
\newtheorem{theorem}[thm]{Theorem}

\begin{document}
\vspace*{-1cm}
\maketitle

\begin{abstract} This paper is a mathematical analysis of conduction effects at interfaces between insulators. Motivated by Haldane--Raghu \cite{HR,RH}, we continue the study of a linear PDE initiated in Fefferman--Lee-Thorp--Weinstein \cite{FLTW3,FLTW4}. This PDE is induced by a continuous honeycomb Schr\"odinger operator with a line defect. 

This operator exhibits remarkable connections between topology and spectral theory. It has essential spectral gaps about the Dirac point energies of the honeycomb background. In a perturbative regime, the authors of  \cite{FLTW3} construct edge states: time-harmonic waves propagating along the interface, localized transversely. At leading order, these edge states are adiabatic modulations of the Dirac point Bloch modes. Their envelops solve a Dirac equation that emerges from a multiscale procedure.

We develop a scattering-oriented approach that derives \textit{all} possible edge states, at arbitrary precision. The key component is a resolvent estimate connecting the Schr\"odinger operator to the emerging Dirac equation. We discuss topological implications via the computation of the spectral flow, or edge index. 
\end{abstract}


\section{Introduction and results}

A central branch of condensed matter physics studies energy propagation between dissimilar media. In favorable conditions, the interface acts like a unidirectional channel for electronic transport: the material is conducting in the edge direction but remains insulating transversely.  In experiments, this property is remarkably robust: it persists even if the interface becomes bent, sharp or disordered. The first theoretical investigations concerned the quantum Hall effect \cite{AMU,KDP,Hal,TKN,Hat}. The research has since focused on topological insulators \cite{KM1,KM2,FKM,MB,HQW,Roy,ZLQ,JMD}, together with their applications in electronics, photonics, acoustics, mechanics and geophysics \cite{KMY,YVW,WCJ,SG11,RZP,I15,M17,DMV,O18,PDV}.

Energy transport along the interface may be interpreted as a bifurcation phenomenon. In certain periodic materials, the introduction of an edge forces Bloch modes to bifurcate into edges states: time-harmonic waves propagating along rather than across the edge. This seemingly goes back to Tamm \cite{Ta}, who looked at bifurcations from local extrema in the band spectrum. Shockley \cite{Sh} next studied bifurcations from linear crossings in the band spectrum on a one-dimensional example. In contrast with Tamm's work, Shockley's analysis applies to insulators with narrow energy gaps. It was later discovered that Shockley's states may be topologically protected: \textit{they may persist against large local perturbations.}

Honeycomb structures are invariant under $2\pi/3$-rotation and  spatial inversion. These symmetries generate Dirac points: conical degeneracies in the band spectrum. Impurities breaking spatial inversion split the dispersion surfaces away and open energy gaps: the material transits from a metal to an insulator. Here we analyze interface effects at the junction of two such insulators. 

Motivated by  Haldane--Raghu \cite{HR,RH}, Fefferman, Lee-Thorp and Weinstein \cite{FLTW3}  introduced a PDE that models parity-breaking perturbations of a \textit{continuous} honeycomb lattice (see \S\ref{sec:1.1}-\ref{sec:1.2}). The perturbed operator exhibits (a) an edge that separates two asymptotically periodic near-honeycomb structures; (b)  gaps in the essential spectrum centered at Dirac  point energies of the honeycomb background. Under a spectral condition on the unperturbed operator (see \cite[\S1.3]{FLTW3} and \S\ref{sec:1.3}), Fefferman, Lee-Thorp and Weinstein designed edge states as adiabatic modulations of the Dirac point Bloch modes. Their envelops are eigenvectors of a Dirac operator produced via a multiscale procedure. See \cite[Theorem 7.3]{FLTW3}.

Here, we follow instead a scattering approach. We recover the results of \cite{FLTW3,FLTW4}. In addition, we obtain:
\begin{itemize}
\item A resolvent estimate connecting the initial PDE to the emerging Dirac equation.
\item The complete characterization of edges states in the energy gap.
\item Full expansions of the edge states at all order in the size of the perturbation.
\end{itemize}
See \S\ref{sec:1.5} and \S\ref{sec:4.3} for precise statements. 

The full identification of edge states represents the most significant advance. It allows to interpret the results topologically. In \S\ref{sec:1.7}, we compute the signed number of eigenvalues that move accross Dirac point energies when the edge-parallel quasimomentum runs from $0$ to $2\pi$. This is a topological invariant of the system -- called spectral flow or edge index -- and it vanishes here. This calculation confirms numerical simulations \cite{RH,FLTW4,LWZ}. It corroborates the prediction of the Kitaev table \cite{Ki,RSFL} combined with the bulk-edge correspondence: breaking spatial inversion while keeping time-reversal invariance does not create protected edge states. 

In the last part of the work, we consider a magnetic analog of the operator studied in \cite{FLTW3,FLTW4}, similar to those of \cite{RH,HR,LWZ}. It models time-reversal breaking instead of parity breaking. We show that the corresponding spectral flow equals either $2$ or $-2$. This confirms the existence of at least two toplogically protected, unidirectionally propagating waves along the edge; see \cite{HR}, the Kitaev table \cite{Ki,RSFL} as well as the numerical results \cite{RH,LWZ}.

\subsection{Periodic operators and Dirac points}\label{sec:1.1} We start with a description of honeycomb   potentials as in \cite{FW}. Let $\Lambda$ be the equilateral $\Z^2$-lattice. It is generated by two vectors $v_1$ and $v_2$, given in canonical coordinates by
\begin{equation}\label{eq:4b}
v_1 = a\matrice{\sqrt{3} \\ 1}, \ \ \ v_2= a\matrice{\sqrt{3} \\ -1},
\end{equation}
where  $a > 0$ is a constant such that $\Det[v_1,v_2] = 1$. The dual basis $k_1, k_2$ consists of two vectors in $(\R^2)^*$ which satisfy $\blr{k_i,v_j} = \delta_{ij}$. The dual lattice is $\Lambda^* = \Z k_1 \oplus \Z k_2$. The corresponding fundamental cell and dual fundamental cell are 
\begin{equation}\label{eq:4c}
\Ll \de \big\{ sv_1 + s' v_2 : s, s' \in [0,1)  \big\}, \ \ \ \ \Ll^* \de \big\{ \tau k_1 + \tau' k_2 : \tau, \tau' \in [0,2\pi)  \big\}.
\end{equation}

\begin{defi}\label{def:3} We say that $V \in C^\infty\big(\R^2,\R\big)$ is a honeycomb potential if:
\begin{itemize}
\item $V$ is $\Lambda$-periodic: $V(x+w) = V(x)$ for $w \in \Lambda$.
\item $V$ is even: $V(x) = V(-x)$.
\item $V$ is invariant under the $2\pi/3$ rotation:
\begin{equation}
V(Rx) = V(x), \ \ \ \ R \de \dfrac{1}{2} \matrice{ -1 & \sqrt{3} \\ -\sqrt{3} & -1 }.
\end{equation}
\end{itemize}
\end{defi}

\vspace*{-.5cm}

\begin{center}
\begin{figure}
\caption{The equilateral lattice with its generating vectors $v_1,v_2$ and dual vectors $k_1, k_2$ together with the fundamental cell $\Ll$. 
}\label{fig:4}
{\includegraphics{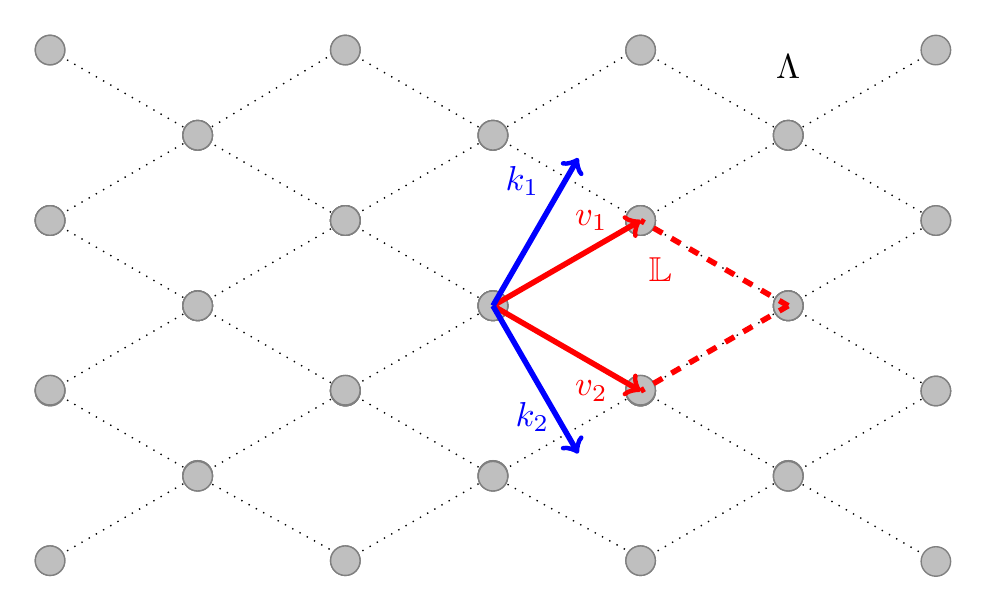}}
\end{figure}
\end{center}

A simple example of honeycomb potential is the periodization of a radial function over the lattice
\begin{equation}
\left(\dfrac{v_1 + v_2}{3}+\Lambda \right) \cup \left(\dfrac{2v_1 + 2v_2}{3}+\Lambda \right),
\end{equation}
see Figure \ref{fig:1}. Given a honeycomb potential $V$, we will study spatially delocalized perturbations of the (unbounded) Schr\"odinger operator
\begin{equation}
P_0 \de -\Delta+V \ : \ L^2\big(\R^2,\C\big) \rightarrow L^2\big(\R^2,\C\big),
\end{equation}
with domain $H^2(\R,\C)$. This operator is periodic with respect to $\Lambda$. This allows to use Floquet--Bloch theory -- see \cite[\S XIII]{RS}: $P_0$ leaves  the space
\begin{equation}
L^2_\xi \de \Big\{ u\in L^2_\loc\big(\R^2,\C\big) : \ u(x+w) = e^{i\blr{\xi,w}} u(x),  \ w \in \Lambda  \Big\}, \ \ \ \ \xi \in \R^2
\end{equation}
invariant. The space $L^2_\xi$ is Hilbertian when equipped with the Hermitian form
\begin{equation}
\blr{f,g}_{L^2_\xi} \de \int_{\Ll} \ove{f(x)} g(x) dx.
\end{equation}
Let $P_0(\xi)$ be formally equal to $P_0 = -\Delta + V$, but acting on $L^2_\xi$. It has compact resolvent and discrete spectrum -- denoted below by $\Sigma_{L^2_\xi}\big(P_0(\xi)\big)$ --  depending on $\xi$:
\begin{equation}
\lambda_{0,1}(\xi) \leq \lambda_{0,2}(\xi) \leq \dots \leq \lambda_{0,j}(\xi) \leq \dots
\end{equation}
The maps $\xi \in \R^2 \mapsto \lambda_{0,j}(\xi)$ are called dispersion surfaces of $P_0$. The $L^2$-spectrum of $P_0$ consists of the ranges of the dispersion surfaces: it equals
\begin{equation}
\Sigma_{L^2}(P_0) = \bigcup_{\xi \in \R} \Sigma_{L^2_\xi}\big( P_0(\xi) \big) =\Big\{ \lambda_{0,j}(\xi) :  \ j \geq 1, \ \xi \in \R^2 \Big\}.
\end{equation}

\vspace*{-.5cm}

\begin{center}
\begin{figure}
\caption{If each gray circle represents the same radially symmetric function -- the \textit{atomic potential} -- the resulting potential has the honeycomb symmetry.} \label{fig:1}
{\includegraphics{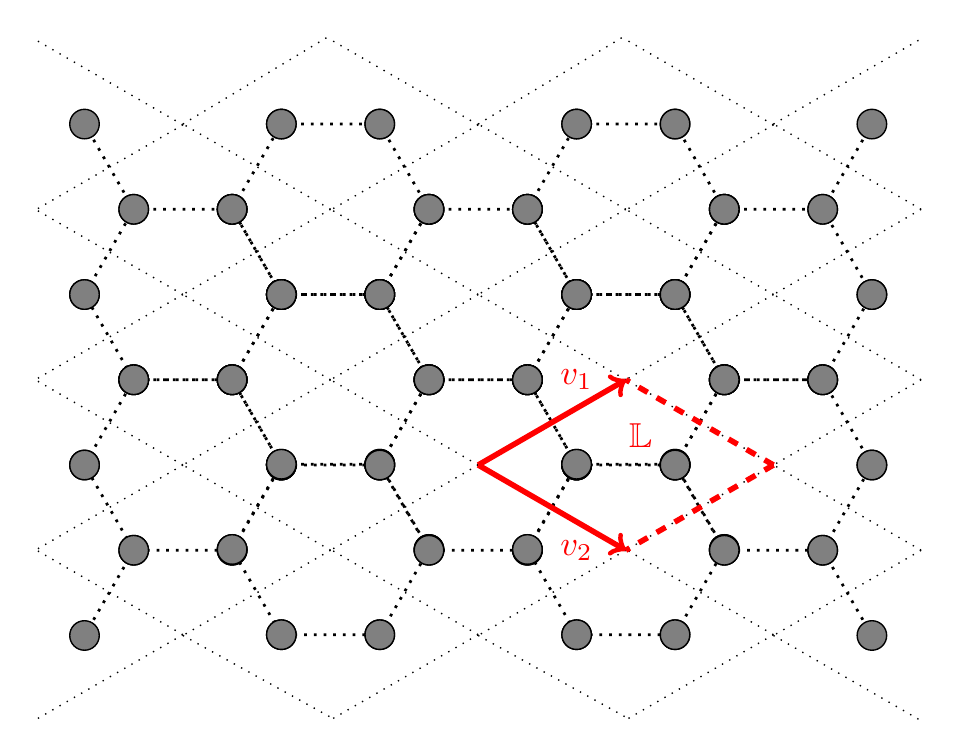}}
\end{figure}
\end{center}

We now discuss Dirac points. Roughly speaking, they correspond to the conical degeneracies in the band spectrum of $P_0$.

\begin{defi}\label{def:1} A pair $(\xi_\star,E_\star) \in \R^2 \times \R$ is a Dirac point of $P_0 = -\Delta+V$ if:
\begin{itemize}
\item[(i)] $E_\star$ is a $L^2_{\xi_\star}$-eigenvalue of $P_0(\xi_\star)$
of multiplicity $2$;
\item[(ii)] There exists an orthonormal basis $\{\phi_1,\phi_2\}$ of $\ker_{L^2_{\xi_\star}}\hspace*{-1.5mm}\big(P_0(\xi_\star)-E_\star\big)$ such that 
\begin{equation}\label{eq:5a}
\phi_1(Rx) = e^{2i\pi/3} \phi_1(x), \ \ \ \ \phi_2(x) = \ove{\phi_1(-x)}, \ \ \ \ \phi_2(Rx) = e^{-2i\pi/3} \phi_2(x).
\end{equation}
\item[(iii)] There exist $j_\star \geq 1$ and $\nu_F > 0$ such that  for $\xi$ close to $\xi_\star$,
\begin{equations}
\lambda_{0,j_\star}(\xi) = E_\star - \nu_F \cdot |\xi-\xi_\star| + O(\xi-\xi_\star)^2,
 \\
\lambda_{0,j_\star+1}(\xi) = E_\star + \nu_F \cdot |\xi-\xi_\star| + O(\xi-\xi_\star)^2.
\end{equations}
\end{itemize}
\end{defi}

When $V$ is a honeycomb potential, Fefferman--Weinstein \cite{FW} showed that $P_0=-\Delta + V$ generically admits Dirac points $(\xi_\star,E_\star)$. We refer to \cite{FW} for details and to \S\ref{sec:3.2} for a review of the identities needed here. Because of \eqref{eq:5a}, $(\xi_\star,E_\star)$ must satisfy
\begin{equation}\label{eq:3y}
\xi_\star \in \big\{\xi_\star^A,\xi_\star^B\big\} \mod 2\pi \Lambda^*, \ \ \ \ 
\xi_\star^A \de \dfrac{2\pi}{3} (2k_1+k_2), \ \ \ \ \xi_\star^B \de \dfrac{2\pi}{3} (k_1+2k_2).
\end{equation}
See Figure \ref{fig:10}. Symmetries impose that $(\xi_\star^A,E_\star)$ is a Dirac point of $P_0$ if and only if $(\xi_\star^B,E_\star)$ is a Dirac point of $P_0$. We call the pair $(\phi_1,\phi_2)$ of \eqref{eq:5a} a Dirac eigenbasis.

As observed in \cite{FW}, Dirac points are stable against small perturbations preserving spatial inversion (parity) and time-reversal symmetry (conjugation). Conversely, breaking parity (while keeping conjugation invariance) generically opens spectral gaps about Dirac point energies. For $\delta \neq 0$, we introduce the operator
\begin{equations}\label{eq:5d}
P_\delta \de P_0 + \delta W = -\Delta + V + \delta W, \ \ \ \ \text{where:} 
\\
W \in C^\infty\big(\R^2,\R\big); \ \ \ W(x+w) = W(x), \ w \in \Lambda; \ \ \ W(-x)=-W(x).
\end{equations}
We will assume in the rest of the paper that the non-degeneracy condition
\begin{equation}\label{eq:3h}
\var_\star \de \blr{\phi_1, W\phi_1}_{L^2_{\xi_\star}} \neq 0
\end{equation}
holds. Under this condition, if $(\xi_\star,E_\star)$ is a Dirac point of $P_0$, then the operator $P_\delta(\xi_\star)$ (equal to $P_\delta$, but acting on $L^2_{\xi_\star}$) admits a $L^2_{\xi_\star}$-spectral gap centered at $E_\star$:
\begin{equation}\label{eq:3j}
\dist\Big(\Sigma_{L^2_{\xi_\star}}\big(P_0(\xi_\star) \big), \ E_\star\Big) = \var_F \cdot \delta + O\big(\delta^2\big), \ \ \ \ \var_F \de |\var_\star|.
\end{equation}
This gap has width $2\var_F \cdot \delta + O(\delta^2)$; see Figure \ref{fig:10}.
This is a simple fact proved via perturbation analysis -- see e.g. \cite[Remark 9.2]{FW} or \S\ref{sec:5.3}. 
 Whether this $L^2_{\xi_\star}$-spectral gap extends to a \textit{global} $L^2$-gap of $P_\delta$ depends on the global behavior of the dispersion surfaces of $P_0$; see \cite[\S1.3 and \S8]{FLTW3}. When it does, the operators $P_{\pm \delta}$ describe insulators at energy $E_\star$ with a narrow gap centered at $E_\star$. These materials are parity-breaking perturbations of the \textit{metal} modeled by $P_0$.
 
 \begin{center}
\begin{figure}
\caption{The picture on the left represents the Dirac points $\xi_\star^A$ and $\xi_\star^B$ inside a dual fundamental cell $\Ll^*$. The two pictures on the right represent the bifurcation of a Dirac point $(\xi_\star,E_\star)$ to an open gap on a one-dimensional section of the Brillouin zone. } \label{fig:10}
{\includegraphics{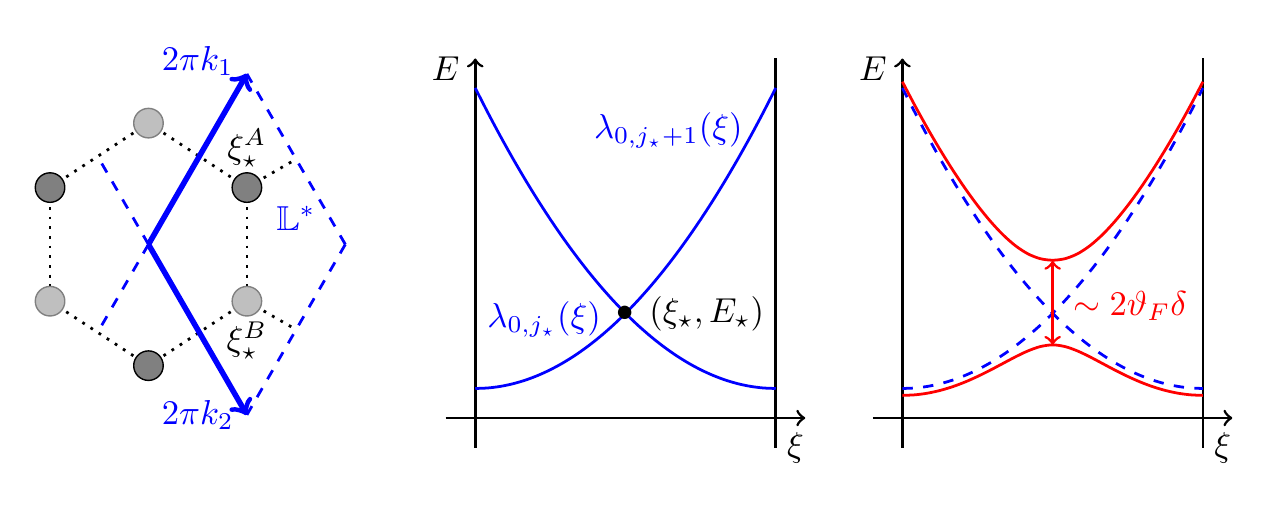}}
\end{figure}
\end{center}

\subsection{Edges and the model}\label{sec:1.2} We now describe the model of Fefferman--Lee-Thorp--Weinstein for honeycomb operators with an edge \cite{FLTW3,FLTW4}. Fix $v = a_1 v_1 + a_2 v_2 \in \Lambda$ with $a_1, a_2 \in \Z$ relatively prime, representing the direction of an edge $\R v$. We introduce $v' \in \Lambda$, $k, \ k' \in \Lambda^*$ such that
\begin{equations}\label{eq:3w}
v' \de b_1 v_1 + b_2 v_2, \ \ \ \ a_1 b_2 - a_2 b_1 = 1, \ \ b_1, \ b_2 \in \Z, \\
k \de b_2 k_1 - b_1 k_2, \ \ \ \ k' \de -a_2 k_1 + a_1 k_2.
\end{equations} 
The pairs $(v,v')$ and $(k,k')$ are dual to one another and span $\Lambda$ and $\Lambda^*$. See \S\ref{sec:6.1}.

Recall that $P_{\pm \delta} = -\Delta +V \pm \delta W$. Fefferman--Lee-Thorp--Weinstein \cite{FLTW3,FLTW4} analyze an operator $\PPP_\delta$ that describes an adiabatic transition from $P_{-\delta}$ to $P_\delta$  transversely to the edge $\R v$. Specifically,
\begin{equation}
\PPP_\delta \de P_0 + \delta \cdot  \kappa_\delta \cdot W = -\Delta + V + \delta \cdot \kappa_\delta \cdot W.
\end{equation}
Above, the function $\kappa_\delta \in C^\infty\big(\R^2,\R\big)$ is an adiabatic modulation of a domain wall $\kappa \in C^\infty(\R,\R)$ along $\R v$:
\begin{equation}\label{eq:3v}
\kappa_\delta(x) = \kappa(\delta \blr{k',x}), \ \ \ \ \exists L > 0, \ \  \kappa(t) = \systeme{ -1 \text{ when } x \leq -L,
\\
1 \ \ \text{ when } \ \ x \geq L.}
\end{equation}
The operator $\PP_\delta$ is a Schr\"odinger operator with potential represented in Figure \ref{fig:12}. It models the soft junction of two insulators modeled by $P_{\pm \delta}$ along the interface~$\R v$.

Although $\PP_\delta$ is not periodic with respect to $\Lambda$, it is periodic with respect to $\Z v$ because $\blr{k',v} = 0$. For every $\zeta \in \R$, $\PP_\delta$ acts as an unbounded operator on
\begin{equation}\label{eq:4n}
L^2[\zeta] \de \left\{ u \in L^2_\loc(\R^2,\C), \ u(x+v) = e^{i\zeta} u(x), \ \int_{\R^2/\Z v} |u(x)|^2 dx   < \infty\right\},
\end{equation}
 with domain $H^2[\zeta]$ -- defined according to \eqref{eq:4n}. Let $\PP_\delta[\zeta]$ be the resulting operator.

We continue the analysis of \cite{FLTW3,FLTW4}: we study the electronic properties of the material modeled by $\PP_\delta$. We investigate whether energy propagates along the edge $\R v$. This boils down to studying \textit{edge states} of $\PP_\delta$. These are time-harmonic waves propagating along $\R v$ and localized transversely to $\R v$. Mathematically, they are the $L^2[\zeta]$-eigenvectors of $\PP_\delta[\zeta]$. Such states correspond to diffusion-less electronic channels along $\R v$; they have great potential in technological applications.

\subsection{The no-fold condition of Fefferman--Lee-Thorp--Weinstein \cite{FLTW3}}\label{sec:1.3} We set $\zeta_\star = \blr{\xi_\star,v}$ and $\zeta_\star^J = \blr{\xi_\star^J,v}$. Thanks to \eqref{eq:3y},
\begin{equation}\label{eq:4k}
\zeta_\star^A = \dfrac{2\pi}{3}(2a_1+a_2); \ \ \ \ \zeta_\star^B = \dfrac{2\pi}{3}(a_1+2a_2).
\end{equation}
Hence, $\zeta_\star \in \{0,2\pi/3,4\pi/3\} \mod 2\pi\Z$.
Recall the no-fold condition \cite[\S1.3]{FLTW3}.

\begin{defi}\label{def:2} The no-fold condition holds along the edge $\R v$ at $\zeta_\star$  if
\begin{equation}
\forall j \geq 1, \ \ \forall \tau \in \R, \ \ \lambda_{0,j}\big(\zeta_\star k + \tau k'\big) = E_\star \ \  \Rightarrow \  \  j \in \{j_\star,j_\star+1\} \text{ and } \tau = \blr{\xi_\star,v'} \hspace*{-2mm} \mod 2\pi. 
\end{equation}
\end{defi}

The essential spectrum of $\PP_\delta[\zeta_\star]$ is obtained from the (essential) spectra of the bulk operators $P_{\pm \delta}[\zeta_\star]$ (the operators formally equal to $P_{\pm \delta}$, but acting on $L^2[\zeta_\star]$). These are conjugated under spatial inversion. Therefore they have the same spectrum. From Floquet--Bloch theory,
\begin{equation}\label{eq:2z}
\Sigma_{L^2[\zeta_\star],\ess}\big(\PP_\delta[\zeta_\star]\big) = \Sigma_{L^2[\zeta_\star]}\big(P_{-\delta}[\zeta_\star]\big) \cup \Sigma_{L^2[\zeta_\star]}\big(P_\delta[\zeta_\star]\big) = \bigcup_{\xi \in \zeta_\star k + \R k'} \Sigma_{L^2_\xi}\big(P_{\delta}(\xi)\big).
\end{equation}
If $(\xi_\star,E_\star)$ is a Dirac point of $P_0$ and $\var_\star \neq 0$, then for small $\delta$, $P_{\pm \delta}(\xi)$ has a $L^2_\xi$-spectral gap centered at $E_\star$ when $\xi$ is $O(\delta)$-away from $\xi_\star$ -- see e.g. \S\ref{sec:5.3}. The no-fold condition requires this gap to extend to a $L^2[\zeta_\star]$-spectral gap of $P_{\pm \delta}[\zeta_\star]$. 

The no-fold condition holds for certain low contrast potentials and the zigzag edge $a_1=1, a_2 = 0$ \cite[Theorem 8.2]{FLTW3}. It holds for high contrast potentials and edges satisfying $a_1 \neq a_2 \mod 3$ \cite[Corollary 6.3]{FLW6}. It may fail in physically relevant cases. See e.g. the case of certain low contrast potentials and the zigzag edge \cite[Theorem 8.4]{FLTW3}; and armchair-type edges $v=a_1v_1+a_2v_2$ where $a_1-a_2 = 0 \mod 3$ \cite[Remark 6.5]{FLW6} or \S\ref{sec:6.1}. In particular, if the no-fold condition holds, \eqref{eq:4k} and $a_1-a_2 \neq 0 \mod 3$ prescribe the possible values of $\zeta_\star$:
\begin{equation}\label{eq:4l}
\zeta_\star \in \big\{ \zeta_\star^A, \zeta_\star^B \big\} = \left\{ \dfrac{2\pi}{3}, \dfrac{4\pi}{3} \right\} \mod 2\pi\Z.
\end{equation}

\subsection{The multiscale approach of \cite{FLTW3} and the Dirac operator.}

Let $(\xi_\star,E_\star)$ be a Dirac point of $P_0$ and $(\phi_1,\phi_2)$ be a Dirac eigenbasis (see Definition \ref{def:1}). The map
\begin{equation}
\eta \in \R^2 \mapsto 2 \big \langle \phi_1, (\eta \cdot D_x) \phi_2 \big \rangle \in \C
\end{equation}
is linear. We look as an application from $\C$ to $\C$. Because of rotational invariance of $P_0 = -\Delta+V$, it acts like multiplication by a complex number:
\begin{equation}
\exists \nu_\star \in \C \setminus \{0\}, \ \ \ \ \forall \eta \in \R^2 \equiv \C, \ \ \  \ \nu_\star \eta = 2 \big \langle \phi_1, (\eta \cdot D_x) \phi_2 \big \rangle.
\end{equation}
See \S\ref{sec:3.2}. Recall that $\var_\star = \lr{\phi_1,W\phi_1}_{L^2_{\xi_\star}} \neq 0$ and that $\kappa$ satisfies \eqref{eq:3v}. In this section, we review the role of the (unbounded) Dirac operator
\begin{equations}\label{eq:5c}
\Di_\star = \matrice{0 & \nu_\star k' \\ \ove{\nu_\star k'} & 0} D_t + \var_\star \matrice{1 & 0 \\ 0 & -1} \kappa \ : \ L^2\big(\R, \C^2\big) \rightarrow L^2\big(\R,\C^2\big)
\end{equations}
 in the analysis of Fefferman--Lee-Thorp--Weinstein \cite{FLTW3}.

When $\var_\star \neq 0$, \cite{FLTW3}
 produces arbitrarily accurate quasimodes of $\PP_\delta[\zeta_\star]$ via a multiscale approach. These are pairs $(u_\delta,E_\delta) \in H^2_{\zeta_\star} \times \R$ satisfying
\begin{equation}
\big( \PP_\delta[\zeta_\star]-E_\delta \big) u_\delta = O_{L^2[\zeta_\star]}\left(\delta^\infty\right), \ \ \ \ E_\delta = E_\star + \delta E_1 + O\left(\delta^2\right).
\end{equation}
They are powers series in $\delta$ whose coefficients solve a hierarchy of equations in order $1, \delta, \delta^2, \dots$. The operator $\Di_\star$ appears in the equation of order $\delta$. This equation admits a solution if and only if $E_1$ is an eigenvalue of $\Di_\star$; see \cite[\S6]{FLTW3}. 

The operator $\Di_\star$ has essential spectrum equal to $(-\infty,\var_F] \cup [\var_F,\infty)$. It has an odd number of eigenvalues $\{  \var_j \}_{j=-N}^N$ in  $(-\var_F, \var_F)$, simple and symmetric about~$0$:
\begin{equation}
\var_{-N} < \dots < \var_{-1} < \var_0 = 0 < \var_1 < \dots < \var_N, \ \ \ \ \var_{-j} = -\var_j.
\end{equation}
In particular, $0$ is always an eigenvalue of $\Di_\star$.
We refer to  see \S\ref{sec:4.2} for details.

When the no-fold condition holds, \cite{FLTW3} uses a sophisticated Lyapounov--Schmidt reduction to prove that each eigenvalue $\var_j$ of $\Di_\star$ seeds a $L^2[\zeta_\star]$-eigenvalue of $\PP_\delta[\zeta_\star]$, with energy $E_\star + \delta \var_j + O(\delta^2)$. They show that to leading order, the corresponding eigenvector equals the first term produced by the multiscale approach: it is
\begin{equation}
\az_1\big( \delta \blr{k',x} \hspace{-.7mm} \big) \cdot \phi_1(x) + \az_2\big( \delta\blr{k',x} \hspace{-.7mm} \big) \cdot  \phi_2(x) + O_{H^2_{\zeta_\star}}\left(\delta^{1/2}\right), \ \ \ \ (\Di_\star - \var_j) \matrice{\az_1 \\ \az_2} = 0.
\end{equation}
In other words, they validate mathematically the formal multiscale procedure at leading order. But some questions persist:
\begin{itemize}
\item Is the multiscale procedure rigorously valid at all order?
\item Do the eigenvalues of $\Di_\star$ seed \textit{all} eigenvalues of $\PP_\delta[\zeta_\star]$ near $E_\star$?
\item How to clarify the relation between $\PP_\delta[\zeta_\star]$ and $\Di_\star$?
\end{itemize}
The present work responds to these questions. 

\subsection{Results}\label{sec:1.5}

Our  first result relates the resolvents of $\PP_\delta[\zeta_\star]$ and $\Di_\star$. It  requires the operator $\Pi$ and its adjoint $\Pi^*$, defined as
\begin{equations}
\Pi : L^2\big(\R^2/\Z v, \C^2\big) \rightarrow L^2\big(\R,\C^2\big), \ \ \ \ \big(\Pi f\big)(t) \de \int_0^1 f(sv + tv') ds; 
\\
\Pi^* : L^2\big(\R, \C^2\big) \rightarrow L^2\big(\R^2/\Z v,\C^2\big), \ \ \ \ \big(\Pi^* g\big)(x) \de g\big(\hspace*{-.5mm}\blr{k',x}\hspace*{-.5mm} \big);
\end{equations}
and the dilation $\UU_\delta$ defined as
\begin{equation}
\UU_\delta : L^2\big(\R,\C^2\big) \rightarrow L^2\big(\R,\C^2\big), \ \ \ \ \big(\UU_\delta f\big)(t) \de f(\delta t).
\end{equation}

Recall that $V$ is a honeycomb potential -- see Definition \ref{def:3}; $W \in C^\infty\big(\R^2,\R\big)$ breaks spatial inversion -- see \eqref{eq:5d}; and $\kappa \in C^\infty(\R,\R)$ is a domain wall function -- see \eqref{eq:3v}. We make the following assumptions:
\begin{itemize}
\item[(H1)] $(\xi_\star,E_\star)$ is a Dirac point of $P_0 = -\Delta + V$ -- see Definition \ref{def:1} -- with $\xi_\star \in \Ll^*$.
\item[(H2)] The no-fold condition -- Definition \ref{def:2} -- holds.
\item[(H3)] The non-degeneracy assumption $\var_\star \neq 0$ holds -- see \eqref{eq:3h}.
\end{itemize}

\begin{theorem}\label{thm:3} Assume $\operatorname{(H1)}$ -- $\operatorname{(H3)}$ hold and fix $\epsilon > 0$. There exists $\delta_0 > 0$ such that~if
\begin{equation}\label{eq:3f}
\delta \in (0,\delta_0), \ \ z \in \Dd(0,\var_F-\epsilon), \ \
\dist \left( \Sigma_{L^2}\big(\Di_\star \big), z \right) \geq \epsilon, \ \  \lambda = E_\star + \delta z
\end{equation}
then $\PP_\delta[\zeta_\star] - \lambda$ is invertible and
\begin{equation}\label{eq:3g}
\big( \PP_\delta[\zeta_\star] - \lambda \big)^{-1} = \dfrac{1}{\delta} \cdot  \matrice{ \phi_1 \\ \phi_2}^\top \cdot \Pi^*   \UU_\delta \cdot \big( \Di_\star - z \big)^{-1} \cdot \UU_\delta^{-1} \Pi  \cdot \ove{\matrice{\phi_1 \\ \phi_2}} + \OO_{L^2[\zeta_\star]}\left(\delta^{-1/3}\right).
\end{equation}
\end{theorem}

The leading order term in \eqref{eq:3g} comes with a coefficient $1/\delta$: the remainder term $\OO_{L^2[\zeta_\star]}(\delta^{-1/3})$ is subleading when $z \in \Dd(0,\var_F-\epsilon)$. Hence, Theorem \ref{thm:3} shows that the resolvents of $\PP_\delta[\zeta_\star]$ and of $\Di_\star$ behave similarly, after suitable conjugations.

 Theorem \ref{thm:3} applies to a spectral range that spans -- modulo $\epsilon$ -- the entire spectral gap of $\PP_\delta[\zeta_\star]$ about $E_\star$. The next result describes the spectrum of $\PP_\delta[\zeta_\star]$ in the essential spectral gap in terms of the eigenvalues
\begin{equation}
\var_{-N} < \dots < \var_{-1} < \var_0 = 0 < \var_1 < \dots < \var_N
\end{equation}
of the Dirac operator $\Di_\star$. Let $X$ be the function space equal to
\begin{equation}\label{eq:5h}
\Big \{  f \in C^\infty\big(\R^2 \times \R,\C\big)  :  \ \forall t \in \R, \ f(\cdot,t) \in L^2_{\xi_\star}; \ \exists a > 0, \ \sup e^{a|t|} |f(x,t)| < \infty  \Big\}.
\end{equation}

\vspace*{-.5cm}

\begin{center}
\begin{figure}
\caption{Eigenvalues of $\Di_\star$ in $(-\var_F,\var_F)$ (top) and eigenvalues of $\PP_\delta$  in the spectral gap containing $E_\star$ (bottom). An approximate rescaling equal to $z \mapsto E_\star+\delta z + O(\delta^2)$ maps the top to the bottom. The red dots represent the zero eigenvalue of $\Di_\star$ and the corresponding one for $\PP_\delta$.  Theorem \ref{thm:3} and Corollary \ref{cor:3} do not apply in the lighter gray area near the essential spectrum. 
}\label{fig:3}
{\includegraphics{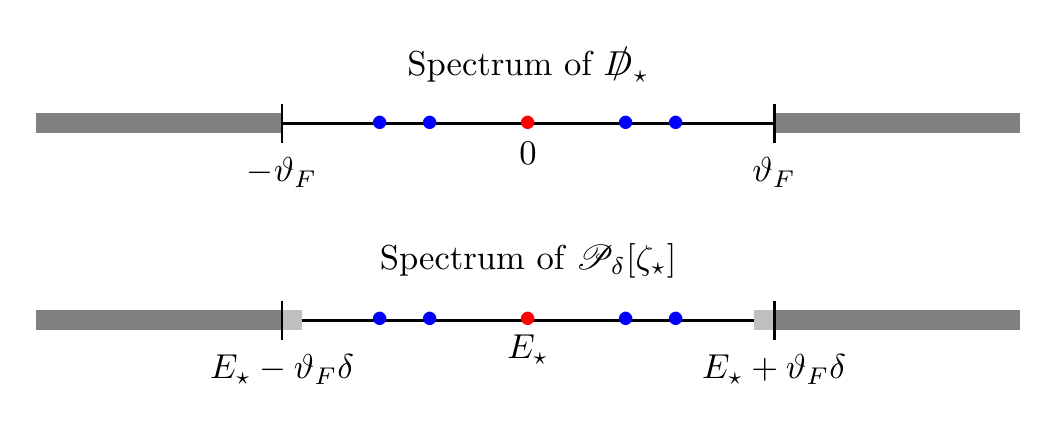}}
\end{figure}
\end{center}

\begin{cor}\label{cor:3} Assume $\operatorname{(H1)}$ -- $\operatorname{(H3)}$ hold and fix $\tvar \in (\var_N,\var_F)$. There exists $\delta_0 > 0$ such that for $\delta \in (0,\delta_0)$, the operator $\PP_\delta[\zeta_\star]$ has exactly $2N+1$ eigenvalues $\{E_{\delta,j}\}_{j \in [-N,N]}$ in $[E_\star-\tvar \delta,E_\star + \tvar\delta]$, that are all simple. 

 The associated eigenpairs $(E_{\delta,j},u_{\delta,j})$ admit full two-scale expansions in powers of $\delta$:
\begin{equations}\label{eq:5e}
E_{\delta,j} = E_\star + \var_j \cdot \delta + a_2 \cdot \delta^2 + \dots + a_M \cdot \delta^M + O\left( \delta^{M+1} \right), 
\\
u_{\delta,j}(x) = f_0\big(x,\delta \blr{k',x} \hspace*{-.5mm}\big) + \delta \cdot f_1\big(x, \delta \blr{k',x}\hspace*{-.5mm}\big) + \dots + \delta^M \cdot f_M\big(x,\delta \blr{k',x}\hspace*{-.5mm}\big) + o_{H^k}\left( \delta^M \right).
\end{equations}
In the above:
\begin{itemize}
\item $M$ and $k$ are any integers; $H^k$ is the $k$-th order Sobolev space.
\item The terms $a_m \in \R$, $f_m \in X$ are recursively constructed via multiscale analysis.
\item The leading order term $f_0$ satisfies
\begin{equation}
f_0(x,t) = \az_1(t)\cdot  \phi_1(x) + \az_2(t) \cdot \phi_2(x), \ \ \ \ \left(\Di_\star - \var_j\right) \matrice{\az_1 \\ \az_2} = 0.
\end{equation}
\end{itemize}
\end{cor}

This corollary: (a) mathematically validates the multiscale procedure of \cite{FLTW3} at all order in $\delta$; (b) shows that all eigenvectors of $\PP_\delta[\zeta_\star]$ are induced by the modes of $\Di_\star$. See Figures \ref{fig:3} and \ref{fig:8}. In particular, (a) improves the result of Fefferman--Lee-Thorp--Weinstein \cite{FLTW3} to arbitrary order in $\delta$. From a general point of view, (b) represents the most important advance. It opens the way for mathematical proofs of the bulk-edge correspondence in \textit{continuous} honeycomb structures. See \S\ref{sec:1.7}.

\vspace{-.8cm}

\begin{center}
\begin{figure}
\caption{Discrete eigenvalues of $\Di_\star$ seed the bifurcation of eigenvalues of $\PP_\delta$ (red dotted curves) from the Dirac point energy $E_\star$ (at $\delta=0$)  of $P_0$  as $\delta$ increases away from zero. The slopes of these curves at $\delta = 0$ (blue lines) are given by the eigenvalues of $\Di_\star$. }\label{fig:8}
{\includegraphics{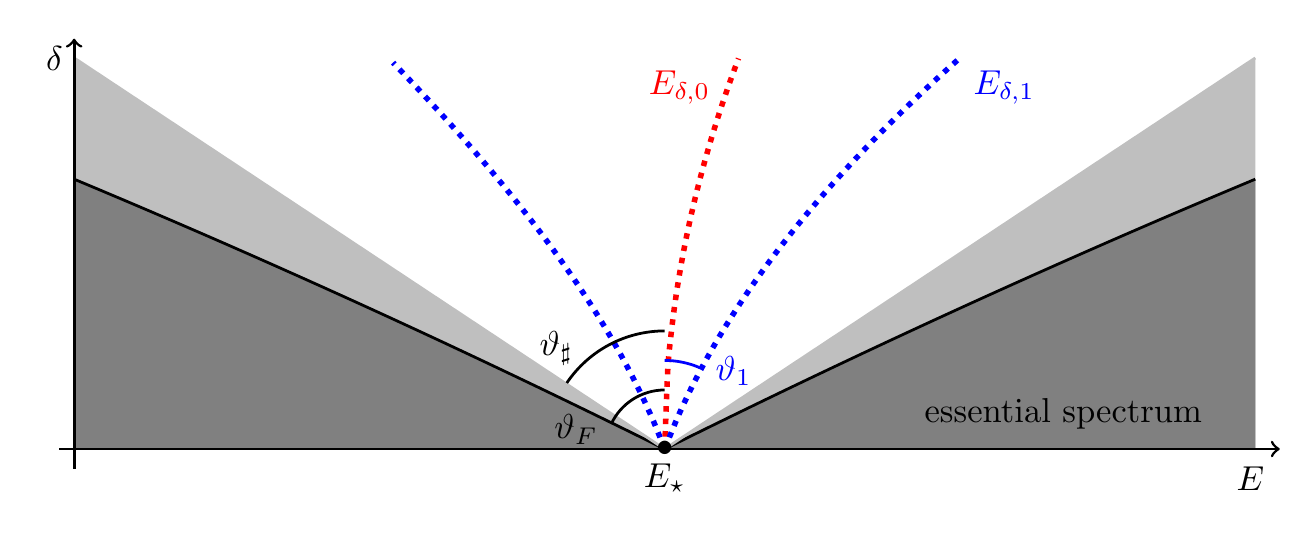}}
\end{figure}
\end{center}

\subsection{Extension to quasimomenta near $\zeta_\star$} Corollary \ref{cor:3} predicts that for $\delta \in (0,\delta_0)$, $\PP_\delta[\zeta_\star]$ has precisely $2N+1$ eigenvalues near $E_\star$. A general perturbation argument shows that $\PP_\delta[\zeta]$ also has $2N+1$ eigenvalues for $\zeta$ close enough to $\zeta_\star$. However this argument does not specify quantitatively how close $\zeta$ needs to be to $\zeta_\star$.

We prove generalizations of Theorem \ref{thm:3} and Corollary \ref{cor:3} that hold for $\zeta$ at distance $O(\delta)$ from $\zeta_\star$; see \S\ref{sec:4.3} for statements. We show that the eigenvalues of $\PP_\delta[\zeta_\star+\mu\delta]$ lying near $E_\star$ and of the Dirac operator
\begin{equation}
\Di(\mu) \de \matrice{0 & \nu_\star k' \\ \ove{\nu_\star k'} & 0} D_t + \mu \matrice{0 & \nu_\star \ell \\ \ove{\nu_\star \ell} & 0} + \var_\star \matrice{1 & 0 \\ 0 & -1} \kappa, \ \ \ \ \ell \de  k - \dfrac{\blr{k,k'}}{|k'|^2} k'
\end{equation} 
are $O(\delta^2)$-away after the rescaling $z \mapsto E_\star + \delta z$. 
  
Interestingly enough, the spectrum of $\Di(\mu)$ can be derived from that of $\Di_\star = \Di(0)$; see \S\ref{sec:4.2} and Figure \ref{fig:7}. We observe that $\Di(\mu)$ has a topologically protected mode that bifurcates linearly from the zero mode of $\Di_\star$. This suggests that under the $\PP_\delta$ time-dependent evolution, $L^2$-wave packets formed from the topologically protected mode of $\Di(\mu)$ propagate dispersion-less along the edge for a very long time. 

All other modes of $\Di(\mu)$ are non-topologically protected and bifurcate quadratically from the modes of $\Di_\star$. $L^2$-wave packets formed from such modes should have a shorter lifetime. This suggests that topologically protected modes are more robust even in the time-dependent situation.

\vspace*{-.5cm}

\begin{center}
\begin{figure}
\caption{The spectrum of $\Di(\mu)$ as a function of $\mu$. The topologically protected eigenvalue (in red) bifurcate linearly while the non-topologically protected eigenvalues (in blue) bifurcate quadratically.} \label{fig:7}
{\includegraphics{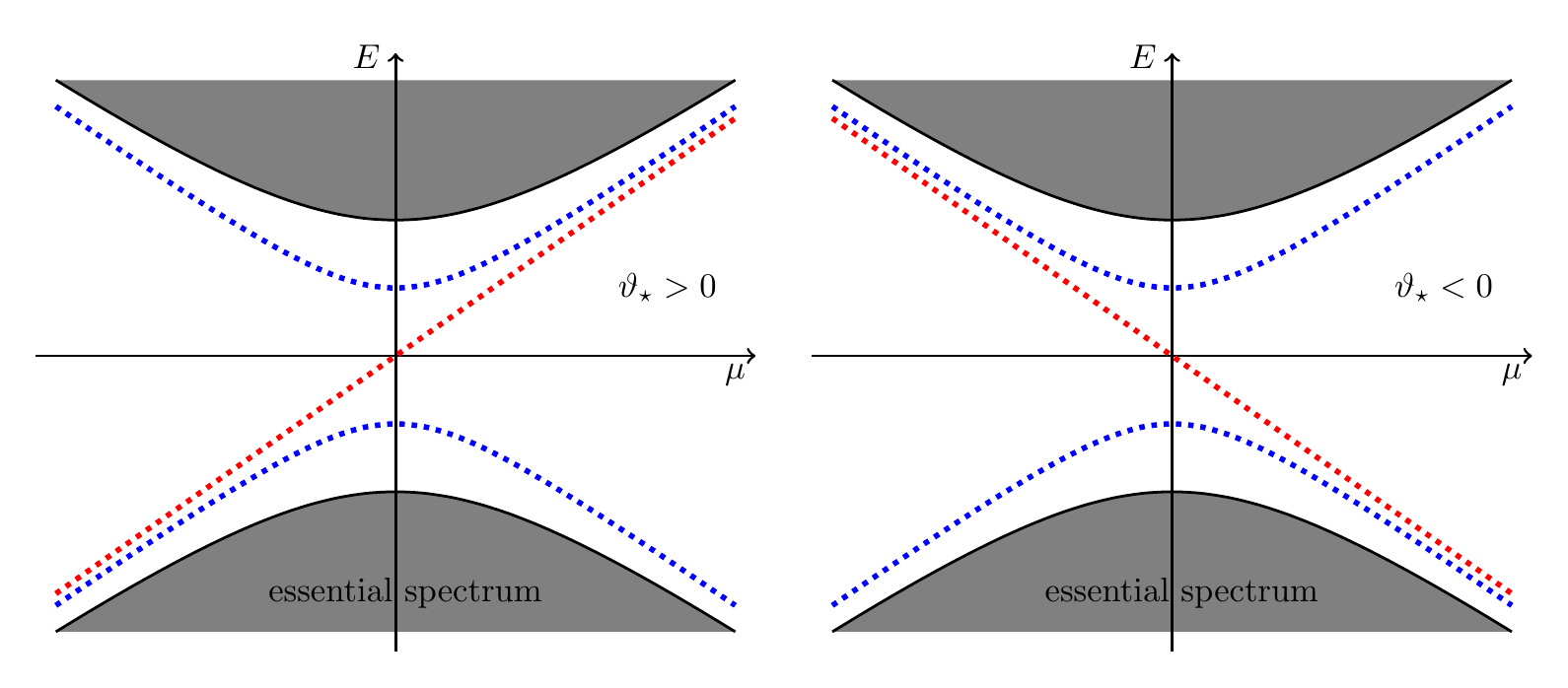}}
\end{figure}
\end{center}

\subsection{A topological perspective}\label{sec:1.7} Recall that $k' \in \Lambda^*$ is the dual direction transverse to an edge $\R v$; and that $\lambda_{0,j}(\xi)$ are the dispersion surfaces of a  honeycomb Schr\"odinger operator $P_0$. Let $(\xi_\star,E_\star) = (\xi_\star,\lambda_{0,j_\star}(\xi_\star))$ denote a Dirac point of $P_0$. We introduce an assumption (H4) that extends (H3) to values $\zeta \neq \zeta_\star$. It asks for the $j_\star$-th $L^2[\zeta]$-gap of $P_0[\zeta]$ to be open when $\zeta \notin \{ 2\pi/3, \ 4\pi/3  \} \mod 2\pi\Z$.
\begin{itemize}
\item[(H4)] For every $\zeta \notin \{ 2\pi/3, \ 4\pi/3  \} \mod 2\pi \Z$, for every $\tau, \tau' \in \R$,
\begin{equation}
\lambda_{0,j_\star}(\zeta k + \tau k') < \lambda_{0,j_\star+1}(\zeta k + \tau' k').
\end{equation}
\end{itemize}
Assumption (H4) holds for non armchair-type edges ($a_1 \neq a_2 \mod 3$) and high-contrast potentials: see \cite[Theorem 6.1 and Remark 6.5]{FLW6}. This follows from two general phenomena:
\begin{itemize}
\item Schr\"odinger operators with multiple-well potentials approach their tight binding limits as the depth of the wells increases \cite{Ha,HS1,Si2,HS2,Ma1,Ou,Ma2,Ca,FLW6,FW3};
\item Wallace's tight binding model of honeycomb lattices \cite{Wal} satisfies a suitable version of (H4).
\end{itemize}

When (H1) -- (H4) hold and $\delta$ is sufficiently small, the $j_\star$-th $L^2[\zeta]$-gap of $P_\delta[\zeta]$ is open. This allows to define the spectral flow of the family
\begin{equation}\label{eq:4m}
\zeta \in [0,2\pi] \mapsto \PP_\delta[\zeta]
\end{equation}
in the $j_\star$-th $L^2[\zeta]$-gap. It is the signed number of $L^2[\zeta]$-eigenvalues of $\PP_\delta[\zeta]$ crossing the $j_\star$-th gap downwards as $\zeta$ runs from $0$ to $2\pi$; see e.g. \cite[\S4]{Wa}. Corollary \ref{cor:2} in \S\ref{sec:4.3} allows to count precisely these eigenvalues. It leads to:

\begin{cor}\label{cor:4} Assume that $\operatorname{(H1)}$ -- $\operatorname{(H4)}$ hold for both Dirac points $(\xi_\star^A,E_\star)$ and $(\xi_\star^B,E_\star)$. There exists $\delta_0 > 0$ such that for all $\delta \in (0,\delta_0)$, the spectral flow of $\PP_\delta$ in the $j_\star$-th $L^2[\zeta]$-gap vanishes.
\end{cor}

This is because $\var_\star^A$ and $\var_\star^B$ are opposite -- where  $\var_\star^J$ corresponds to $\var_\star$ for the Dirac point $(\xi_\star^J,E_\star)$. See Figure \ref{fig:5}.  The spectral flow is a topological invariant: it does not change if a $2\pi$-periodic family of compact operators $H^2[\zeta] \rightarrow L^2[\zeta]$ is added to $\PP_\delta[\zeta]$. Hence Corollary \ref{cor:4} is very robust. However, it is a disappointing result: it suggests that the edge states of Corollary \ref{cor:3} shall not be topologically stable. We conjecture:

\vspace*{-.65cm}

\begin{center}
\begin{figure}
\caption{The spectrum of $\PP_\delta[\zeta]$ as a function of $\zeta$. The dark gray region represents the essential spectrum. The dotted curves are the eigenvalues of $\PP_\delta[\zeta]$ (the edge state energies). Zooming about $\delta^{-1}$ times near $(2\pi/3,E_\star)$ or $(4\pi/3,E_\star)$ produces Figure \ref{fig:7}. Because of complex conjugation, $\var_\star^A = -\var_\star^B$: near $2\pi/3$ (resp. $4\pi/3$), the red curves moves upwards (resp. downwards).  This results in a spectral flow cancellation.} \label{fig:5}
{\includegraphics{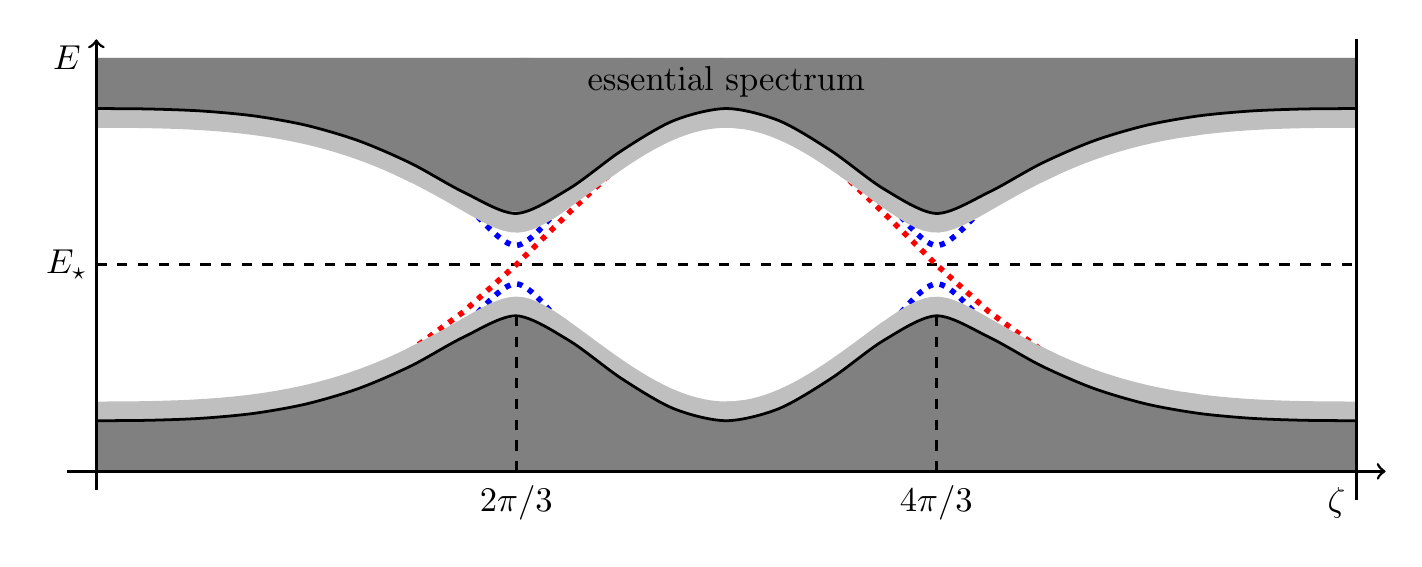}}
\end{figure}
\end{center}

\begin{conj} Assume that $\operatorname{(H1)}$ -- $\operatorname{(H4)}$ hold for both Dirac points $(\xi_\star^A,E_\star)$ and $(\xi_\star^B,E_\star)$. There exists $\delta_0 > 0$ such that for every $\delta \in (0,\delta_0)$, there exists a family $\zeta \in \R \mapsto B_\delta(\zeta)$ such that:
\begin{itemize}
\item $B_\delta(\zeta)$ is a compact operator $H^2[\zeta] \rightarrow L^2[\zeta]$;
\item $B_\delta(\zeta)$ depends continuously on $\zeta$ (with respect to the operator norm on $H^2[\zeta] \rightarrow L^2[\zeta]$) and $B_\delta(\zeta+2\pi) = B_\delta(\zeta)$ for every $\zeta \in \R$;
\item $\PP_\delta[\zeta] + B_\delta(\zeta) : H^2[\zeta] \rightarrow L^2[\zeta]$ has no eigenvalues in the essential spectral gap containing $E_\star$.
\end{itemize}
\end{conj}

On a positive note, our approach also applies  to ``magnetic" Schr\"odinger operators
\begin{equations}\label{eq:4j}
\tPP_\delta = P_0 + \delta \cdot \kappa_\delta \cdot \tW, \ \ \ \ \tW = \Aa \cdot D_x + D_x \cdot \Aa, 
\\
\Aa \in C^\infty\big(\R^2,\R^2\big); \ \ \ \ \Aa(x+w) = \Aa(x), \ w \in \Lambda; \ \ \ \ \Aa(-x) = -\Aa(x).
\end{equations} 
The perturbation $\tW$ no longer breaks spatial inversion; instead it breaks time-reversal symmetry (complex conjugation). See \cite{RH,HR,LWZ} for related models. We replace (H3) with 
\begin{enumerate}
\item[(H3')] The non-degeneracy condition $\te_\star \de \big\langle \phi_1,\tW \phi_1 \big \rangle_{L^2_{\xi_\star}} \neq 0$ holds.
\end{enumerate}
When (H1), (H2) and (H3') hold, the operator $\Pp_\delta[\zeta_\star]$ has an essential spectral gap centered at $E_\star$ -- similarly to $\PP_\delta[\zeta_\star]$. If moreover (H4) holds, then we can define the spectral flow of the family $\zeta \mapsto \Pp_\delta[\zeta]$.

\begin{center}
\begin{figure}
\caption{The spectrum of the magnetic-like perturbation $\tPP_\delta$ of $P_0$ for positive $\te_\star$. The topologically protected mode of the Dirac operator induces precisely two edge states energy curves. In contrast with Figure \ref{fig:5}, $\te_\star^A = \te_\star^B$: both red curves moves upwards. The resulting spectral flow is $-2$, indicating topologically protected states.} \label{fig:6}
{\includegraphics{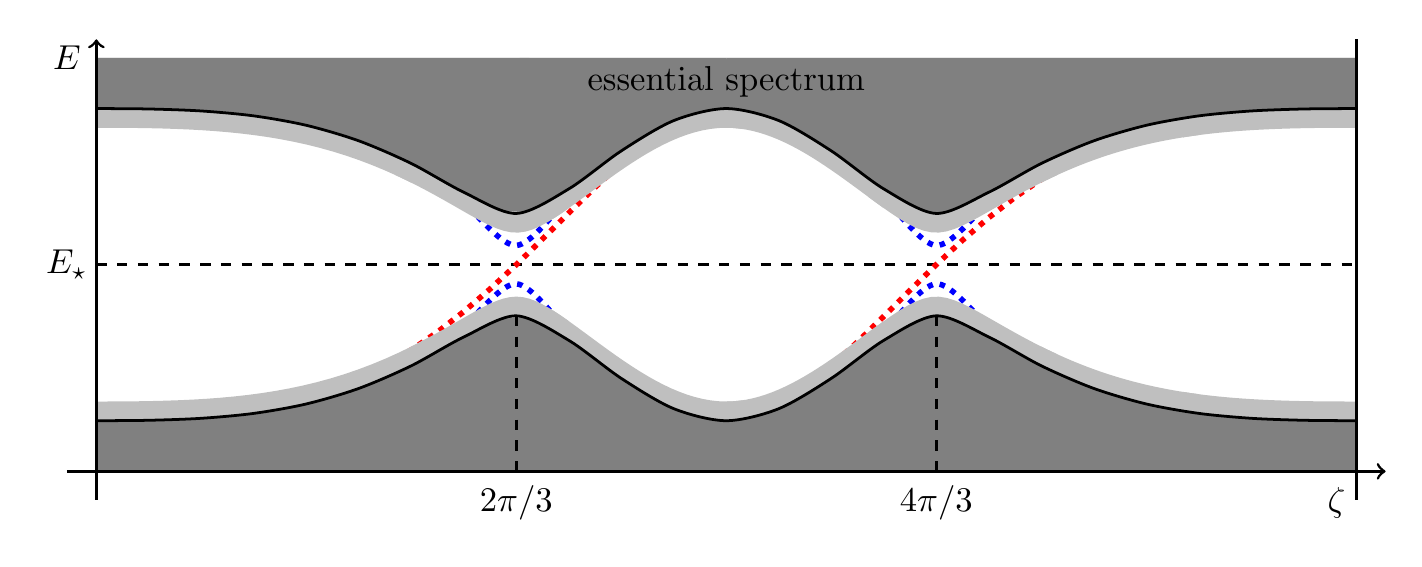}}
\end{figure}
\end{center}

\vspace*{-.8cm}

\begin{cor}\label{cor:5} Assume that $\operatorname{(H1)}$, $\operatorname{(H2)}$, $\operatorname{(H3')}$ and $\operatorname{(H4)}$ hold for both Dirac points $(\xi_\star^A,E_\star)$ and $(\xi_\star^B,E_\star)$. There exists $\delta_0 > 0$ such that for all $\delta \in (0,\delta_0)$, the spectral flow of $\tPP_\delta$ equals $-2 \cdot \sgn(\te_\star)$.
\end{cor}

Corollary \ref{cor:5} shows that $\tPP_\delta$ admits two topologically protected edge states. This corroborates results of Haldane--Raghu \cite{HR,RH}, where two quasimodes are produced via a multiscale approach. They were not proved to be topologically protected there: a statement in the spirit of Corollary \ref{cor:2} is missing. The authors perform a formal computation of the bulk index: they show that it should equal $2$ or $-2$. We will study rigorously the bulk aspects of our problem in a future work.

\subsection{Strategy} Our proof has three essential components.
\begin{itemize}
\item The simplest step consists in deriving Corollary \ref{cor:3} from Theorem \ref{thm:3}, see \S\ref{sec:4.3}. Theorem \ref{thm:3} is used to count the exact number $2N+1$ of eigenvalues in the essential spectral gap (slightly away from the edges). We derive the full expansion of edge states in powers of $\delta$ using (a) the formal multiscale procedure of \cite{FLTW3} to produce $2N+1$, almost orthogonal, arbitrarily accurate quasimodes; (b) a general selfadjoint principle that implies that these quasimodes must all be near genuine eigenvectors. 
\item We derive resolvent estimates for the bulk operators $P_{\pm \delta}[\zeta_\star]$. We first obtain resolvent estimates for the operators $P_{\pm \delta}(\xi) : H^2_\xi \rightarrow L^2_\xi$ in \S\ref{sec:5}. We prove that near $(\xi_\star,E_\star)$, these operators essentially behave like Pauli matrices. In \S\ref{sec:6} we integrate these estimates along the dual edge $\zeta_\star k + \R k'$ and derive the expansion
 \begin{equations}
 \big(P_{\pm \delta}[\zeta_\star]- \lambda \big)^{-1} = \dfrac{1}{\delta} \cdot \matrice{\phi_1 \\ \phi_2}^\top \Pi^* \cdot  \UU_\delta \left( \Di_{\star,\pm} - z \right)^{-1} \UU_\delta^{-1} \cdot \Pi \ove{\matrice{\phi_1 \\ \phi_2}} + \OO_{L^2[\zeta]}\left(\delta^{-1/3}\right).
 \end{equations}
Above,  $\Pi$ and $\UU_\delta$ are the operators introduced in \S\ref{sec:1.5}; and $\Di_{\star,\pm}$ are the formal limits of $\Di_\star$ as $t$ goes to $\pm \infty$.
\item We use a sophisticated version of the Lippmann--Schwinger principle to connect the resolvents of $\PPP_\delta[\zeta_\star]$ and of $P_{\pm\delta}[\zeta_\star]$. This requires to construct a parametrix for $\PPP_\delta[\zeta_\star]$. After algebraic manipulations -- essentially cyclicity arguments -- homogenization effects take place and produce the operator $\Di_\star$. This leads to the resolvent estimate of Theorem \ref{thm:3}.
\end{itemize}

\subsection{Relation to earlier work} The mechanism responsible for the production of edge states is the bifurcation of eigenvalues from the edge of the continuous spectrum. Such problems have a long history: see e.g. \cite{Ta,Sh,Si,DH,FK,Bor,BG1,Bo7,PLAJ,Bo8,HW,Ze} for states generated by defects in periodic backgrounds; and \cite{GW,BorG,DW,DVW,Di,DD,Dr2,Dr4,Dr3,DR} for localized highly oscillatory perturbations.

Fefferman, Lee-Thorp and Weinstein \cite{FLTW3,FLTW4} produced the closest results to our analysis. They were the first to prove existence of edge states for continuous honeycomb lattices, in the small/adiabatic regime $\delta \rightarrow 0$. They built up on their own work \cite{FLTW1,FLTW2} where they proved existence of defect states for dislocated one-dimensional materials. 

Our work improves and extends \cite{FLTW3,FLTW4} in the following way:
\begin{itemize}
\item It connects the resolvents of  $\PP_\delta[\zeta]$ and $\Di(\mu)$.
\item It provides full expansions of edge states in powers of $\delta$.
\item It identifies \textit{all} edge states with energy near Dirac point energies.
\end{itemize}
The third point allows to interpret topologically the results in terms of the spectral flow of $\zeta \mapsto \PPP_\delta[\zeta]$. This is a robust invariant of the system, also called edge index. We conjecture that the modes of $\PP_\delta[\zeta]$ should not be topologically protected: the edge index vanishes. However, for the magnetic operator $\tPP_\delta[\zeta]$ introduced in \eqref{eq:4j}, two such states are topologically protected: they persist under large (suitable) deformations. 

We refer to \cite{GMS,PST,WLW,WW} for the study of similar operators with perturbations that vary adiabatically in all directions; and to \cite{DL1,DL2,CHR1,CHR2,CHR3} for analysis of perturbations small with respect to the inverse scale of variation. The scaling studied here is peculiar: the perturbation varies adiabatically in one direction only. 

Our strategy generalizes the one-dimensional work \cite{DFW}, developed to improve the results of \cite{FLTW1,FLTW2}. The contruction of genuine edge states from quasimodes in \S\ref{sec:4.3} follows the same classical procedure as \cite[\S3.3]{DFW} and is sketched here. We derive the fiberwise resolvent estimates for $P_{\pm \delta}(\xi)$ in \S\ref{sec:12.1}-\ref{sec:5.3} as in \cite[\S4.1-4.2]{DFW}. We continued the work \cite{DFW} in \cite{Dr}. There, we showed that the defect states of \cite{FLTW1,FLTW2} are topologically stable in the following sense. The model embeds naturally in a one-parameter family of \textit{dislocated systems}, related to  \cite{Ol,Ko,HK,DPR,HK2,HK3}. We compute the spectral flow in terms of bulk quantities. We show that it is equal to the bulk index -- the Chern number of a Bloch eigenbundle for the bulk. Hence, \cite{Dr} provides a novel continuous setting where the bulk-edge correspondence holds -- adding to \cite{KS,KS2,Taa,FSF,Ba,Ba2,BR}.

\subsection{Further perspectives} Our results stimulate future lines of research:
\begin{itemize}
\item Armchair-type edges are edges so that the associated dual line $\zeta_\star k + \R k'$ passes through both Dirac momenta $\xi_\star^A$ and $\xi_\star^B$. They correspond to  directions
\begin{equation}
v=a_1v_1+a_2 v_2, \ \ \ \ v_1 \wedge v_2 =1, \ \ \ \ a_1 = a_2 \mod 3,
\end{equation}
see \S\ref{sec:6.1}. The no-fold condition barely fails for such edges: $\PPP_\delta[\zeta_\star]$ still has an essential gap in, say, the sharp contrast regime. See \cite[Corollary 6.3]{FLW6}. We expect our techniques to be robust enough to handle such edges. In particular, a $2 \times 2$ bloc of uncoupled Dirac operators should emerge in the resolvent estimates. 
\item This work may open the way to prove the no-fold conjecture of Fefferman--Lee-Thorp--Weinstein \cite{FLTW3}. It predicts that long-lived \textit{resonant} edge states should appear when the no-fold condition fails. This is supported by the existence of highly accurate localized quasimodes, still produced by the formal multiscale procedure of \cite{FLTW3}.  See  \cite{GeS,SV,TZ,St1,St2,Ga1} for the relation between quasimodes and resonances in other settings.
\item The eigenvalue curve $\zeta \mapsto E_{\delta,0}^\zeta$ of $\PP_\delta[\zeta]$ corresponding to the topologically protected mode of $\Di(\mu)$ intersects $E_\star$ transversely. See the red curves in Figure \ref{fig:5}. This contrasts with the eigenvalue curves $\zeta \mapsto E_{\delta,j}^\zeta$, $j \neq 0$, which exhibit quadratic extrema near $\zeta_\star$; see the blue curves in Figure \ref{fig:5}. This indicates that $L^2$-wave packets constructed from the topologically protected modes of $\Di(\mu)$ should have a longer lifetime. Mathematical and experimental investigations of this phenomenon would be interesting. The techniques could lead to a time-dependent analysis of quasimodes when the no-fold condition fails. See \cite{GeS} for a related investigation in the shape resonance context; and \cite{CMS,AZ1,AZ2,FW2,AS} for related investigations in gap-less settings.
\item In a forthcoming work, we will investigate the relation between the bulk and edge indexes of $\PPP_\delta[\zeta]$ or $\tPP_\delta[\zeta]$, as in \cite{HR}. The bulk-edge correspondence predicts that these should be equal. It is widely unexplored in continuous, asymptotically periodic settings: apart from \cite{BR,Dr}, the only investigations concern the quantum Hall effect \cite{KS,KS2,Taa}. The dicrete setting is better understood \cite{KRS,EGS,GP,ASV,Ba,Sha,Br,GS,GT,ST}. It would also be nice to study it in quantum graph models of graphene -- see \cite{KP,BZ,BHJ,Le} for setting and spectral results.
\item The recent numerical approach \cite{TWL} could be applied to $\tPP_\delta$ as $\delta$ increases away from $0$. Corollary \ref{cor:5} shows implies that two edge states persist as long as the gap remains open. However their qualitative description (Corollary \ref{cor:6}) should progressively break down as $\delta$ increases. It would be interesting to investigate numerically how their shape changes.
\end{itemize}

\subsection*{Notations} Here is a list of notations used in this work:
\begin{itemize}
\item If $z \in \C$, $\bar{z}$ denotes its complex conjugate and $|z|$ its modulus. We will sometimes identify a vector $x = [x_1,x_2]^\top \in \R^2$ with the complex number $x_1 + i x_2$.
\item $\Ss^1 \subset \C$ is the circle $\{z \in \C : |z|=1\}$; 
\item $\Dd(z,r) \subset \C$ denotes the disk centered at $z \in \C$, of radius $r$.  
\item If $E,F \subset \C$, $\dist(E,F)$ denotes the Euclidean distance between $E$ and $F$.
\item $D_x$ is the operator $\frac{1}{i}[\p_{x_1},\p_{x_2}]^\top = \frac{1}{i} \nabla$.
\item $L^2$ denotes the space of square-summable functions  and $H^s$
are the classical Sobolev spaces. 
\item If $\HH$ and $\HH'$ are Hilbert space and $\psi \in \HH$, we write $|\psi|_\HH$ for the norm of $\HH$; if $A : \HH \rightarrow \HH'$ is a bounded operator, the operator norm of $A$ is
\begin{equation}
\| A \|_{\HH \rightarrow \HH'} \de \sup_{|\psi|_\HH=1} |A\psi|_{\HH'}.
\end{equation}
If $\HH = \HH'$, we simply write $\| A \|_{\HH} = \| A \|_{\HH \rightarrow \HH}$.
\item If $\psi_\epsi \in \HH$ -- resp. $A_\epsi : \HH \rightarrow \HH$ is a linear operator --  and $f : \R \setminus \{0\} \rightarrow \R$, we write $\psi_\epsi = O_{\HH}\big(f(\epsi)\big)$ -- resp. $A_\epsi = \OO_{\HH \rightarrow \HH'}\big(f(\epsi)\big)$ -- when there exists $C > 0$ such that $|\psi_\epsi|_\HH \leq Cf(\epsi)$  -- resp. $\|A_\epsi\|_{\HH \rightarrow \HH'} \leq C f(\epsi)$ -- for $\epsi \in (0,1]$. If $\HH = \HH'$, we simply write $A_\epsi = \OO_\HH\big(f(\epsi)\big)$.
\item We denote the spectrum of a (possibly unbounded) operator $A$ on $\HH$ by $\Sigma_\HH(A)$. It splits into an essential part $\Sigma_{\HH,\ess}(A)$ and a discrete part $\Sigma_{\HH,\pp}(A)$.
\item $\Lambda$ is the lattice $\Z v_1\oplus \Z v_2$ -- see \S\ref{sec:1.1}. An edge is a line $\R v \subset \R^2$ with $v = a_1v_1+a_2v_2\in \Lambda$, $a_1, a_2$ relatively prime integers. We associate to $v$ vectors $v', k$ and $k'$ via \eqref{eq:3w}.
\item The space $L^2_\xi$ consists of $\xi$-quasiperiodic functions with respect to $\Lambda$:
\begin{equation}
L^2_\xi \de \left\{ u\in L^2_\loc \big(\R^2,\C\big) : \ u(x+w) = e^{i\blr{\xi,w}} u(x), w \in \Lambda \right\}.
\end{equation}
\item $\ell \in (\R^2)^*$ is the projection of $k$ orthogonally to $k'$:
\begin{equation}
\ell \de k - \dfrac{\blr{k,k'}}{|k'|^2} k'.
\end{equation} 
\item $L^2[\zeta]$ is the space
\begin{equation}
L^2[\zeta] \de \left\{ u \in L^2_\loc(\R^2,\C), \ u(x+v) = e^{i\zeta} u(x), \ \int_{\R^2/\Z v} |u(x)|^2 dx   < \infty\right\}.
\end{equation}
\item $V \in C^\infty(\R^2,\R)$ is a honeycomb potential -- see Definition \ref{def:3}.
\item $W \in C^\infty(\R^2,\R)$ is $\Lambda$-periodic and odd, see \eqref{eq:5d}.
\item $P_\delta $ is the operator $-\Delta + V + \delta W$ on $L^2$; for $\xi \in \R^2$, $P_\delta(\xi)$ is the operator formally equal to $P_\delta$ but acting on $L^2_\xi$. For $\zeta \in \R$, $P_\delta[\zeta]$ is the operator formally equal to $P_\delta$ but acting on $L^2[\zeta]$.
\item $\PP_\delta$ is the operator $-\Delta + V +\delta\cdot \kappa_\delta \cdot W$ on $L^2$, where $\kappa_\delta(x) = \kappa(\delta \blr{k',x})$ and $\kappa$ is a domain-wall function -- see \eqref{eq:3v}. $\PP_\delta[\zeta]$ is the operator formally equal to $\PP_\delta$ but acting on $L^2[\zeta]$.
\item $(\xi_\star,E_\star)$ denotes a Dirac point of $P_0 = -\Delta + V$, associated to a Dirac eigenbasis $(\phi_1,\phi_2)$ -- see Definition \ref{def:1}.
\item $\zeta_\star$ is the real number $\blr{\xi_\star,v}$.
\item $\xi_\star^A, \xi_\star^B, \zeta_\star^A, \zeta_\star^B$ are defined in \eqref{eq:3y} and \eqref{eq:4k}, respectively.
\item $\nu_\star$ is a complex number associated to $(\xi_\star,E_\star)$ and to the Dirac eigenbasis $(\phi_1,\phi_2)$, such that $|\nu_\star| = \nu_F$ -- see \S\ref{sec:3.2}.
\item $\var_\star = \blr{\phi_1, W\phi_1}_{L^2_{\xi_\star}}$ is always assumed to be non-zero; we also define $|\var_\star| = \var_F$.
\item The Pauli matrices are
\begin{equation}\label{eq:1f}
\bsu = \matrice{0 & 1 \\ 1 & 0}, \ \ \bsd = \matrice{0 & -i \\ i & 0}, \ \  \bst = \matrice{1 & 0 \\ 0 & -1}.
\end{equation}
These matrices satisfy $\bsj^2 = \Id$ and $\bsi \bsj = -\bsj \bsi$ for $i \neq j$.
\end{itemize}

\subsection*{Acknowledgements} I would like to thank Michael Weinstein for his in-depth introduction to the subject and for suggesting the project. Support from the Simons Foundation through M. Weinstein's Math+X investigator award $\#$376319 and from NSF DMS-1800086 are gratefully acknowledged.


\section{Honeycomb potentials, Dirac points and edges}

\subsection{Equilateral lattice} We review briefly the definitions of \S\ref{sec:1.1}. The equilateral lattice $\Lambda$ is $\Lambda = \Z v_1 \oplus \Z v_2$ given in canonical coordinates by
\begin{equation}
v_1 = a\matrice{\sqrt{3} \\ 1}, \ \ \ v_2= a\matrice{\sqrt{3} \\ -1},
\end{equation}
where $a>0$ is a constant such that $\Det[v_1,v_2]=1$. 
Let $k_1, k_2 \in (\R^2)^*$ be dual vectors: $\blr{k_i,v_j} = \delta_{ij}$. Identifying $(\R^2)^*$ with $\R^2$ via the scalar product, 
\begin{equation}
[k_1,k_2] \cdot [v_1,v_2] = \Id\ \Rightarrow \ [k_1,k_2] = [v_1,v_2]^{-1} = \dfrac{1}{6a} \matrice{\sqrt{3} & \sqrt{3} \\ 3 & -3}.
\end{equation}
Our definition does not involve a factor $2\pi$ -- in contrast with some other conventions. The fundamental cell $\Ll$ and dual cell $\Ll^*$ are 
\begin{equation}
\Ll \de \big\{ t_1 v_1 + t_2 v_2 : \ t_1, t_2 \in [0,1)  \big\}, \ \ \Ll^* \de \big\{ \tau_1 k_1 + \tau_2 k_2 : \ \tau_1, \tau_2 \in [0,2\pi)  \big\}.
\end{equation}

\subsection{Symmetries}\label{sec:2.2} Recall that the space of $\xi$-quasiperiodic functions is
\begin{equation}
L^2_\xi \de \left\{ u\in L^2_\loc \big(\R^2,\C\big) : \ u(x+w) = e^{i\blr{\xi,w}} u(x), w \in \Lambda \right\}.
\end{equation}
We introduce three operators: $\RR$ (rotation); $\II$ (spatial inversion); and $\CCC$ (complex conjugation). These are given by
\begin{equation}
\RR u(x) = u(Rx), \ \ R \de \dfrac{1}{2} \matrice{ -1 & \sqrt{3} \\ -\sqrt{3} & -1 } ; \ \ \ \ \II u(x) = u(-x); \ \ \ \ \CCC u(x) = \ove{u(x)}.
\end{equation}

We study the action of these operators on the spaces $L^2_\xi$. Note that $Rv_1 = -v_2$ and $Rv_2 = v_1-v_2$. Hence, $R$ leaves $\Lambda$ invariant. If $u \in L^2_\xi$ then
\begin{equations}
(\RR u)(x+v) =   u(Rx+Rv) = e^{i\blr{\xi, Rv}} (\RR u)(x) = e^{i\blr{R^*\xi, v}} (\RR u)(x),
\\
(\II u)(x+v) =   u(-x-v) = e^{-i\blr{\xi, v}} (\II u)(x), \ \ \ \ (\CCC u)(x+v) = \ove{u(x+v)} = e^{-i\blr{\xi, v}} (\CCC u)(x).
\end{equations}
It follows that
\begin{equation}\label{eq:2f}
\RR L^2_\xi = L^2_{R^{-1} \xi}; \ \ \ \ \II L^2_\xi = L^2_{-\xi}; \ \ \ \ \CCC L^2_\xi = L^2_{-\xi}.
\end{equation}

Let $\xi_\star^A$ and $\xi_\star^B$ be given by \eqref{eq:3y}:
\begin{equation}
\xi_\star^A = \dfrac{2\pi}{3} (2k_1+k_2), \ \ \ \ \xi_\star^B = \dfrac{2\pi}{3} (k_1+2k_2).
\end{equation}
We observe that
\begin{equation}
R^{-1} \xi_\star^A = \xi_\star^A+ 2\pi(k_1+k_2), \ \ \ \ R^{-1} \xi_\star^B = \xi_\star^B+2\pi k_1.
\end{equation}
In particular, $R^{-1} \xi_\star = \xi_\star \mod 2\pi\Lambda^*$ when $\xi_\star \in \{\xi_\star^A, \xi_\star^B\}$. Thanks to \eqref{eq:2f}, we see that the space $L^2_{\xi_\star}$ is $\RR$-invariant. Since $\RR^3 = \Id$, we deduce that $\RR : L^2_{\xi_\star} \rightarrow L^2_{\xi_\star}$ has three eigenvalues: $1, \tau, \otau$ with $\tau = e^{2i\pi/3}$. Since $\RR$ is a unitary operator,  $L^2_{\xi_\star}$ admits an orthogonal decomposition
\begin{equations}
L^2_{\xi_\star} = L^2_{\xi_\star,1} \oplus L^2_{\xi_\star,\tau} \oplus L^2_{\xi_\star,\otau},  \ \ \ \  \ \ 
L^2_{\xi_\star,z} \de \text{ker}_{L^2_{\xi_\star}}(\RR-z).
\end{equations}

The operator $\CCC \II$ maps $L^2_{\xi_\star}$ to itself. If $u \in L^2_{\xi_\star,\tau}$ then
\begin{equation}
\RR \big(\CCC \II u \big)(x) = \ove{u(-Rx)} = \ove{\tau \cdot u(-x)} = \otau \cdot \big(\CCC \II u \big)(x).
\end{equation}
Therefore $\CCC \II L^2_{\xi_\star,\tau} = L^2_{\xi_\star,\otau}$.

\subsection{Dirac points}\label{sec:3.2} We recall that $P_0 = -\Delta + V$, where $V$ is a honeycomb potential -- see Definition \ref{def:3}. We denote by 
\begin{equation}\label{eq:5b}
\lambda_{0,1}(\xi) \leq \lambda_{0,2}(\xi) \leq \dots \leq \lambda_{0,j}(\xi) \leq \dots
\end{equation}
the dispersion surfaces of $P_0$, i.e. the $L^2_\xi$-eigenvalues of $P_0(\xi)$. Conical intersections in the band spectrum \eqref{eq:5b} are called Dirac points -- see Definition \ref{def:1}. Fefferman and Weinstein \cite{FW} -- see also  \cite{CV,Gru,BC,Le,KMO,AFL} for related perspectives -- showed that honeycomb Schr\"odinger operators generically admit Dirac points
\begin{equation}
(\xi_\star, E_\star) \in \{ \xi_\star^A, \xi_\star^B \} \times \R, \ \ \ \ \xi_\star^A \de \dfrac{2\pi}{3}(2k_1+k_2), \ \ \ \ \xi_\star^B \de \dfrac{2\pi}{3}(k_1+2k_2).
\end{equation}
The eigenspace $\ker_{L^2_{\xi_\star}}\hspace*{-1.5mm}\big(P_0(\xi_\star) - E_\star\big)$ is spanned by an orthonormal basis $\{\phi_1,\phi_2\}$ with
\begin{equation}
\phi_1 \in L^2_{\xi_\star,\tau}, \ \ \ \ \phi_2 = \ove{\II \phi_1} \in L^2_{\xi_\star,\ove{\tau}}.
\end{equation}
We call $(\phi_1,\phi_2)$ a Dirac eigenbasis. It is unique modulo the $\Ss^1$-action $(\phi_1,\phi_2) \mapsto (\w \phi_1, \ove{\w} \phi_2)$, $\w \in \Ss^1$.

\begin{lem}\label{lem:1i} Let $(\xi_\star,E_\star)$ be a Dirac point of $P_0$ with Dirac eigenbasis $(\phi_1,\phi_2)$. Then
\begin{equation}
\blr{\phi_1, D_x \phi_1}_{L^2_{\xi_\star}} = \blr{\phi_2, D_x \phi_2}_{L^2_{\xi_\star}} = 0.
\end{equation}
In addition, there exists $\nu_\star \in \C$ with $|\nu_\star| = \nu_F$ such that for all $\eta \in \R^2$ (canonically identified with a complex number),
\begin{equation}
2\blr{\phi_1, (\eta \cdot D_x) \phi_2}_{L^2_{\xi_\star}} = \nu_\star \eta, \ \ \ \ 2\blr{\phi_2, (\eta \cdot D_x) \phi_1}_{L^2_{\xi_\star}} = \ove{\nu_\star \eta}.
\end{equation}
\end{lem}

This lemma could be deduced from \cite[Proposition 4.5]{FLTW3}. We wrote a proof in Appendix \ref{app:2}. It relies on some algebraic relations relating $P_0$, $\RR$ and $\II$; and on perturbation theory of eigenvalues.

\subsection{Breaking the symmetry}

We will consider Schr\"odinger operators $P_\delta = -\Delta + V + \delta W$, where:
\begin{equation}
W \in C^\infty\big(\R^2,\R\big); \ \ 
W(x+w) = W(x), \ w \in \Lambda; \ \ W(x) = -W(-x). 
\end{equation}

\begin{lem}\label{lem:1h} Let $(\xi_\star,E_\star)$ be a Dirac point of $P_0$ with Dirac eigenbasis $(\phi_1,\phi_2)$ -- see Definition \ref{def:1}. Then $\blr{\phi_1,W\phi_2}_{L^2_{\xi_\star}} = \blr{\phi_2,W\phi_1}_{L^2_{\xi_\star}} = 0$. Furthermore,
\begin{equation}
\var_\star \de \blr{\phi_1,W\phi_1}_{L^2_{\xi_\star}} = -\blr{\phi_2,W\phi_2}_{L^2_{\xi_\star}}.
\end{equation}
\end{lem}

See the proofs of \cite[(6.19), (6.20)]{FLTW3} or Appendix \ref{app:2}. These identities rely on $\II$ being an isometry. If $\w \in \Ss^1$, the change $(\phi_1,\phi_2) \mapsto (\w \phi_1, \ove{\w} \phi_2)$ of Dirac eigenbasis leaves $\var_\star$ invariant. 

\subsection{Edges} \label{sec:6.1} Let $a_1$ and $a_2$ be two relatively prime integers and $v = a_1 v_1 + a_2 v_2$. Introduce $v' = b_1 v_1 + b_2 v_2$, where $a_1 b_2 - a_2 b_1 = 1$. The vectors $v$ and $v'$ span $\Lambda$:
\begin{equations}\label{eq:4a}
b_1 v - a_1 v' = (b_1 a_2 - a_1 b_2) v_2 = - v_2, \\
b_2 v - a_2 v = (b_2 a_1 - a_2 b_1) v_1 = v_1.
\end{equations}
Let $k$ and $k'$ be dual vectors. We claim that $k = b_2 k_1 - b_1 k_2$ and $k' = -a_2 k_1 + a_1 k_2$:
\begin{equations}
\blr{k, v} = b_2 a_1 - b_1 a_2 = 1, \ \ \blr{k, v'} = - a_2 a_1 + a_1 a_2 = 0,
\\
\blr{k', v} = b_2 b_1 - b_1 b_2 = 0, \ \ \blr{k', v'} = -a_2 b_1 + a_1 b_2 = 1.
\end{equations}

Let $(\xi_\star^A,E_\star)$ be a Dirac point in the sense of Definition \ref{def:1} and $\R v$ be an edge.
Assume that $\xi_\star^B$ belongs to the dual edge $\zeta_\star^A k + \R k' \mod 2\pi \Lambda^*$. In this case we can write $\xi_\star^B = \zeta_\star^A k + \tau k'$ with $\tau \neq \lr{\xi_\star^A,v'} \mod 2\pi \Z$. Since $\lambda_{0,j_\star}(\xi_\star^B) = E_\star$, the no-fold condition fails when $\xi_\star^B \in \zeta_\star^A k + \R k' \mod 2\pi \Lambda^*$ (see Definition \ref{def:2}). Given the expressions \eqref{eq:3y} of $\xi_\star^A$ and $\xi_\star^B$ and \eqref{eq:3w} of $v'$, this arises precisely when
\begin{equation}
\dfrac{2a_1+a_2}{3} - \dfrac{a_1+2a_2}{3} \in \Z  \ \ \Leftrightarrow \ \ a_2-a_1 \in 3\Z.
\end{equation}
In particular, if the no-fold condition holds then $a_1 - a_2 \neq 0 \mod 3$. This implies that $\{\zeta_\star^A,\zeta_\star^B\} = \{2\pi/3,4\pi/3\} \mod 2\pi \Z$ because of \eqref{eq:4k}.

\section{The characterization of edge states}\label{sec:4}

This work studies the eigenvalues of the operator
\begin{equation}
\PP_\delta[\zeta] = - \Delta + V + \delta \cdot \kappa_\delta \cdot W : L^2[\zeta] \rightarrow L^2[\zeta].
\end{equation}
Above, $\kappa_\delta$ is a domain-wall function -- see \eqref{eq:3v} -- and $L^2[\zeta]$ is the space \eqref{eq:4n}. The operator $\PP_\delta[\zeta]$ is a Schr\"odinger operator that interpolates between $P_{\delta}[\zeta]$ at $-\infty$ and $P_{\delta}[\zeta]$ at $+\infty$. See Figure \ref{fig:12}. In this section we review the multiscale approach of \cite{FLTW3,FLTW4} and we derive Corollary \ref{cor:3} assuming Theorem \ref{thm:3}, in a slightly more general setting.

\begin{center}
\begin{figure}
\caption{$\PP_\delta[\zeta]$ is a Schr\"odinger operator with a typical potential represented above, with the zigzag edge $v_1-v_2$. Each red (resp. blue) circle represents an atomic (e.g. radial) potential. The resulting potential is not periodic with respect to $\Lambda$; rather it is periodic with respect to $\Z v$. 
}\label{fig:12}
{\includegraphics{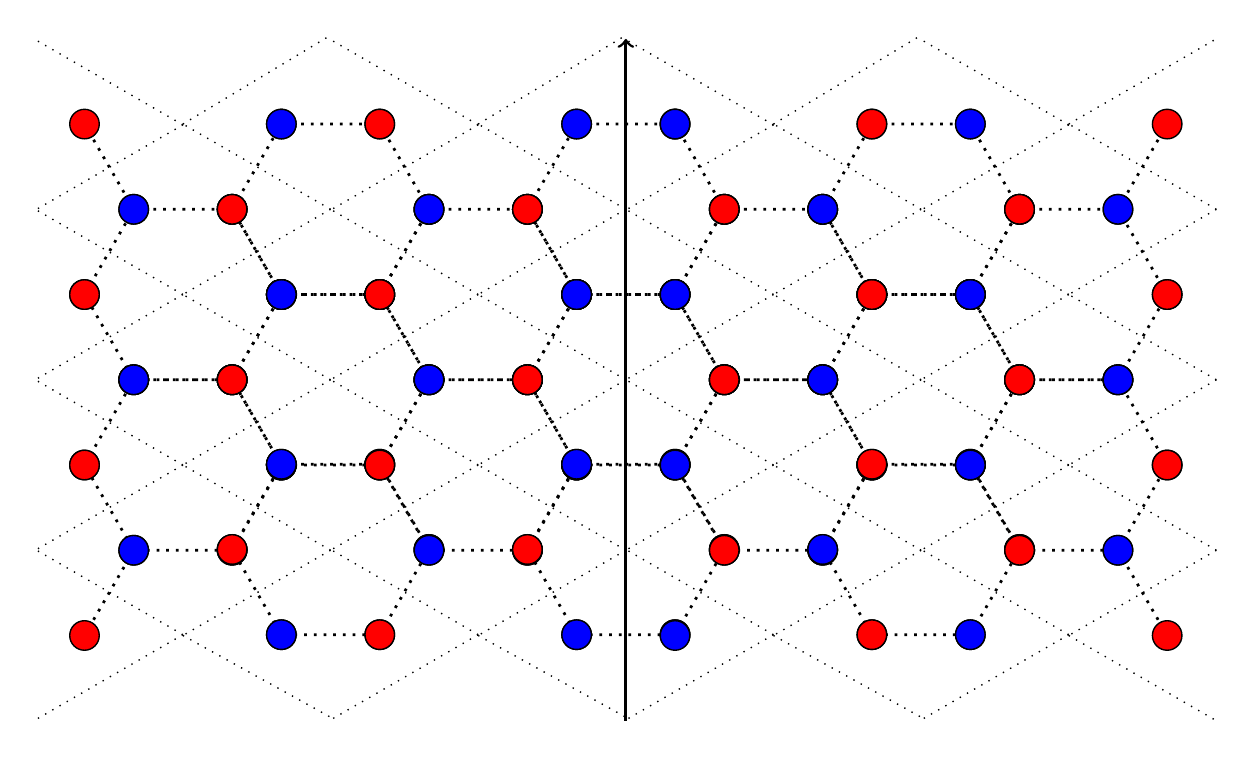}}
\end{figure}
\end{center}

\subsection{The formal multiscale approach}\label{sec:3.1} The eigenvalue problem for $\PP_\delta[\zeta]$ is
\begin{equation}\label{eq:1d}
\systeme{\big(-\Delta + V(x) + \delta \kappa_\delta(x)  W(x) - E_\delta\big) u_\delta = 0, \\
u_\delta(x + v) = e^{i\zeta} u_\delta(x),} \ \ \ \int_{\R^2/\Z v} |u_\delta(x)|^2 dx < \infty.
\end{equation}
The multiscale procedure of Fefferman--Lee-Thorp--Weinstein \cite[\S6]{FLTW3} produces approximate solutions of \eqref{eq:1d}. We review it below.

We first observe that if we write a function $u_\delta \in C^\infty(\R^2,\C)$ as
\begin{equation}\label{eq:1z}
u_\delta(x) = U_\delta(x,\delta \blr{k', x}), \ \ \ \ U_\delta \in C^\infty(\R^2 \times \R, \C),
\end{equation}
then $u_\delta$ solves \eqref{eq:1d} if and only if $U_\delta$ solves
\begin{equation}\label{eq:1t}
\systeme{ \big((D_x+\delta k' D_t)^2 + V(x) + \delta \kappa(t)W(x) - E_\delta\big) U_\delta = 0, \\ U_\delta(x+v,t) = e^{i\zeta} U_\delta(x,t), } \ \ \ \ \int_{\R^2/\Z v} \big| U_\delta(x,\delta\blr{k',x}) \big|^2 dx < \infty.
\end{equation}

We now produce approximate solutions to the system \eqref{eq:1t} when $\zeta$ is near $\zeta_\star = \blr{\xi_\star, v}$. We fix $(\xi_\star,E_\star)$ a Dirac point of $P_0$ and we write $\zeta = \zeta_\star + \mu \delta$, $\zeta_\star = \lr{\xi_\star,v}$. We make an ansatz for $U_\delta$ and $E_\delta$:
\begin{equation}\label{eq:5f}
U_\delta(x,t) = e^{i\mu\delta \blr{\ell,x} } \cdot \left(\sum_{j=1, 2} \az_j (t) \cdot \phi_j(x) + \delta \cdot V_\delta(x,t) \right), \ \ \ \ E_\delta = E_\star + \var\delta  + O\left(\delta^2\right)
\end{equation} 
where: 
\begin{itemize}
\item $(\phi_1,\phi_2)$ is a Dirac eigenbasis for $(\xi_\star,E_\star)$ -- see Definition \ref{def:1};
\item $\az_1, \az_2$ are smooth, exponentially decaying functions on $\R$, to be specified below;
\item $V_\delta \in X$ -- the space defined in \eqref{eq:5h}.
\item $\ell = k - \frac{\blr{k',k}}{|k'|^2} k'$ is the projection of $k$ to the orthogonal of $\R k'$.
\item $\var \in \R$ is a real number that will be specified below.
\end{itemize}
Since $\phi_1, \ \phi_2  \in L^2_{\xi_\star}$, $V_\delta \in X$ and $\az_1, \az_2 \in L^2(\R)$, the ansatz \eqref{eq:5f} implies:
\begin{equation}
U_\delta(x+v,t) = e^{i\zeta} U_\delta(x,t), \ \ \ \ \int_{\R^2/\Z v} \big|U_\delta(x,\delta\blr{k',x})\big|^2 dx < \infty.
\end{equation}
In particular the boundary and decay conditions of \eqref{eq:1t} hold under \eqref{eq:5f}. 

The eigenvalue problem \eqref{eq:1t}  becomes a hierarchy of equations, obtained by identifying terms of order $1, \delta, \delta^2, \dots$. Since $(P_0-E_\star) \phi_j = 0$, the equation for the terms of order $1$ is automatically satisfied. The equation for the terms of order $\delta$ is
\begin{equations}\label{eq:1w}
e^{i\mu\delta \blr{\ell, x}} (P_0-E_\star) V_\delta(x,t)
 + e^{i\mu\delta \blr{\ell, x}} \big(2(k'\cdot D_x) D_t + \kappa(t)W(x)-\var\big) \hspace*{-.7mm}  \sum_{j=1,2} \hspace*{-.7mm} \az_j (t) \phi_j(x)
\\ + 2\mu e^{i \mu \delta \blr{\ell, x}} (\ell \cdot D_x)  \sum_{j=1,2} \az_j (t) \phi_j(x) = 0.
\end{equations}
Note that for every $t \in \R$, $(P_0 - E_\star) V_\delta(\cdot,t)$ is orthogonal to $\phi_1$ and $\phi_2$. Therefore, for this system to have a solution, we must have for every $t \in \R$ and $k=1,2$, 
\begin{equation}\label{eq:1v}
\Big \langle \phi_k, \big(2(k'\cdot D_x) D_t + 2\mu (\ell \cdot D_x) \kappa(t)W-\var\big)  \sum_{j=1,2} \az_j (t) \cdot \phi_j\Big \rangle_{L^2_{\xi_\star}} = 0.
\end{equation}

The scalar products $\blr{ \phi_j, (k'\cdot D_x)\phi_k}_{L^2_{\xi_\star}}$, $\blr{ \phi_j, (\ell\cdot D_x)\phi_k}_{L^2_{\xi_\star}}$ and $\blr{ \phi_j, W\phi_k}_{L^2_{\xi_\star}}$ appear in the solvability condition \eqref{eq:1v}. They were computed in Lemma \ref{lem:1i} and \ref{lem:1h}. Using these formula, \eqref{eq:1v} simplifies to
\begin{equation}
\left(\Di(\mu) -\var \right) \matrice{\az_1 \\ \az_2} = 0, \ \ \ \ \Di(\mu) \de \matrice{0 & \nu_\star k' \\ \ove{\nu_\star k'} & 0} D_t + \mu \matrice{0 & \nu_\star \ell \\ \ove{\nu_\star \ell} & 0} + \var_\star \matrice{1 & 0 \\ 0 & -1} \kappa.
\end{equation}
This system has exponentially decaying solutions $[\az_1,\az_2]^\top$ if and only if $\var$ is an eigenvalue of $\Di(\mu)$. Under this condition, \eqref{eq:1w} has a solution $V_\delta$. In other words, this constructs a function $U_\delta$ such that \eqref{eq:1t} is satisfied modulo $O_X(\delta^2)$.

We can iterate this procedure to arbitrarily high orders in $\delta$. It produces a function $U_\delta$ such that \eqref{eq:1t} is satisfied modulo $O_X(\delta^M)$, for any $M$. Identifying $U_\delta$ with $u_\delta$ according to \eqref{eq:1z}, this procedure produces for any $M$ and any eigenvalue $\var$ of $\Di(\mu)$ a function $u_{\delta,M}$ that solves
\begin{equation}
\big( \PP_\delta[\zeta]-E_\delta \big) u_{\delta,M} = O_X\big(\delta^M\big), \ \ \ \ E_\delta = E_\star +\delta \var + O\big(\delta^2\big).
\end{equation}
This is an approximate solution to the eigenvalue problem \eqref{eq:1d}.

It is natural to ask whether these approximate solutions are close to eigenvectors. The work \cite{FLTW3} shows that this holds at first order in $\delta$. Below we state results that imply that this holds at \textit{any} order in $\delta$. This dramatically refines the main result of \cite{FLTW3}. Our approach relies on resolvent estimates rather than by-hand construction of eigenvectors. It comes with further improvements of \cite{FLTW3}:
\begin{itemize}
\item The precise counting of  eigenvalues of $\PP_\delta[\zeta]$;
\item An estimate that connects the resolvents of $\PP_\delta[\zeta]$ and $\Di(\mu)$.
\end{itemize}
These results are stated in \S\ref{sec:4.3} and first require a spectral analysis of $\Di(\mu)$.

\subsection{The Dirac operator $\Di(\mu)$}\label{sec:4.2}
The Dirac operator
\begin{equation}
\Di(\mu) = \matrice{0 & \nu_\star k' \\ \ove{\nu_\star k'} & 0} D_t + \mu \matrice{0 & \nu_\star \ell \\ \ove{\nu_\star \ell} & 0} + \var_\star \matrice{1 & 0 \\ 0 & -1} \kappa
\end{equation}
emerges in the multiscale analysis of \cite{FLTW3}. We saw that its eigenvalues are particularly relevant in the construction of approximate eigenvectors of $\PP_\delta[\zeta]$, $\zeta = \zeta_\star + \delta \mu$. In this section we relate the spectra of $\Di(\mu)$ and $\Di_\star = \Di(0)$.

\begin{lem}\label{lem:1a} The essential and discrete spectra of $\Di_\star$ and $\Di(\mu)$ are related through:
\begin{equations}
\Sigma_{L^2,\ess}\big(\Di(\mu)\big) = \R \setminus \Big[-\sqrt{\var_F^2 + \mu^2 \cdot \nu_F^2 |\ell|^2 }, \sqrt{\var_F^2 + \mu^2 \cdot \nu_F^2 |\ell|^2} \Big], 
\\
\Sigma_{L^2,\pp}\big(\Di(\mu)\big) = \big\{ \mu \cdot \nu_F |\ell| \cdot \sgn(\var_\star), \ \ \pm \sqrt{\var_j^2 + \mu^2 \cdot \nu_F^2 |\ell|^2 } \ \text{ with } 0 \neq \var_j \in \Sigma_{L^2,\pp}\big(\Di_\star\big) \big\}.
\end{equations}
All the eigenvalues of $\Di_\star$ and $\Di(\mu)$ are simple.
\end{lem}

The proof of Lemma \ref{lem:1a} relies on a super-symmetry: there exists a $2 \times 2$ matrix $\bmd$ such that $\bmd^2 = \Id$ and $\bmd \Di_\star = -\bmd\Di_\star$. We postpone it to Appendix \ref{app:1}. We also mention that $\Di_\star$ may have more than one eigenvalues -- see \cite{LWW}. For a general perspective for applications of super-symmetries in spectral theory, see \cite[\S6-12]{CFK}.

\subsection{Parallel quasimomentum near $\zeta_\star$}\label{sec:4.3} We are now ready to state the main result of our work. Recall that the assumptions $\operatorname{(H1)}$ -- $\operatorname{(H3)}$ were introduced in \S\ref{sec:1.5}; and that $\Pi, \Pi^*$ and $\UU_\delta$ are defined by
\begin{equations}
\Pi : L^2\big(\R^2/\Z v, \C^2\big) \rightarrow L^2\big(\R,\C^2\big), \ \ \ \ \big(\Pi f\big)(t) \de \int_0^1 f(sv + tv') ds; 
\\
\Pi^* : L^2\big(\R, \C^2\big) \rightarrow L^2\big(\R^2/\Z v,\C^2\big), \ \ \ \ \big(\Pi^* g\big)(x) \de g\big(\blr{k',x}\big);
\\
\UU_\delta : L^2\big(\R,\C^2\big) \rightarrow L^2\big(\R,\C^2\big), \ \ \ \ \big(\UU_\delta f\big)(t) \de f(\delta t).
\end{equations}

\begin{theorem}\label{thm:2} Assume that the assumptions $\operatorname{(H1)}$ -- $\operatorname{(H3)}$ hold.  Fix $\tmu > 0$ and $\epsilon > 0$. There exists $\delta_0 > 0$ such that if
\begin{equations}
\mu \in (-\tmu,\tmu), \  \delta \in (0,\delta_0), \  z \in \Dd\Big(0,\sqrt{\var_F^2 + \mu^2\cdot \nu_F^2 | \ell|^2} -\epsilon \Big), \  
 \dist\Big( \Sigma_{L^2}\big(\Di(\mu) \big), z \Big) \geq \epsilon, 
\end{equations}
\begin{equations}\label{eq:3r}
 \zeta = \zeta_\star +  \delta\mu, \ \ \lambda = E_\star + \delta z 
\end{equations}
then $\PP_\delta[\zeta] - \lambda$ is invertible and its resolvent $\big( \PP_\delta[\zeta] - \lambda \big)^{-1}$ equals
\begin{equation}
 \dfrac{1}{\delta} \cdot  \matrice{ \phi_1 \\ \phi_2}^\top  e^{-i  \mu \delta \blr{\ell,x} } \cdot \Pi^*   \UU_\delta \cdot \big( \Di(\mu)-z \big)^{-1}  \cdot\UU_\delta^{-1}  \Pi \cdot e^{i  \mu \delta \blr{\ell,x} } \ove{\matrice{\phi_1 \\ \phi_2}} + \OO_{L^2[\zeta]}\left(\delta^{-1/3}\right).
\end{equation}
\end{theorem}

It suffices to take $\mu=0$ in Theorem  \ref{thm:2} to derive Theorem \ref{thm:3}. 

\begin{cor}\label{cor:2} Assume $\operatorname{(H1)}$ -- $\operatorname{(H3)}$ hold and fix $\tvar \in (\var_N,\var_F)$ and $\tmu > 0$. There exists $\delta_0 > 0$ such that for 
\begin{equation}
\delta \in (0,\delta_0), \ \  \mu \in (- \tmu,\tmu), \ \ \zeta = \zeta_\star + \delta \mu,
\end{equation}
the operator $\PP_\delta[\zeta]$ has exactly $2N+1$ eigenvalues $\{E_{\delta,j}^\zeta\}_{j \in [-N,N]}$ in 
\begin{equation}
\left[E_\star - \delta \sqrt{\tvar^2 + \mu^2 \cdot \nu_F^2 | \ell|^2 }, \ E_\star + \delta\sqrt{\tvar^2 +  \mu^2\cdot \nu_F^2 | \ell|^2} \ \right].
\end{equation}

\noindent These eigenvalues are simple. Furthermore, for each $j \in [-N,N]$, the eigenpairs $(E_{\delta,j}^\zeta,u_{\delta,j}^\zeta)$ admit full expansions in powers of $\delta$:
\begin{equations}
E_{\delta,j}^\zeta = E_\star + \var_j^\mu \cdot \delta + a_2^\mu \cdot \delta^2 + \dots + a_M^\mu \cdot \delta^M + O\left( \delta^{M+1} \right), 
\\
u_{\delta,j}^\zeta(x) = e^{i(\zeta-\zeta_\star) \blr{\ell,x}} \Big( f_0^\mu\big(x,\delta \blr{k',x} \hspace{-.5mm} \big) + \dots + \delta^M \cdot  f_M\big(x,\delta \blr{k',x}\big) \hspace{-.5mm} \Big) + o_{H^k}\left( \delta^M \right).
\end{equations}
In the above expansions:
\begin{itemize}
\item $M$ and $k$ are any integer; $H^k$ is the $k$-th order Sobolev space.
\item $\var_j^\mu$ is the $j$-th eigenvalue of $\Di(\mu)$, described in Lemma \ref{lem:1a}.
\item The terms $a_m^\mu \in \R$, $f_m^\mu \in X$ are recursively constructed via the multiscale analysis of \cite{FLTW3} -- see \S\ref{sec:3.1}.
\item The leading order term $f_0^\mu$ satisfies
\begin{equation}
f_0^\mu(x,t) = \az_1^\mu(t) \phi_1(x) + \az_2^\mu(t) \phi_2(x), \ \ \ \ \left(\Di(\mu) - \var_j^\mu\right) \matrice{\az_1^\mu \\ \az_2^\mu} = 0.
\end{equation}
\end{itemize}
\end{cor}

\begin{proof}[Proof of Corollary \ref{cor:2} assuming Theorem \ref{thm:2}] In order to locate eigenvalues of $\PP_\delta[\zeta]$, it suffices to integrate the resolvent on contours enclosing regions where Theorem \ref{thm:2} does not apply.

Let $\var_j$ be an eigenvalue of $\Di(\mu)$ and $\epsilon > 0$ so that $\Di(\mu)$ has no other eigenvalues in $\Dd(\var_j,\epsilon)$. We compute the residue
\begin{equations}\label{eq:3i}
\dfrac{1}{2\pi i} \oint_{\p \Dd(E_\star + \delta\var_j,\epsilon\delta)} \big(\lambda - \PP_\delta(\zeta)\big)^{-1} d\lambda.
\end{equations}
This the projector on the spectrum of $\PP_\delta(\zeta)$ that is enclosed by $\p \Dd(E_\star + \delta\var_j,\epsilon\delta)$. Because of Theorem \ref{thm:2} and of the relation $\lambda = E_\star + \delta z$, $d\lambda = \delta dz$, \eqref{eq:3i} equals
\begin{equations}
\matrice{ \phi_1 \\ \phi_2}^\top  e^{-i  \mu \delta \blr{\ell,x} }  \cdot \Pi^*   \UU_\delta \cdot \dfrac{1}{2\pi i} \oint_{\p\Dd(\var_j,\epsilon)} \big( z- \Di(\mu) \big)^{-1} dz \cdot \UU_\delta^{-1}  \Pi \cdot e^{i  \mu \delta \blr{\ell,x} } \ove{\matrice{\phi_1 \\ \phi_2}} + \OO_{L^2[\zeta]}\left(\delta^{2/3}\right).
\end{equations}
The residue 
\begin{equation}
\dfrac{1}{2\pi i} \oint_{\p\Dd(\var_j,\epsilon)} \big( z- \Di(\mu) \big)^{-1} dz
\end{equation}
is a rank-one projector on $\ker_{L^2}(\Di-\var_j)$. We write it $\az^\zeta \otimes \az^\zeta$, where $|\az^\zeta|_{L^2} = 1$. We deduce that the residue \eqref{eq:3i} equals
\begin{equations}
\matrice{ \phi_1 \\ \phi_2}^\top  e^{-i  \mu \delta \blr{\ell,x} }  \cdot \Pi^*   \UU_\delta \cdot \az \otimes \az \cdot \UU_\delta^{-1}  \Pi \cdot e^{i  \mu \delta \blr{\ell,x} } \ove{\matrice{\phi_1 \\ \phi_2}} + \OO_{L^2[\zeta]}\left(\delta^{2/3}\right)
\\
= v_0^\zeta \otimes v_0^\zeta + \OO_{L^2[\zeta]}\left(\delta^{2/3}\right), \ \ \ \ v_0^\zeta \de \delta^{1/2} \matrice{ \phi_1 \\ \phi_2}^\top  e^{-i  \mu \delta \blr{\ell,x} }  \Pi^*   \UU_\delta \cdot \az.
\end{equations}
Above we used that $(\UU_\delta^{-1})^* = \delta \cdot \UU_\delta$.

We deduce that \eqref{eq:3i} is a projector that takes the form $v_0^\zeta \otimes v_0^\zeta + \OO_{L^2[\zeta]}(\delta^{2/3})$, where $|v_0^\zeta|_{L^2[\zeta]} = 1$. In particular, it is non-zero. Moreover, it has rank at most one. Indeed, normalized vectors in its range must be of the form $v_0^\zeta + O_{L^2[\zeta]}(\delta^{2/3})$, therefore two of them cannot be orthogonal for $\delta$ sufficiently small. We deduce that \eqref{eq:3i} has rank exactly one: $\PP_\delta[\zeta]$ has exactly one eigenvalue in $\Dd(E_\star + \delta\var_j,\epsilon\delta)$.

The rest of the proof is identical to \cite[Proof of Corollary 1]{DFW}. It relies on:
\begin{itemize}
\item The fact that $\PP_\delta[\zeta]$ has exactly one eigenvalue in the disk enclosed by $\p \Dd(E_\star + \delta\var_j,\epsilon\delta)$ -- proved just above;
\item A general variational argument that shows that an approximate eigenpair $(\psi,E)$ for a selfadjoint problem that has only one eigenvalue near $E$ must be close to a genuine eigenpair -- see \cite[Lemma 3.1]{DFW}.
\item The construction of arbitrarily accurate approximate eigenpairs thanks to the multiscale procedure of \cite{FLTW3} -- see \S\ref{sec:3.1}.
\end{itemize}
We refer to \cite[Proof of Corollary 1]{DFW} for details.
\end{proof}

Most of the rest of the paper is devoted to the proof of Theorem \ref{thm:2}.


\section{The Bloch resolvent}\label{sec:5}
 
Recall that $V$ is a honeycomb potential -- see Definition \ref{def:3} --  and that $W \in C^\infty(\R^2,\R)$ is odd and $\Lambda$-periodic. In this section we study the resolvent of $P_\delta(\xi)$, the operator formally equal to $P_\delta = -\Delta+V+\delta W$  but acting on quasiperiodic spaces $L^2_\xi$.

Under the no-fold condition, we prove in Lemma \ref{lem:1x} that $(P_\delta(\xi)-z)^{-1}$ is subdominant away from the Dirac quasimomenta $\xi_\star$. The situation is more subtle near $\xi_\star$. In Lemma \ref{lem:1f} we show that when the non-degeneracy assumption \eqref{eq:3h} holds and $(\xi,\lambda)$ is near a Dirac point $(\xi_\star,E_\star)$, $(P_\delta(\xi)-\lambda)^{-1}$ behaves like the resolvent of a rank-two operator.

\subsection{Resolvent away from Dirac momenta}\label{sec:12.1}

We recall that $\Ll$ is the fundamental cell associated to the generators $v_1$ and $v_2$, see \eqref{eq:4c}. Given $\xi \in \R^2$, we define $\rho(\xi)$ as
\begin{equation}
\rho(\xi) \de \dist \big( \xi + 2\pi \Lambda^*, \ \zeta_\star k + \R k'\big).
\end{equation}

\begin{lem}\label{lem:1x} Assume that the assumptions $\operatorname{(H1)}$ and $\operatorname{(H2)}$ hold. Let $c > 0$. There exist $\delta_0, \ \epsi_0 > 0$ such that if 
\begin{equation}\label{eq:3o}
\delta \in (0,\delta_0), \ \ \xi \in \Ll^*, \ \ \rho(\xi) \leq \epsi_0,\ \ \left|\xi - \xi_\star \right| \geq \delta^{1/3}, \ \ \lambda \in \Dd\left( E_\star, c\delta \right)
\end{equation}
then $P_\delta(\xi)-\lambda$ is invertible and 
\begin{equation}
\left\| (P_\delta(\xi) - \lambda)^{-1} \right\|_{L^2_\xi \rightarrow H^2_\xi} = O\left(\delta^{-1/3}\right).
\end{equation}
\end{lem}

\vspace*{-.8cm}

\begin{center}
\begin{figure}
\caption{If $v = v_1-v_2$ is the zigzag edge then $k' = k_1+k_2$. A $\epsi_0$-neighborhood of the dual line $\zeta_\star k + \R k'$ is represented above as the blue strip. Lemma \ref{lem:1x} applies to quasimomenta in the area enclosed in black. This domain of validity extends by periodicity to the whole blue strips away from $\xi_\star \mod 2\pi \Lambda^*$. }
{\includegraphics{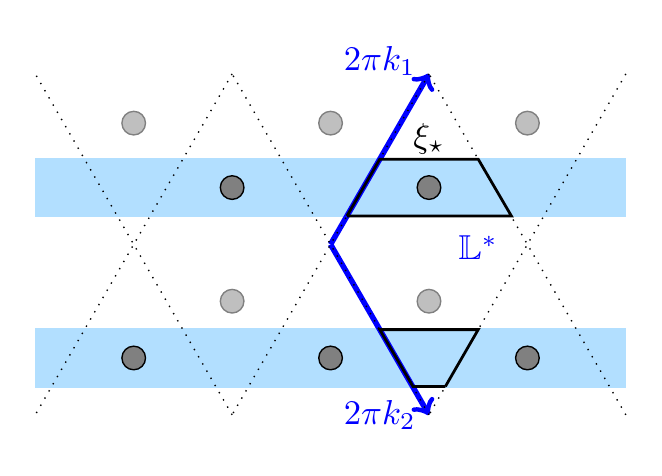}}
\end{figure}
\end{center}

\begin{proof} 1. We first show that there exists $\epsi_0 > 0$ such that
\begin{equation}\label{eq:5k}
\xi \in \Ll^* \setminus\{\xi_\star\}, \ \ \rho(\xi) \leq \epsi_0, \ \  \Rightarrow \ \ \lambda_{0,j_\star}(\xi) < E_\star - 2\epsi_0 \cdot |\xi-\xi_\star|.
\end{equation}
Indeed, if this does not hold then we can find $\xi_n$ such that 
\begin{equation}
\xi_n \in \Ll^* \setminus\{\xi_\star\}, \ \ \rho(\xi_n) \leq \dfrac{1}{n}, \ \  \lambda_{0,j_\star}(\xi_n) \geq E_\star - \dfrac{2}{n} \cdot |\xi-\xi_\star|.
\end{equation}
Since $\xi_n \in \Ll^*$, $\xi_n$ is bounded.
There exists a subsequence $\xi_{\vp(n)}$ of $\xi_n$ that converges to an element $\xi_\infty$ in the closure of $\Ll^*$, with $\rho(\xi_\infty) = 0$. Because $\lambda_{0,j_\star}$ is continuous, we have $\lambda_{0,j_\star}(\xi_\infty) \geq E_\star$. Since $\rho(\xi_\infty) = 0$, there exist $\eta \in \Lambda^*$ and $\tau_0 \in \R$ such that 
\begin{equation}
\xi_\infty + 2\pi \eta = \zeta_\star k+\tau_0 k'. 
\end{equation}
We look at the function $\vp(\tau) \de \lambda_{0,j_\star}(\zeta_\star k + \tau k')$. It is $2\pi$-periodic and it equals $E_\star$ precisely when $\tau = \lr{\xi_\star,v'} \mod 2\pi$ because of (H2). Moreover,
\begin{equation}
\vp(\lr{\xi_\star,v'} + \epsilon) = E_\star - \nu_F|\epsilon k'| + O\left(\epsilon^2\right).
\end{equation}
Therefore, the intermediate value theorem shows that $\vp(\tau) < E_\star$ unless $\tau = \lr{\xi_\star,v'} \mod 2\pi$. We deduce that $\tau_0 = \lr{\xi_\star,v'} \mod 2\pi$. Hence $\xi_\infty = \xi_\star \mod 2\pi \Lambda^*$. Since $\xi_\infty$ is in the closure of $\Ll^*$, $\xi_\infty = \xi_\star$. Since it also belongs to $\zeta_\star k + \R k'$, we have  $\xi_\infty = \xi_\star$. Since $\xi_\star$ is a Dirac point, we deduce
\begin{equation}
E_\star - \dfrac{2}{\vp(n)} \cdot |\xi_{\vp(n)}-\xi_\star| \leq \lambda_{0,j_\star}(\xi_{\vp(n)})  \leq E_\star - \nu_F \cdot |\xi_{\vp(n)}-\xi_\star| + O(\xi_{\vp(n)}-\xi_\star)^2.
\end{equation}
This cannot hold for large $n$, unless $\xi_{\vp(n)} = \xi_\star$, which is excluded. We deduce that \eqref{eq:5k} holds. A similar argument implies that
\begin{equation}\label{eq:5l}
\xi \in \Ll^* \setminus \{\xi_\star\}, \ \ \rho(\xi) \leq \epsi_0 \ \  \Rightarrow \ \ \lambda_{0,j_\star+1}(\xi) > E_\star + 2\epsi_0 \cdot |\xi-\xi_\star|.
\end{equation}

2. From \eqref{eq:5k} and \eqref{eq:5l}, we deduce that for $\delta > 0$,
\begin{equations}
\xi \in \Ll^*, \ \ \rho(\xi)  \leq \epsi_0, \ \ |\xi - \xi_\star| \geq \delta^{1/3} \ \Rightarrow \ \systeme{ \lambda_{0,j_\star}(\xi) < E_\star - 2\epsi_0 \delta^{1/3}, \\
\lambda_{0,j_\star+1}(\xi) > E_\star + 2\epsi_0 \delta^{1/3}.}
\end{equations}
In particular, if $c > 0$ is given and $\lambda \in \Dd(E_\star,c\delta)$ then
\begin{equation}
\xi \in \Ll^*, \ \ \rho(\xi) \leq \epsi_0, \ \ |\xi - \xi_\star| \geq \delta^{1/3} \ \Rightarrow \ \systeme{ \ \Re\big(\lambda_{0,j_\star}(\xi)-\lambda\big) \  < c\delta - 2\epsi_0 \delta^{1/3}, \\
\Re\big(\lambda_{0,j_\star+1}(\xi)-\lambda\big)  > 2\epsi_0 \delta^{1/3}-c\delta.}
\end{equation}
In particular, when $\delta_0$ is sufficiently small, $\delta \in (0,\delta_0)$ and $\lambda \in \Dd(E_\star,c\delta)$,
\begin{equation}
\xi \in \Ll^*, \ \ \rho(\xi) \leq  \epsi_0, \ \ |\xi - \xi_\star| \geq \delta^{1/3} \ \Rightarrow \ \systeme{ \ \ \  \Re\big(\lambda_{0,j_\star}(\xi)-\lambda\big) \   < - \epsi_0 \delta^{1/3}, \\
\Re\big(\lambda_{0,j_\star+1}(\xi)-\lambda\big) > \epsi_0 \delta^{1/3}.}
\end{equation}
Since the dispersion surfaces are labeled in increasing order, we deduce that if \eqref{eq:3o} is satisfied then
\begin{equation}
\dist\big(\Sigma_{L^2_\xi}\big(P_0(\xi)\big),\lambda\big) \geq \epsi_0 \delta^{1/3}, \ \ \ \
\big(P_0(\xi) - \lambda\big)^{-1} = \OO_{L^2_\xi}\left(\delta^{-1/3}\right).
\end{equation}
We derived the estimate on $(P_0(\xi) - \lambda)^{-1}$ using the spectral theorem.

2. Assume that \eqref{eq:3o} holds. Thanks to Step 1, $P_0(\xi) - \lambda$ is invertible and
\begin{equation}
P_\delta(\xi) - \lambda = P_0(\xi) -\lambda + \delta W = (P_0(\xi) -\lambda) \cdot \left(\Id + \big(P_0(\xi) -\lambda\big)^{-1} \delta W \right).
\end{equation}
The second term equals $\Id + \OO_{L^2_\xi}(\delta^{2/3})$. In particular it is invertible by a Neumann series for $\delta$ sufficiently small, with uniformly bounded inverse. We deduce that $P_\delta(\xi) - \lambda$ is invertible with inverse $\OO_{L^2_\xi}\left(\delta^{-1/3}\right)$.

3. To conclude we must show that the inverse of $P_\delta(\xi) - \lambda$ is $\OO_{L^2_\xi \rightarrow H^2_\xi}(\delta^{-2/3})$. This is a standard consequence of the elliptic estimate: using $\delta = O(1)$, $\lambda = O(1)$, we see that for any $f \in H^2_\xi$,
\begin{equation}
|f|_{H^2_\xi} \leq |f|_{L^2_\xi} + |\Delta f|_{H^2_\xi} \leq C|f|_{L^2_\xi} + |(P_\delta(\xi)-\lambda) f|_{H^2_\xi}.
\end{equation}
We apply this inequality to $f = (P_\delta(\xi)-\lambda)^{-1} u$ to deduce that
\begin{equation}
\left\| (P_\delta(\xi)-\lambda)^{-1} \right\|_{L^2_\xi \rightarrow H^2_\xi} \leq C \left\|(P_\delta(\xi)-\lambda)^{-1}\right\|_{L^2_\xi} + 1.
\end{equation}
In particular, the estimate $\OO_{L^2_\xi}(\delta^{-2/3})$ proved in Step 2 improves automatically to a bound $\OO_{L^2_\xi \rightarrow H^2_\xi}\left(\delta^{-1/3}\right)$. This completes the proof.
\end{proof}

\subsection{Resolvent near Dirac momenta}\label{sec:5.3}
 
Fix a Dirac point $(\xi_\star,E_\star)$ of $P_0(\xi)$ and assume that $\var_\star$ -- defined in \eqref{eq:3h} -- is non-zero. Identify $\xi-\xi_\star \in \R^2$ with the corresponding complex number and introduce
 $M_\delta(\xi)$ the $2 \times 2$ matrix
\begin{equation}
M_\delta(\xi) \de \matrice{E_\star + \delta \var_\star &\nu_\star \cdot (\xi-\xi_\star)  \\  \ove{\nu_\star \cdot (\xi-\xi_\star)}& E_\star-\delta\var_\star}.
\end{equation}

\begin{lem}\label{lem:1f} Let $\te \in (0,1)$. If
\begin{equation}\label{eq:1a}
\delta > 0, \ \ \xi \in \R^2, \ \ \var_F \de |\var_\star| \neq 0, \ \ \lambda \in \Dd\left( E_\star,\te \sqrt{\var_F^2 \cdot \delta^2 + \nu_F^2 \cdot |\xi-\xi_\star|^2} \right)
\end{equation}
then the matrix $M_\delta(\xi)-\lambda$ is invertible and
\begin{equation}
\left\| \big(M_\delta(\xi) - \lambda\big)^{-1} \right\|_{\C^2} = O\left( (\delta+|\xi-\xi_\star|)^{-1} \right).
\end{equation}
\end{lem}

\begin{proof} The matrix $M_\delta(\xi)$ is hermitian. It has eigenvalues 
\begin{equation}
\mu_\delta^\pm (\xi) \de E_\star \pm \sqrt{\var_F^2 \cdot \delta^2 + \nu_F^2 \cdot |\xi-\xi_\star|^2}.
\end{equation}
If \eqref{eq:1a} holds then the eigenvalues $\mu_\delta^\pm (\xi)-\lambda$ of $M_\delta(\xi) - \lambda$ satisfy
\begin{equation}
|\mu_\delta^\pm (\xi)-\lambda| \geq (1-\te) \sqrt{\var_F^2 \cdot \delta^2 + \nu_F^2 \cdot |\xi-\xi_\star|^2} \geq \dfrac{1-\te}{\sqrt{2}} \cdot \big( \nu_F \cdot |\xi-\xi_\star| + \var_F  \cdot \delta \big).
\end{equation}
By the spectral theorem, we deduce that $(M_\delta(\xi)-\lambda)^{-1}$ exists and has operator-norm bounded by $O\left((|\xi-\xi_\star|+\delta)^{-1}\right)$. This completes the proof.
\end{proof}

Introduce the operator
\begin{equation}\label{eq:1r}
\Pi_0(\xi) : L^2_\xi \rightarrow \C^2, \ \ \Pi_0(\xi)  u \de \matrice{ \blr{e^{i\blr{\xi-\xi_\star,x}} \phi_1,   u}_{L^2_\xi} \\ \blr{e^{i\blr{\xi-\xi_\star,x}} \phi_2,  u }_{L^2_\xi}}.
\end{equation}

\begin{lem}\label{lem:1g} Assume that the assumptions $\operatorname{(H1)}$ and $\operatorname{(H3)}$ hold. 
Let $\te \in (0,1)$. There exists $\delta_0 > 0$ such that if
\begin{equation}\label{eq:3e}
\delta \in (0,\delta_0), \ \ |\xi - \xi_\star| \leq \delta^{1/3}, \ \ \lambda \in \Dd\left( E_\star,\te \sqrt{ \var_F^2 \cdot \delta^2+ \nu_F^2 \cdot |\xi-\xi_\star|^2} \right)
\end{equation}
then $P_\delta(\xi)-\lambda$ is invertible and 
\begin{equation}
 \big( P_\delta(\xi) - \lambda \big)^{-1} = \Pi_0(\xi)^* \cdot \big( M_\delta(\xi)-\lambda \big)^{-1} \cdot \Pi_0(\xi) + \OO_{L^2_\xi \rightarrow H^2_\xi}(1).
\end{equation}
\end{lem}

\begin{proof} 1. Introduce the $\xi$-dependent family of vector spaces
\begin{equation}
\VVV(\xi) = \C \cdot e^{i\blr{\xi-\xi_\star,x}} \phi_1 \oplus \C \cdot e^{i\blr{\xi-\xi_\star,x}} \phi_2.
\end{equation}
We split $L^2_\xi$ as $\VVV(\xi) \oplus \VVV(\xi)^\perp$. With respect to this decomposition, we write $P_\delta(\xi)$ as a block-by-block operator:
\begin{equation}\label{eq:5p}
P_\delta(\xi) - \lambda = \matrice{A_\delta(\xi)-\lambda & B_\delta(\xi) \\ C_\delta(\xi) & D_\delta(\xi)-\lambda}.
\end{equation}
We use below $\lr{\cdot, \cdot}$ to denote the $L^2_\xi$-scalar product.

2. We show that
\begin{equation}\label{eq:3b}
 B_\delta(\xi)  = \OO_{\VVV(\xi)^\perp \rightarrow \VVV(\xi)}\big(\delta+|\xi-\xi_\star|\big), \ \ \ \ C_\delta(\xi) = \OO_{\VVV(\xi) \rightarrow \VVV(\xi)^\perp}\big(\delta+|\xi-\xi_\star|\big).
\end{equation}
Note that $C_\delta(\xi) = B_\delta(\xi)^*$ hence we just have to estimate $B_\delta(\xi)$, i.e. show that
\begin{equation}\label{eq:3a}
u \in \VVV(\xi)^\perp, \ |u|_{L^2_\xi} =1 \ \  \Rightarrow \ \ \blr{e^{i\blr{\xi-\xi_\star,x}} \phi_j, P_\delta(\xi) u} = O\big(\delta+ |\xi-\xi_\star|\big)
\end{equation}
where the implicit constant does not depend on $u$.
We have
\begin{equations}
\blr{e^{i\blr{\xi-\xi_\star,x}} \phi_j, P_\delta(\xi) u} = \blr{P_\delta(\xi)  \cdot e^{i\blr{\xi-\xi_\star,x}} \phi_j,  u} 
= \blr{(-\Delta+V+\delta W)\cdot e^{i\blr{\xi-\xi_\star,x}} \phi_j,  u}
\\
= \blr{e^{i\blr{\xi-\xi_\star,x}} (-\Delta+V)\phi_j, u} + \blr{\left[-\Delta, e^{i\blr{\xi-\xi_\star,x}}\right] \phi_j,  u} + \delta \blr{ W e^{i\blr{\xi-\xi_\star,x}} \phi_j,  u}  
\\
= (E_\star+|\xi-\xi_\star|^2) \blr{e^{i\blr{\xi-\xi_\star,x}} \phi_j, u} + 2\blr{e^{i\blr{\xi-\xi_\star,x}} (\xi-\xi_\star) \cdot D_x \phi_j,  u} +  \delta \blr{ W e^{i\blr{\xi-\xi_\star,x}} \phi_j,  u}.
\end{equations}
The first bracket vanishes because $u \in \VVV(\xi)^\perp$. The second and third brackets are $O(\xi-\xi_\star)$ and $O(\delta)$, respectively -- and this holds uniformly in $u$ with $|u|_{L^2_\xi} = 1$. This gives \eqref{eq:3a}, itself implying \eqref{eq:3b}.

3. Here we prove that if \eqref{eq:3e} is satisfied then 
\begin{equations}
D_\delta(\xi) - \lambda : \VVV(\xi)^\perp \cap H^2_\xi \rightarrow \VVV(\xi)^\perp \cap L^2_\xi \ \ \text{ is invertible and } \\
\big( D_\delta(\xi) - \lambda \big)^{-1}  = \OO_{\VVV(\xi)^\perp}(1).
\end{equations}
It suffices to construct an operator $E_\delta(\xi,\lambda) : \VVV(\xi)^\perp \rightarrow \VVV(\xi)^\perp$ such that
\begin{equation}\label{eq:2l}
E_\delta(\xi,\lambda) = \OO_{\VVV(\xi)^\perp}(1), \ \ \ \ E_\delta(\xi,\lambda) \cdot \big(D_\delta(\xi) - \lambda\big) = \Id_{\VVV(\xi)^\perp} + \OO_{\VVV(\xi)^\perp}\big(\delta+|\xi-\xi_\star|\big).
\end{equation}

The space $\VVV(\xi_\star) = \ker_{L^2_{\xi_\star}}\hspace*{-1.5mm}\big(P_0(\xi_\star) - E_\star\big)$ has dimension $2$; $P_0(\xi)$ depends smoothly on $\xi$ in the sense that $e^{-i\xi x} \cdot P_0(\xi) \cdot e^{i\xi x}$ forms a smooth family of operators $H^2_0 \rightarrow L^2_0$. Therefore, there exist $\eta > 0$ and $\epsi > 0$ such that
\begin{equation}\label{eq:5m}
|\xi-\xi_\star| \leq \epsi \ \Rightarrow \ P_0(\xi) \text{ has precisely two eigenvalues in } [E_\star-\eta,E_\star+\eta].
\end{equation}
See \cite[\S VII.1.3, Theorem 1.8]{Ka}. 
Let $\WWW(\xi)$ be the vector space spanned by the two eigenvectors of $P_0(\xi)$ with energy in $[E_\star-\eta,E_\star+\eta]$. Let $Q_0(\xi)$ be the operator formally equal to $P_0(\xi)$ but acting on $\WWW(\xi)^\perp$. From \eqref{eq:5m}, for $|\xi-\xi_\star| \leq \epsi$, the spectrum of $Q_0(\xi)$ consists of the eigenvalues of $P_0(\xi)$ outside $[E_\star-\eta,E_\star+\eta]$. The spectral theorem implies that if $\delta_0$ is small enough, under \eqref{eq:3e},
\begin{equations}\label{eq:2m}
Q_0(\xi)-\lambda : \ \WW(\xi)^\perp \cap H^2_\xi \rightarrow \WW(\xi)^\perp \cap L^2_\xi \text{ is invertible and } \\
 \big( Q_0(\xi) - \lambda \big)^{-1} = \OO_{\WWW(\xi)^\perp}(1).
\end{equations}

 Let $J(\xi) : \VVV(\xi)^\perp \rightarrow \WWW(\xi)^\perp$ obtained by orthogonally projecting an element $u \in \VVV(\xi)^\perp \subset L^2_\xi$ to $\WWW(\xi)^\perp$. We set
\begin{equation}
E_\delta(\xi,\lambda) \de J(\xi)^* \cdot \big( Q_0(\xi) - \lambda \big)^{-1} \cdot J(\xi) \ : \ \VVV(\xi)^\perp \rightarrow \VVV(\xi)^\perp.
\end{equation}
The first estimate of \eqref{eq:2l} is satisfied because of \eqref{eq:2m}. We want to check the second estimate. Observe that
\begin{equations}\label{eq:2k}
E_\delta(\xi,\lambda) \cdot \big(D_\delta(\xi) - \lambda\big)  
= 
E_\delta(\xi,\lambda) \cdot \pi_{\VVV(\xi)^\perp} \big( P_0(\xi) - \lambda + \delta W \big)
\\
=
J(\xi)^* \cdot \big( Q_0(\xi)-\lambda \big)^{-1} \cdot J(\xi) \cdot \pi_{\VVV(\xi)^\perp} \big( P_0(\xi) - \lambda \big) + \OO_{\VVV(\xi)}(\delta).
\end{equations}
Above, $\pi_{\VVV(\xi)^\perp} : L^2_\xi \rightarrow L^2_\xi$ is the orthogonal projection from $L^2_\xi$ to $\VVV(\xi)$, also seen as  an operator $L^2_\xi \mapsto \VVV(\xi)^\perp$. We introduce similarly $\pi_{\WW(\xi)^\perp}$.  Then
\begin{equations}\label{eq:5n}
J(\xi) \cdot \pi_{\VVV(\xi)^\perp}  =\pi_{\WWW(\xi)^\perp} \cdot \left( \Id - \pi_{\VVV(\xi)} \right)  = \pi_{\WWW(\xi)^\perp} - \left( \Id - \pi_{\WWW(\xi)} \right) \cdot \pi_{\VVV(\xi)} 
\\
= \pi_{\WWW(\xi)^\perp}  - \left(\pi_{\VVV(\xi)} - \pi_{\WWW(\xi)} \right) \cdot \pi_{\VVV(\xi)}.
\end{equations}
The individual eigenvectors associated to the eigenvalues of $P_0(\xi)$ in $[E_\star-\eta,E_\star+\eta]$ do not depend smoothly on $\xi$ but the projector $\pi_{\WWW(\xi)}$ depends smoothly on $\xi$ -- see \cite[\S VII1.3, Theorem 1.7]{Ka}.  Since $\VVV(\xi_\star) = \WWW(\xi_\star)$, this implies $\pi_{\VVV(\xi)} - \pi_{\WWW(\xi)} = \OO_{L^2_\xi}(\xi-\xi_\star)$. We deduce that
\begin{equation}\label{eq:2j}
J(\xi) \cdot \pi_{\VVV(\xi)^\perp}  = \pi_{\WWW(\xi)^\perp} + \OO_{\WWW(\xi)^\perp}(\xi-\xi_\star).
\end{equation}

We combine \eqref{eq:2k} and \eqref{eq:2j} to obtain
\begin{equations}
E_\delta(\xi,\lambda) \cdot \big(D_\delta(\xi) - \lambda\big) = J(\xi)^* \cdot \big( Q_0(\xi)-\lambda \big)^{-1} \cdot \pi_{\WWW(\xi)^\perp} (P_0(\xi) - \lambda)  + \OO_{L^2_\xi}(\delta)
\\
= J(\xi)^* \pi_{\WWW(\xi)^\perp}  + \OO_{\VVV(\xi)^\perp}\big(\delta+|\xi-\xi_\star|\big).
\end{equations}
The operator $J(\xi)^*$ takes an element in $\WWW(\xi)^\perp$ and project it to $\VVV(\xi)^\perp$. By the same argument as \eqref{eq:5n} and \eqref{eq:2j} (inverting $\VVV(\xi)$ and $\WWW(\xi)$), 
\begin{equation}
J(\xi)^* \pi_{\WWW(\xi)^\perp} = \pi_{\VVV(\xi)^\perp} + O_{\VVV(\xi)^\perp}(\xi-\xi_\star).
\end{equation}
We conclude that the second estimate of \eqref{eq:2l} is satisfied. It follows that $D_\delta(\xi)-\lambda : \VVV(\xi)^\perp \rightarrow \VVV(\xi)^\perp$ is invertible under \eqref{eq:3e}.

4. We now study $A_\delta(\xi)-\lambda$. This operator acts on the two-dimensional space $\VVV(\xi)$; its matrix in the basis $\{ e^{i\blr{\xi-\xi_\star,x}} \phi_1, e^{i\blr{\xi-\xi_\star,x}} \phi_2 \}$ is
\begin{equation}\label{eq:3c}
\matrice{ \blr{ e^{i\blr{\xi-\xi_\star,x}} \phi_1, (P_\delta(\xi)-\lambda)e^{i\blr{\xi-\xi_\star,x}} \phi_1 } &  \blr{ e^{i\blr{\xi-\xi_\star,x}} \phi_1, (P_\delta(\xi)-\lambda)e^{i\blr{\xi-\xi_\star,x}} \phi_2 } \\
\blr{ e^{i\blr{\xi-\xi_\star,x}} \phi_2 (P_\delta(\xi)-\lambda)e^{i\blr{\xi-\xi_\star,x}} \phi_1 } & \blr{ e^{i\blr{\xi-\xi_\star,x}} \phi_2, (P_\delta(\xi)-\lambda)e^{i\blr{\xi-\xi_\star,x}} \phi_2 }}.
\end{equation}
We observe that
\begin{equations}
e^{-i\blr{\xi-\xi_\star,x}} (P_\delta(\xi)-\lambda) e^{i\blr{\xi-\xi_\star,x}} = P_\delta(\xi_\star) + [e^{-i\blr{\xi-\xi_\star,x}},-\Delta] e^{i\blr{\xi-\xi_\star,x}}
\\
= P_\delta(\xi_\star) -\lambda + [\Delta, e^{-i\blr{\xi-\xi_\star,x}}] e^{i\blr{\xi-\xi_\star,x}} 
= P_\delta(\xi_\star) -\lambda  +  2 ((\xi-\xi_\star) \cdot D_x) - |\xi-\xi_\star|^2.
\end{equations}
Therefore the matrix elements in \eqref{eq:3c} are given by
\begin{equations}
\blr{\phi_j, \big( P_\delta(\xi)-\lambda\big) \phi_k} = \blr{\phi_j, \left( P_\delta(\xi_\star)  +  2  (\xi-\xi_\star) \cdot D_x - \lambda  - |\xi-\xi_\star|^2 \right) \phi_k}
\\
= \big( E_\star - |\xi-\xi_\star|^2 \big) \delta_{jk} + \blr{\phi_j, \big( \delta W  +  2 (\xi-\xi_\star) \cdot D_x\big) \phi_k}.
\end{equations}
We deduce from Lemma \ref{lem:1i} and \ref{lem:1h} that the matrix \eqref{eq:3c} is equal to $M_\delta(\xi) + O(\xi-\xi_\star)^2$. Using a Neumann series argument based on \eqref{eq:3c}, when \eqref{eq:3e} holds, $A_\delta(\xi)-\lambda$ is invertible; and
\begin{equation}\label{eq:4q}
\big( A_\delta(\xi)-\lambda \big)^{-1} =  \Pi_0(\xi)^* \cdot \big( M_\delta(\xi)-\lambda \big)^{-1} \cdot \Pi_0(\xi) + \OO_{\VVV(\xi)}\left( \dfrac{|\xi-\xi_\star|^2}{\delta^2 + |\xi-\xi_\star|^2} \right).
\end{equation}
Because of Lemma \ref{lem:1f}, we also observe that
\begin{equation}\label{eq:4r}
\big( A_\delta(\xi)-\lambda \big)^{-1}  = \OO_{\VVV(\xi)} \left((\delta+|\xi-\xi_\star|)^{-1} \right).
\end{equation}

5. Schur's lemma allows to invert block-by-block operators of the form \eqref{eq:5p} under certain conditions on the blocks; see \cite[Lemma 4.1]{DFW} for the version needed here. We need to verify that:
\begin{equations}\label{eq:4o}
A_\delta(\xi) - \lambda : \ \VVV(\xi) \rightarrow \VVV(\xi) \ \text{ is invertible; } \\
D_\delta(\xi) - \lambda - C_\delta(\xi) \cdot \big( A_\delta(\xi)-\lambda \big)^{-1} \cdot B_\delta(\xi) : \VVV(\xi)^\perp \rightarrow \VVV(\xi)^\perp \ \text{ is invertible}.
\end{equations}
The first statement holds because of Step 4. Regarding the second statement, we observe that because of \eqref{eq:3b} and \eqref{eq:4r},
\begin{equation}
C_\delta(\xi) \cdot \big( A_\delta(\xi)-\lambda \big)^{-1} \cdot B_\delta(\xi) = \OO_{\VVV(\xi)^\perp}\big(\delta + |\xi-\xi_\star| \big) = \OO_{\VVV(\xi)^\perp}\big(\delta^{1/3}\big).
\end{equation}
Because of Step 3, $D_\delta(\xi) - \lambda$ is invertible and its inverse is $\OO_{\VVV(\xi)^\perp}(1)$. Therefore a Neumann series argument shows that the second statement in \eqref{eq:4o} holds. It also shows that the inverse is $\OO_{\VVV(\xi)^\perp}(1)$.

We apply Schur's lemma -- see \cite[Lemma 4.1]{DFW}. From \eqref{eq:5p}, we obtain that $P_\delta(\xi) - \lambda : H^2_\xi \rightarrow L^2_\xi$ is invertible when \eqref{eq:3e} holds; and moreover
\begin{equation}
\big(P_\delta(\xi) - \lambda\big)^{-1} = \matrice{(A_\delta(\xi)-\lambda)^{-1} & 0 \\ 0 & 0} + \OO_{L^2_\xi}(1).
\end{equation}
Using \eqref{eq:4q} and the projector \eqref{eq:1r}, we deduce that
\begin{equation}\label{eq:5q}
\big( P_\delta(\xi) - \lambda \big)^{-1} = \Pi_0(\xi)^* \cdot \big( M_\delta(\xi)-\lambda \big)^{-1} \cdot \Pi_0(\xi) + \OO_{L^2_\xi}(1).
\end{equation}
The error term in \eqref{eq:5q} improves automatically to $\OO_{L^2_\xi \rightarrow H^2_\xi}(1)$ because of elliptic regularity -- see the argument at the end of the proof of Lemma \ref{lem:1x}. 
\end{proof}

\section{The bulk resolvent along the edge}\label{sec:6}

Let $v \in \Lambda$ be the direction of an edge. We define accordingly $v', \ k, \ k'$ and $\ell$ -- see \S\ref{sec:6.1}.
For $\zeta \in \R$, we set
\begin{equation}
L^2[\zeta] \de \left\{ u \in L^2_\loc \big( \R^2,\C \big), \ u(x+v) = e^{i\zeta} u(x), \ \int_{\R^2/\Z v} |u(x)|^2 dx   < \infty\right\}.
\end{equation}
Let $P_\delta[\zeta]$ be the operator formally equal to  $P_\delta$ acts but acting on $L^2[\zeta]$. We are interested in the resolvent of $P_\delta[\zeta]$ for $\delta$ small and $\zeta$ near $\zeta_\star = \blr{\xi_\star, v}$. We recall:
\begin{equations}\label{eq:5r}
\Pi : L^2\big(\R^2/\Z v, \C^2\big) \rightarrow L^2\big(\R,\C^2\big), \ \ \ \ \big(\Pi f\big)(t) \de \int_0^1 f(sv + tv') ds; 
\\
\Pi^* : L^2\big(\R, \C^2\big) \rightarrow L^2\big(\R^2/\Z v,\C^2\big), \ \ \ \ \big(\Pi^* g\big)(x) \de g\big(\blr{k',x}\big);
\\
\UU_\delta : L^2\big(\R,\C^2\big) \rightarrow L^2\big(\R,\C^2\big), \ \ \ \ \big(\UU_\delta f\big)(t) \de f(\delta t).
\end{equations}
Let $\Di_\pm(\mu) : H^1(\R,\C^2) \rightarrow L^2(\R,\C^2)$ be the formal limits of $\Di(\mu)$ as $t \rightarrow \pm \infty$: 
\begin{equation}
\Di_\pm(\mu) \de \matrice{\var_\star & \nu_\star k'   \\ \ove{\nu_\star k' } & -\var_\star}D_t + \mu \matrice{0 & \nu_\star \ell  \\ \ove{\nu_\star \ell} & 0 } \pm \matrice{\var_\star & 0 \\ 0 & -\var_\star}.
\end{equation}
The main result of this section relates the resolvent of $P_{\pm \delta}[\zeta]$ at $E_\star+\delta z$ to that of $\Di_\pm(\mu)$ at $z$ for small enough $\delta$.  The assumptions $\operatorname{(H1)}$ -- $\operatorname{(H3)}$ were defined in \S\ref{sec:1.5}.
 
 \begin{theorem}\label{thm:1} Assume that the assumptions $\operatorname{(H1)}$ -- $\operatorname{(H3)}$ hold and fix $\tmu > 0$ and $\te \in (0,1)$. There exists $\delta_0 > 0$ such that if
 \begin{equations}
 \delta \in (0,\delta_0), \ \ \mu \in (-\tmu,\tmu), \ \ z \in \Dd\left( 0, \te \sqrt{ \var_F^2 + \mu^2 \cdot \nu_F^2 | \ell|^2 } \right), \\
 \zeta = \zeta_\star + \delta \mu, \ \  \ \ \lambda = E_\star + \delta z
 \end{equations}
 then the operators $P_{\pm\delta}[\zeta]- \lambda : H^2[\zeta] \rightarrow L^2[\zeta]$ are invertible. Furthermore, 
 \begin{equations}
 \left.\big(P_{\pm\delta}[\zeta]- \lambda \big)^{-1} = S_{\pm\delta}(\mu,z) + \OO_{L^2[\zeta]}\left(\delta^{-1/3}\right) \right., \\ (k' \cdot D_x) \big(P_{\pm\delta}[\zeta]- \lambda \big)^{-1} = S_{\pm\delta}^D(\mu,z) + \OO_{L^2[\zeta]}\left(\delta^{-1/3}\right),
\\
\text{where:} \ \ \ \  S_{\pm\delta}(\mu,z) \de \dfrac{1}{\delta} \cdot \matrice{\phi_1 \\ \phi_2}^\top e^{i  \mu \delta \blr{\ell,x} } \Pi^* \cdot  \UU_\delta \big( \Di_\pm(\mu) - z \big)^{-1} \UU_\delta^{-1} \cdot \Pi e^{-i  \mu \delta \blr{\ell,x} } \ove{\matrice{\phi_1 \\ \phi_2}},
 \\
  S_{\pm\delta}^D(\mu,z) \de \dfrac{1}{\delta} \cdot \matrice{(k' \cdot D_x)\phi_1 \\ (k' \cdot D_x)\phi_2}^\top e^{i  \mu \delta \blr{\ell,x} } \Pi^* \cdot  \UU_\delta \big( \Di_\pm(\mu) - z \big)^{-1} \UU_\delta^{-1} \cdot \Pi e^{-i  \mu \delta \blr{\ell,x} } \ove{\matrice{\phi_1 \\ \phi_2}}.
 \end{equations}
 \end{theorem}

\subsection{Strategy} We first observe that it suffices to prove Theorem \ref{thm:1} for $P_\delta[\zeta]$. Indeed, to go from $P_\delta[\zeta]$ to $P_{-\delta}[\zeta]$ we simply replace $W$ with $-W$. The only parameter to change is $\var_\star$, which becomes $-\var_\star$. This simply transforms $\Di_+(\mu)$ to $\Di_-(\mu)$.

To prove Theorem \ref{thm:1}, we decompose $P_\delta[\zeta]$ fiberwise using the operators $P_\delta(\xi)$ (formally equal to $P_\delta$ acts but acting on $L^2_\xi$). Specifically:
\begin{equation}
P_\delta[\zeta] = \dfrac{1}{2\pi} \int_{\R/(2\pi \Z)}^\oplus P_\delta(\zeta k + \tau k') \cdot d\tau = \dfrac{1}{2\pi} \int_{[0,2\pi]}^\oplus P_\delta(\zeta k + \tau k') \cdot d\tau.
\end{equation}
When $P_\delta[\zeta]-\lambda$ is invertible, we are interested in the resolvent
\begin{equation}\label{eq:3z}
\big(P_\delta[\zeta]-\lambda\big)^{-1} = \dfrac{1}{2\pi} \int_{[0,2\pi]}^\oplus \big(P_\delta(\zeta k + \tau k')-\lambda \big)^{-1} d\tau.
\end{equation}
The fiber resolvents $\big(P_\delta(\zeta k + \tau k')-\lambda \big)^{-1}$ were studied in \S\ref{sec:5}. We first show that if $\zeta k + \tau k'$ satisfies $\rho(\zeta k + \tau k') \geq \delta^{1/3}$ then this quasimomentum does not contribute significantly to the resolvent $\big(P_\delta[\zeta]-\lambda\big)^{-1}$.

Then we study the contributions from quasimomenta $\zeta k + \tau k'$ at distance at most $\delta^{1/3}$ from $\xi_\star$. The Dirac operator $\Di_+(\mu)$ emerges from a rescaled direct integration of the dominant rank-two matrix exhibited in Lemma \ref{lem:1g}.

\subsection{Reduction to $\zeta k + \tau k'$ near $\xi_\star$} We start the proof of Theorem \ref{thm:1}. Below $\te \in (0,1)$ and $\tmu > 0$ are fixed numbers. Let $n$ be the integer such that
\begin{equation}
\lr{\xi_\star,v'} \in [2\pi n, 2\pi n + 2\pi).
\end{equation}
Write $\xi = \zeta k + \tau k'$, $\tau \in [2\pi n, 2\pi n + 2\pi)$ and introduce
\begin{equation}
I \de \left\{ \tau \in [2\pi n, 2\pi n + 2\pi) : |\xi-\xi_\star| \leq \delta^{1/3} \right\}, \ \ \ \ I^c \de  [2\pi n, 2\pi n + 2\pi) \setminus I.
\end{equation}
Observe that $\rho(\xi) =  \delta |\mu \ell|$. In particular for $\delta$ small enough $\rho(\xi)$ is smaller than the threshold $\epsi_0$ given by Lemma \ref{lem:1x}. That lemma yields
\begin{equations}
(P_\delta[\zeta] - \lambda)^{-1}   = \dfrac{1}{2\pi} \int_{\tau \in I}^\oplus \big(P_\delta(\zeta k + \tau k')-\lambda\big)^{-1} d\tau +\dfrac{1}{2\pi} \int_{\tau \in I^c}^\oplus \big(P_\delta(\zeta k + \tau k')-\lambda\big)^{-1} d\tau 
\end{equations}
\begin{equations}\label{eq:1o}
= \dfrac{1}{2\pi} \int_{\tau \in I}^\oplus \big(P_\delta(\zeta k + \tau k')-\lambda\big)^{-1} d\tau +  \OO_{L^2[\zeta] \rightarrow H^2[\zeta]}\left(\delta^{-1/3}\right).
\end{equations}

We would like to apply Lemma \ref{lem:1g} to the leading term of \eqref{eq:1o}. We check the assumptions: we must verify that when $\lamdba$ belongs to the range allowed in Theorem \ref{thm:1}, $\lambda$ belongs to the range required by Lemma \ref{lem:1g}. This is equivalent to
\begin{equation}\label{eq:1k}
\Dd\left(E_\star, \te \delta \sqrt{\var_F^2 + \mu^2 \cdot \nu_F^2 | \ell|^2 }\right) \subset \Dd\left(E_\star, \te  \sqrt{\var_F^2 \delta^2 + \nu_F^2 \cdot |\xi-\xi_\star|^2 }\right).
\end{equation}
To check that \eqref{eq:1k} holds,
we observe that
\begin{equations}\label{eq:1p}
|\xi-\xi_\star|^2 = |k'|^2 (\tau-\tau_\star)^2 + \mu^2 \delta^2 |\ell|^2 , \\
\tau_\star \de \blr{\xi_\star, v'} - \mu \delta \dfrac{ \blr{k,k'}}{|k'|^2}, \ \ \ \ \ell \de k-\dfrac{ \blr{k,k'}}{|k'|^2} k'.
\end{equations}
This implies
\begin{equation}
\te \delta \sqrt{\var_F^2 + \mu^2 \cdot \nu_F^2 | \ell|^2 } = \te \sqrt{\var_F^2 \delta^2 + \mu^2  \nu_F^2 \cdot  |\ell|^2 \delta^2 } \leq  \te \sqrt{\var_F^2 \delta^2 + \nu_F^2 \cdot |\xi-\xi_\star|^2 }.
\end{equation}
Therefore we can apply Lemma \ref{lem:1g} to the leading term of \eqref{eq:1o}. It shows that 
\begin{equations}\label{eq:2s}
P_\delta[\zeta] = T_\delta[\zeta] + \OO_{L^2[\zeta] \rightarrow H^2[\zeta]}\left(\delta^{-1/3}\right), 
\\
T_\delta[\zeta] \de \dfrac{1}{2\pi } \int_{\tau \in I}^\oplus \Pi_0(\zeta k + \tau k')^* \cdot \big(M_\delta(\zeta k + \delta \tau k') - \lambda \big)^{-1} \cdot \Pi_0(\zeta k + \tau k') d\tau.
\end{equations}

Because of \eqref{eq:1p}, $\tau_\star = \lr{\xi_\star,v'} + O(\delta)$. From \S\ref{sec:6.1} and the definition of $n$, $\lr{\xi_\star,v'} \in \{2\pi n+ 2\pi/3,2\pi n+ 4\pi/3\}$. Hence $\tau_\star$ is in the interior of $I$ for $\delta$ sufficiently small. It follows that $I$ is an interval centered at $\tau_\star$: 
\begin{equation}\label{eq:1q}
I = \big[\tau_\star-\delta \cdot \az_\delta,  \tau_\star+\delta \cdot \az_\delta\big], \ \ \ \ \az_\delta \de \dfrac{\sqrt{\delta^{2/3} - \mu^2 \cdot \nu_F^2 | \ell|^2 \delta^2}}{|k'|\delta } = \dfrac{\delta^{-2/3}}{|k'|} + O\left(\delta^{2/3}\right).
\end{equation}
We make the substitution $\tau \mapsto \tau_\star + \delta \tau$. The vector $\zeta k + \tau k'$ becomes $\zeta k + (\tau_\star + \delta \tau)k' = \xi_\star + \delta (\tau k'+\mu \ell)$; the interval $I$ becomes $[-\az_\delta, \az_\delta]$; $d\tau$ becomes $\delta d\tau$ and 
\begin{equations}
M_\delta\big(\zeta k + \delta (\tau k'+\mu \ell) \big) = E_\star+ \delta \matrice{\var_\star & \nu_\star   (\tau k'+\mu \ell) \\ \ove{\nu_\star  (\tau k'+\mu \ell)} & -\var_\star}, 
\\
\Big(  M_\delta\big(\zeta k + \delta (\tau k'+\mu \ell) \big) - \lambda \Big)^{-1} = \dfrac{1}{\delta} \matrice{\var_\star -z & \nu_\star   (\tau k'+\mu \ell) \\ \ove{\nu_\star  (\tau k'+\mu \ell)} & -\var_\star-z}^{-1}, \ \ \ \ z \de \dfrac{\lambda-E_\star}{\delta}.
\end{equations}
We deduce that $T_\delta[\zeta]$ equals
\begin{equations}
\dfrac{1}{2\pi } \int_{|\tau| < \az_\delta}^\oplus \Pi_0\big(\xi_\star + \delta (\tau k'+\mu \ell) \big)^* \cdot \matrice{\var_\star -z & \nu_\star   (\tau k'+\mu \ell) \\ \ove{\nu_\star  (\tau k'+\mu \ell)} & -\var_\star-z}^{-1} \cdot \Pi_0\big(\xi_\star + \delta (\tau k'+\mu \ell) \big) \cdot d\tau.
\end{equations}

Thanks to the definition \eqref{eq:1q} of $\Pi_0$,
 \begin{equation}
 \Pi_0\big(\xi_\star + \delta (\tau k'+\mu \ell)\big)u = \matrice{\blr{e^{i\delta  \blr{\tau k'+\mu \ell,x}}\phi_1,u} \\ \blr{e^{i\delta \blr{\tau k'+\mu \ell, x}}\phi_2,u}}.
\end{equation}
We conclude that the operator $T_\delta[\zeta]$ has kernel
\begin{equation}\label{eq:3t}
\dfrac{1}{2\pi } \matrice{\phi_1(x) \\ \phi_2(x)}^\top \cdot \int_{|\tau| \leq \az_\delta} \matrice{\var_\star -z & \nu_\star   (\tau k'+\mu \ell) \\ \ove{\nu_\star  (\tau k'+\mu \ell)} & -\var_\star-z}^{-1} e^{i\delta \blr{\tau k'+\mu \ell,x-y}} d\tau \cdot  \ove{\matrice{\phi_1(y) \\ \phi_2(y)}}.
\end{equation}

\subsection{Kernel identities and proof of Theorem \ref{thm:1}} Recall that $\Pi, \ \Pi^*$ and $\UU_\delta$ are defined in \eqref{eq:5r}. 

\begin{lem}\label{lem:1j} There exists $C > 0$ such that for every $\delta \in (0,1)$, the following holds. Let $\Psi \in L^\infty\big(\R, M_2(\C)\big)$, possibly depending on $\delta$, and $A_\Psi$ be the operator with kernel
\begin{equation}
(x,y) \mapsto 
\matrice{\phi_1(x) \\ \phi_2(x)}^\top \cdot \dfrac{1}{2\pi}\int_\R \Psi(\tau) e^{i\delta \blr{\tau k'+\mu \ell,x-y}} d\tau \cdot  \ove{\matrice{\phi_1(y) \\ \phi_2(y)}}.
\end{equation}
Then $A_\Psi$ is bounded on $L^2[\zeta]$ with $\| A_\Psi \|_{L^2[\zeta]} \leq C\delta^{-1}|\Psi|_\infty$; and
\begin{equation}\label{eq:2q}
A_\Psi = \dfrac{1}{\delta} \cdot \matrice{\phi_1 \\ \phi_2}^\top e^{i  \mu \delta \blr{\ell,x} } \Pi^* \cdot  \UU_\delta \Psi(D_t) \UU_\delta^{-1} \cdot \Pi e^{-i  \mu \delta \blr{\ell,x} } \ove{\matrice{\phi_1 \\ \phi_2}}.
\end{equation}
If in addition $\tau \cdot \Psi \in L^\infty\big(\R,M_2(\C)\big)$
then $(k' \cdot D_x) A_\Psi$ is bounded on $L^2[\zeta]$ with 
\begin{equations}
\| (k' \cdot D_x)  A_\Psi \|_{L^2[\zeta]} \leq C \delta^{-1} |\Psi|_\infty + C |\tau \cdot \Psi|_\infty; \ \ \ \ and: \ \ \ \ (k' \cdot D_x) A_\Psi = 
\\
\dfrac{1}{\delta} \cdot \matrice{(k' \cdot D_x) \phi_1 \\ (k' \cdot D_x) \phi_2}^\top e^{i  \mu \delta \blr{\ell,x} } \Pi^* \cdot  \UU_\delta \Psi(D_t) \UU_\delta^{-1} \cdot \Pi e^{-i  \mu \delta \blr{\ell,x} } \ove{\matrice{\phi_1 \\ \phi_2}} + \OO_{L^2[\zeta]}\big(|\blr{\tau} \cdot \psi|_\infty \big).
\end{equations}
\end{lem}

\begin{proof} 1. We first note that the operator $\delta^{-1} \cdot \UU_\delta \Psi(D_t) \UU_\delta^{-1}$ has kernel
\begin{equation}\label{eq:2p}
(t,t') \in \R \times \R^2 \mapsto \dfrac{1}{2\pi} \int_{\R} e^{i \delta \tau (t-t')} \Psi(\tau) \cdot d\tau.
\end{equation}
Let $\delta_0$ denote the Dirac mass. We claim that the operator $\Pi$ has kernel:
\begin{equation}\label{eq:2o}
(t',y) \in \R \times \R^2/(\Z v) \mapsto \delta_0\big( \lr{k', y} - t' \big).
\end{equation}
Fix $f \in C_0^\infty(\R^2/\Z v, \C^2)$. The integral
\begin{equation}
\int_{\R^2/\Z v} \delta_0\big( \lr{k', y}-t' \big)  f(y) dy
\end{equation}
is well-defined. We perform the substitution $y \mapsto (\blr{k,y}, \blr{k',y})$; the inverse substitution is $(s,t) \mapsto sv + tv'$; the Jacobian determinant is $dy = \Det[v,v'] \cdot dsdt$. Since $v, v'$ are related to $v_1, v_2$ by \eqref{eq:4a} and $\Det[v_1,v_2] = 1$ because of \eqref{eq:4b}, $\Det[v,v'] = 1$. The above integral becomes 
\begin{equations}
\int_{\R^2/\Z e_1} \delta_0(t - t')   f(sv + tv') ds dt
=
\int_{\R/\Z}  f(sv + t'v') ds.
\end{equations}
We recover the formula \eqref{eq:5r} for $\Pi f$. From \eqref{eq:2o}, we deduce that the kernel of $\Pi^*$ is
\begin{equation}\label{eq:2n}
(x,t) \in \R^2/ \Z v \times \R \mapsto \delta_0\big( \lr{k',x} - t \big).
\end{equation}

To obtain \eqref{eq:2p}, we compose the kernels \eqref{eq:2o}, \eqref{eq:2p} and \eqref{eq:2n}. This forces $t$ to be $\lr{k',x}$ and $t'$ to be $\lr{k',y}$. Hence the operator $\delta^{-1} \cdot \Pi^* \cdot  \UU_\delta \Psi(D_t) \UU_\delta^{-1} \cdot \Pi$ has kernel
\begin{equation}
(x,y) \mapsto \dfrac{1}{2\pi} \int_{\R} e^{i \delta \tau  \blr{k', x-y}}  \Psi(\tau) \cdot d\tau.
\end{equation}
This implies \eqref{eq:2q}. 

2. We prove the $L^2[\zeta]$-bound. The operator $\Pi$ maps $L^2(\R^2/\Z v,\C)$ to $L^2(\R,\C^2)$, independently of $\delta$. Its adjoint maps $L^2(\R,\C^2)$ to $L^2(\R^2/ \Z v,\C)$, independently of $\delta$. The dilations $\UU_\delta$ and $\UU_\delta^{-1}$ map $L^2(\R, \C^2)$ to itself, with bounds $\delta^{-1/2}$ and $\delta^{1/2}$, respectively. The operator $\Psi(D_t)$ is a Fourier multiplier, hence it maps  $L^2(\R, \C^2)$ to itself, with bound $|\Psi|_\infty$. Combining all these bounds together we get
\begin{equation}
\| A_\Psi \|_{L^2[\zeta]} \leq C \delta^{-1} |\Psi|_\infty.
\end{equation}

3. We observe that the operator $(k' \cdot D_x) A_\Psi$ has kernel
\begin{equations}
\matrice{(k' \cdot D_x) \phi_1(x) \\ (k' \cdot D_x) \phi_2(x)}^\top \cdot \dfrac{1}{2\pi}\int_\R \Psi(\tau) e^{i\delta \blr{\tau k'+\mu \ell,x-y}} d\tau \cdot  \ove{\matrice{\phi_1(y) \\ \phi_2(y)}}
\\
+ 
\matrice{\phi_1(x) \\ \phi_2(x)}^\top \cdot \dfrac{1}{2\pi}\int_\R \Psi(\tau) \cdot  \tau  \delta |k'|^2 e^{it\delta \blr{\tau k'+\mu \ell,x-y}} d\tau \cdot  \ove{\matrice{\phi_1(y) \\ \phi_2(y)}}.
\end{equations}
Above, we used that $\ell \cdot k' = 0$. These two terms are kernels of operators studied in Steps 1 and 2. The first one has $L^2[\zeta]$-operator norm controlled by $C\delta^{-1} |\Psi|_\infty$ and the second one by $C |\tau \cdot \Psi|_\infty$. This concludes the proof.
\end{proof}

\begin{lem}\label{lem:1k} Let $\tvar \in (0,\var_F)$. There exists $C > 0$ such that for any $z \in \Dd(0,\tvar)$, the following holds. Let $\Psi_0 : \R \rightarrow M_2(\C)$ be given by
\begin{equation}\label{eq:2r}
\Psi_0(\tau) \de \matrice{\var_\star -z & \nu_\star   (\tau k'+\mu \ell) \\ \ove{\nu_\star  (\tau k'+\mu \ell)} & -\var_\star -z}^{-1}.
 \end{equation}
Then $\tau \cdot \Psi_0 \in L^\infty\big(\R,M_2(\C)\big)$ and for every $a \geq 0$,
\begin{equation}\label{eq:2t}
\sup_{|\tau| \geq a} \big\|\Psi_0(\tau) \big|_{\C^2} \leq C a^{-1}, \ \ \ \ \sup_{|\tau| \geq a} \big\|\tau \Psi_0(\tau) \big\|_{\C^2} \leq C.
\end{equation}
\end{lem}

To prove Lemma \ref{lem:1k}, it suffice to observe that $\Psi_0$ is an invertible, first-order matrix-valued symbol. Therefore its inverse belongs to $S^{-1}$. This yields the bounds \eqref{eq:2t}. 
Let  $\Di_+(\mu) : H^1(\R,\C^2) \rightarrow L^2(\R,\C^2)$ be the Dirac operator defined by
\begin{equation}\label{eq:1j}
\Di_+(\zeta) \de \matrice{\var_\star & \nu_\star k'   \\ \ove{\nu_\star k' } & -\var_\star}D_t + \mu \matrice{0 & \nu_\star \ell  \\ \ove{\nu_\star \ell} & 0 } + \matrice{\var_\star & 0 \\ 0 & -\var_\star}.
\end{equation}
We are now ready to prove Theorem \ref{thm:1}.

\begin{proof}[Proof of Theorem \ref{thm:1}] 
1. Because of \eqref{eq:2s}, it suffices to prove Theorem \ref{thm:1} when $P_\delta[\zeta]$ is replaced by $T_\delta[\zeta]$. 
 We first apply Lemma \ref{lem:1j} with $\Psi_0$ given by \eqref{eq:2r}. It shows that
\begin{equations}
A_{\Psi_0} \de \dfrac{1}{\delta} \cdot \matrice{\phi_1 \\ \phi_2}^\top e^{i  \mu \delta \blr{\ell,x} } \Pi^* \cdot  \UU_\delta \big( \Di_+(\mu) - z \big)^{-1} \UU_\delta^{-1} \cdot \Pi e^{-i  \mu \delta \blr{\ell,x} } \ove{\matrice{\phi_1 \\ \phi_2}} \ \ \ \ \text{ has kernel: } 
\\
(x,y) \mapsto \matrice{\phi_1(x) \\ \phi_2(x)}^\top \cdot\dfrac{1}{2\pi} \int_\R \Psi_0(\tau) e^{i\delta \blr{\tau k'+\mu \ell,x-y}} d\tau \cdot  \ove{\matrice{\phi_1(y) \\ \phi_2(y)}}.
\end{equations}

2. We now apply Lemma \ref{lem:1j} with $\Psi_1(\tau) \de \1_{\R\setminus [-\az_\delta,\az_\delta]}(\tau) \cdot \Psi_0(\tau)$ (recall that $\az_\delta = |k'|^{-1}\delta^{-2/3} + O(\delta^{2/3})$ was defined in \eqref{eq:1q}). It shows that $A_{\Psi_1}$ has kernel
\begin{equation}
\matrice{\phi_1(x) \\ \phi_2(x)}^\top \cdot \dfrac{1}{2\pi} \int_{|\tau| \geq \az_\delta} \Psi_0(\tau) e^{i\delta \blr{\tau k'+\mu \ell,x-y}} d\tau \cdot  \ove{\matrice{\phi_1(y) \\ \phi_2(y)}}.
\end{equation}
Thanks to the bounds of Lemma \ref{lem:1k}, $A_{\Psi_1} = \OO_{L^2[\zeta]}(\delta^{-1} \az_\delta^{-1}) = \OO_{L^2[\zeta]}(\delta^{-1/3})$.

3. When we subtract the kernel of $A_{\Psi_1}$ to the kernel of $A_{\Psi_0}$, we get the kernel of $T_\delta[\zeta]$, see \eqref{eq:3t}. This shows that $T_\delta[\zeta] = A_{\Psi_0} - A_{\Psi_1}$. Hence $T_\delta[\zeta] = $
 \begin{equation}
\dfrac{1}{\delta} \cdot \matrice{\phi_1 \\ \phi_2}^\top e^{i  \mu \delta \blr{\ell,x} } \Pi^* \cdot  \UU_\delta \big( \Di_+(\mu) - z \big)^{-1} \UU_\delta^{-1} \cdot \Pi e^{-i  \mu \delta \blr{\ell,x} } \ove{\matrice{\phi_1 \\ \phi_2}} + \OO_{L^2[\zeta]}\left(\delta^{-1/3}\right).
\end{equation}

4. Lemma \ref{lem:1j} and the bounds of Lemma \ref{lem:1k} imply that $(k' \cdot D_x) A_{\Psi_0}$ is equal to 
\begin{equation}
\dfrac{1}{\delta} \cdot \matrice{(k' \cdot D_x) \phi_1 \\ (k' \cdot D_x) \phi_2}^\top e^{i  \mu \delta \blr{\ell,x} } \Pi^* \cdot  \UU_\delta \big( \Di_+(\mu) - z \big)^{-1} \UU_\delta^{-1} \cdot \Pi e^{-i  \mu \delta \blr{\ell,x} }\ove{\matrice{\phi_1 \\ \phi_2}} + \OO_{L^2[\zeta]}(1).
\end{equation}
It also imply that $(k' \cdot D_x) A_{\Psi_1} = \OO_{L^2[\zeta]}(\delta^{-1/3})$. We conclude that $(k' \cdot D_x) T_\delta[\zeta]$ equals
\begin{equation}
\dfrac{1}{ \delta} \cdot \matrice{(k' \cdot D_x) \phi_1 \\ (k' \cdot D_x) \phi_2}^\top e^{i  \mu \delta \blr{\ell,x} } \Pi^* \cdot  \UU_\delta \big( \Di_+(\mu) - z \big)^{-1} \UU_\delta^{-1} \cdot \Pi e^{-i  \mu \delta \blr{\ell,x} }\ove{\matrice{\phi_1 \\ \phi_2}} + \OO_{L^2[\zeta]}\left(\delta^{-1/3}\right).
\end{equation}
This completes the proof of the theorem. \end{proof}

\section{The resolvent of the edge operator} Recall that $\kappa$ is a domain wall function -- see \eqref{eq:3v} -- and introduce the operator
\begin{equation}
\PP_\delta = -\Delta + V + \delta \cdot  \kappa_\delta \cdot  W, \ \ \ \ \kappa_\delta(x) = \kappa(\delta\blr{k',x}).
\end{equation}
We denote by $\PP_\delta[\zeta]$ the operator formally equal to $\PP_\delta$ but acting on $L^2[\zeta]$.
In this section we prove Theorem \ref{thm:2}: we connect the resolvent of $\PP_\delta[\zeta]$ to that of the Dirac operator $\Di(\mu)$ emerging from the multiscale analysis of \cite{FLTW3}:
\begin{equation}
\Di(\mu) =  \matrice{0 & \nu_\star k' \\ \ove{\nu_\star k'} & 0} D_t + \mu \matrice{0 & \nu_\star \ell \\ \ove{\nu_\star \ell} & 0} + \var_\star \matrice{1 & 0 \\ 0 & -1} \kappa.
\end{equation}
The strategy is as follows. We first prove a formula for $(\PP_\delta[\zeta]-\lambda)^{-1}$ in terms of the asymptotic operators $(P_{\pm \delta}[\zeta]-\lambda)^{-1}$. We then apply Theorem \ref{thm:1} to exhibit the leading order term in this formula.

We use a cyclicity argument to simplify this leading order term. An averaging effect emerges as the driving phenomena connecting $\PP_\delta[\zeta]$ to $\Di(\mu)$. This yields Theorem \ref{thm:2}.

\subsection{Parametrix}\label{sec:6.0} We first construct a parametrix for $\PP_\delta[\zeta]-\lambda$. Introduce
\begin{equation}\label{eq:2e}
\QQQ_\delta(\zeta,\lambda) \de  \sum_\pm \chi_{\pm,\delta} \cdot (P_{\pm \delta}[\zeta]-\lambda)^{-1}, \ \ \ \ \chi_\pm \de \dfrac{1 \pm \kappa}{2}.
\end{equation}
This operator is well defined -- and depends holomorphically in $\lambda$ -- as long as $\lambda \notin \Sigma_{L^2[\zeta]}(P_\delta[\zeta])$. Formally speaking, it behaves asymptotically like $(\PP_\delta[\zeta]-\lambda)^{-1}$.

A calculation similar to \cite[\S5.2]{DFW} yields
\begin{equations}
\big( \PP_\delta[\zeta]-\lambda \big) \QQQ_\delta(\zeta,\lambda) =  \Id + \KKK_\delta(\zeta,\lambda), \ \ \ \ \text{where } \ \KKK_\delta(\zeta,\lambda) =
\\
 \dfrac{\delta}{2} \left(2(D_t\kappa)_\delta \cdot (k' \cdot D_x) + \delta |k'|^2 (D_t^2\kappa)_\delta + (\kappa_\delta^2-1) W \right)\left((P_\delta[\zeta]-\lambda)^{-1} - (P_{-\delta}[\zeta]-\lambda)^{-1}\right).
\end{equations}
This identity shows that if $\Id + \KKK_\delta(\zeta,\lambda)$ is invertible then $\PP_\delta[\zeta] - \lambda$ is invertible. When this holds, $(\PP_\delta[\zeta] - \lambda)^{-1}$ has an expression in terms of $\QQQ_\delta(\zeta,\lambda)$ and $\KKK_\delta(\zeta,\lambda)$:
\begin{equation}
\big( \PP_\delta[\zeta] - \lambda \big)^{-1} = \QQQ_\delta(\zeta,\lambda) \cdot \big( \Id + \KKK_\delta(\zeta,\lambda) \big)^{-1}.
\end{equation}

The operators $\QQQ_\delta(\zeta,\lambda)$ and $\KKK_\delta(\zeta,\lambda)$ have expressions in terms of $(P_{\pm \delta}[\zeta]-\lambda)^{-1}$. An application of Theorem \ref{thm:1} estimates $\QQQ_\delta(\zeta,\lambda)$ and $\KKK_\delta(\zeta,\lambda)$, assuming
\begin{equations}\label{eq:3l}
\delta \in (0,\delta_0), \ \ 
\mu \in (-\tmu,\tmu), \ \ z \in \Dd\Big(0,\sqrt{\var_F^2 +  \mu^2\cdot \nu_F^2 |\ell|^2} \Big),\\
\lambda = E_\star + \delta z, \ \ \zeta = \zeta_\star + \delta \mu.
\end{equations}
We introduce the operator
\begin{equations}
R_0(\mu,z) : L^2\big(\R,\C^2\big) \rightarrow H^2\big(\R,\C^2\big),
\\
R_0(\mu,z)
\de \left( \Di_+(\mu)^2 - z^2 \right)^{-1} = \left( \nu_F^2|k'|^2 D_t^2 + \mu^2| \nu_\star\ell|^2 + \var_F^2 - z^2 \right)^{-1}.
\end{equations}
It is well defined when $z$ is away from the spectrum of $\Di_\pm(\mu)$ -- in particular when
\begin{equation}
z \in \Dd\Big(0,\sqrt{\var_F^2 +  \mu^2\cdot \nu_F^2 |\ell|^2} \Big).
\end{equation}

\begin{lem}\label{lem:1d} Let $\tmu > 0$, $\te \in (0,1)$. There exists $\delta_0 > 0$ such that under the assumptions of Theorem \ref{thm:1},  $\KKK_\delta(\zeta,\lambda)$ and $\QQQ_\delta(\zeta,\lambda)$ admit the expansions
\begin{equations}
\KKK_\delta(\zeta,\lambda) = \KK_\delta(\mu,z) + \OO_{L^2[\zeta]}\left(\delta^{2/3}\right), \ \ \ \  \QQQ_\delta(\zeta,\lambda) = \QQ_\delta(\mu,z) + \OO_{L^2[\zeta]}\left(\delta^{-1/3}\right).
\end{equations}
Above, $\KK_\delta(\mu,z)$ is equal to
\begin{equations}
\var_\star \left(2(D_t\kappa)_\delta \matrice{ k'\cdot D_x \phi_1 \\ k'\cdot D_x \phi_2}^\top + (\kappa_\delta^2-1) W \matrice{ \phi_1 \\ \phi_2}^\top \right) 
  e^{-i  \mu \delta \blr{\ell,x} }  \Pi^*  
 \UU_\delta 
  \bst \cdot R_0(\mu,z) \cdot \UU_\delta^{-1} \Pi  e^{i  \mu \delta \blr{\ell,x} } \ove{\matrice{\phi_1 \\ \phi_2}}
  \\
  \text{and} \ \ \ \ 
\QQ_\delta(\mu,z) \de \dfrac{1}{\delta } \cdot  \matrice{ \phi_1 \\ \phi_2}^\top  e^{-i  \mu \delta \blr{\ell,x} }  \Pi^*  \UU_\delta \cdot (\Di(\mu)+z) \cdot R_0(\mu,z) \cdot \UU_\delta^{-1} \Pi  e^{i  \mu \delta \blr{\ell,x} } \ove{\matrice{\phi_1 \\ \phi_2}}. \ \ \ \ \ \ \ \ \ 
\end{equations}
\end{lem}

The proof is a calculation using the relation between $\QQQ_\delta(\zeta,\lambda)$ and  $\KKK_\delta(\zeta,\lambda)$ with the edge resolvents $(P_{\pm \delta}[\zeta] -\lambda)^{-1}$; and the expansions of these resolvents provided by Theorem \ref{thm:1}. We deferred it to Appendix \ref{app:3}.

\subsection{Weak convergence} We  are interested in the eigenvalues of $\PP_\delta[\zeta]$. We previously studied eigenvalue problems in seemingly different situations \cite{Dr2,Dr4,Dr3} as well as in a one-dimensional analog \cite{DFW}. The proofs of these results rely on a \textit{cyclicity} principle: if $A$ and $B$ are two matrices then the non-zero eigenvalues of $AB$ and $BA$ are equal (together with their multiplicity).

Although the leading order terms $\KK_\delta(\mu,z)$ and $\QQ_\delta(\mu,z)$ have complicated expressions, they exhibit a structure favorable to apply the cyclicity principle. This will provide a simple formula for the product
\begin{equation}
\QQ_\delta(\mu,z) \cdot \big(\Id + \KK_\delta(\mu,z)\big)^{-1}
\end{equation}
and complete the proof of Theorem \ref{thm:2}. 

A preliminary step is the computation of a weak limit that arises when permuting factors in $\KK_\delta(\mu,z)$: the operator $L^2(\R, \C^2) \rightarrow L^2(\R, \C^2)$ given by
\begin{equation}
\var_\star \UU_\delta^{-1} \Pi  e^{i  \mu \delta \blr{\ell,x} } \ove{\matrice{\phi_1 \\ \phi_2}} \cdot \left(2(D_t\kappa)_\delta \cdot  \matrice{ k'\cdot D_x \phi_1 \\ k'\cdot D_x \phi_2}^\top + (\kappa_\delta^2-1) W \matrice{ \phi_1 \\ \phi_2}^\top \right) 
 \cdot  e^{-i  \mu \delta \blr{\ell,x} }  \Pi^* \UU_\delta \bst
\end{equation}
\begin{equation}\label{eq:1b} 
= \var_\star \UU_\delta^{-1} \cdot \Pi  \ove{\matrice{\phi_1 \\ \phi_2}} \left(2(D_t\kappa)_\delta \cdot  \matrice{ k'\cdot D_x \phi_1 \\ k'\cdot D_x \phi_2}^\top + (\kappa_\delta^2-1) W \matrice{ \phi_1 \\ \phi_2}^\top \right) \bst
 \Pi^* \cdot \UU_\delta .
\end{equation}

\begin{lem}\label{lem:1n} The operator \eqref{eq:1b} is a multiplication operator by a function $\UUU^\delta : \R \rightarrow M_2(\C)$ with two-scale structure:
\begin{equation}\label{eq:1c}
\UUU^\delta(t) = \UUU\left(\dfrac{t}{\delta}, t  \right), \ \ \ \ \UUU \in C^\infty_0\big(\R/\Z\times \R , M_2(\C)\big).
\end{equation} 
The function $\UUU^\delta$ converges weakly to 
\begin{equation}\label{eq:2w}
\UUU^0 \in C^\infty_0\big( \R, M_2(\C) \big), \ \ \ \ 
\UUU^0(t) \de \var_F^2 \left(\kappa(t)^2-1\right) + \var_\star\matrice{0 & -\nu_\star k' \\ \ove{\nu_\star k'} & 0} (D_t\kappa)(t).
\end{equation}
Finally, if $\UUU^\delta - \UUU^0$ is seen as a multiplication operator from $H^1$ to $H^{-1}$,
\begin{equation}\label{eq:1l}
\UUU^\delta - \UUU^0  = \OO_{H^1 \rightarrow H^{-1}}(\delta).
\end{equation}
\end{lem}

\begin{proof} 1. We set
\begin{equations}
F(x,t) \de \var_\star \ove{\matrice{\phi_1(x) \\ \phi_2(x)}} \left((D_t\kappa)(t) \cdot  \matrice{ 2 k'\cdot D_x \phi_1(x) \\ 2 k'\cdot D_x \phi_2(x)}^\top + \big(\kappa(t)^2-1\big) W(x) \matrice{ \phi_1(x) \\ \phi_2(x)}^\top \right) \bst,
\\
F^\delta(x) \de F\big(x,\delta \blr{k',x}\big).
\end{equations}
Fix $g \in C_0^\infty\big(\R,\C^2\big)$. The action of  the operator \eqref{eq:1b} on $g$ is given by:
\begin{equations}
\big(\UU_\delta^{-1} \cdot \Pi F^\delta \Pi^*\cdot  \UU_\delta g \big)(t) 
= \int_0^1 F^\delta\left(sv+\dfrac{t}{\delta}v'\right)   g\left(\blr{k', \delta \left( sv+\dfrac{t}{\delta}v'\right)}\right) ds
\\
 = \int_0^1 F\left(sv+\dfrac{t}{\delta}v',t\right) g(t) ds = \int_0^1 F\left(sv+\dfrac{t}{\delta}v',t\right) ds \cdot g(t).
\end{equations}
Therefore \eqref{eq:1b} is the multiplication operator by
\begin{equation}
\UUU^\delta(t) \de \int_0^1 F\left(sv+\dfrac{t}{\delta}v',t\right) ds.
\end{equation}
Note that $F$ is $\Lambda$-periodic in the first two variables and compactly supported in the last variable. Therefore $\UUU^\delta$ has the two-scale structure \eqref{eq:1c}:
\begin{equation}\label{eq:5s}
\UUU^\delta(t) = \UUU\left( \dfrac{t}{\delta},t \right), \ \ \ \ \UUU(\tau,t) \de \int_0^1 F\left(sv+\tau v',t\right) ds.
\end{equation}

2. The function $\UUU$ is periodic in the first variable and compactly supported in the second one. Therefore the weak limit of $\UUU^\delta$ is 
\begin{equation}\label{eq:5t}
\UUU^0(t) \de \int_0^1 \UUU(\tau,t) d\tau = \int_0^1 \int_0^1 F\left(sv+\tau v',t\right) d\tau ds = \int_{\Ll} F\left(x,t\right) dx.
\end{equation}
In the last inequality, we changed variables: $sv+\tau v'$ became $sv_1+\tau v_2$ (with Jacobian equal to $1$); hence $[0,1]^2$ became $\Ll$, the fundamental cell of $\R^2/\Lambda$ given in \eqref{eq:4c}. Going back to the definition of $F$, we end up with:
\begin{equations}
\UUU^0(t) =  \left(\matrice{\var_F^2 & 0 \\ 0 & -\var_F^2} \left(\kappa(t)^2-1\right) + \var_\star\matrice{0 & \nu_\star k' \\ \ove{\nu_\star k'} & 0} (D_t\kappa)(t)\right) \bst
\\
=
 \var_F^2 \left(\kappa(t)^2-1\right) + \var_\star\matrice{0 & -\nu_\star k' \\ \ove{\nu_\star k'} & 0} (D_t\kappa)(t).
\end{equations}

3. We show the quantitative estimate \eqref{eq:1l}. Since $\UUU^\delta$ and $\UUU^0$ are functions on $\R$,
\begin{equation}
\left\| \UUU^\delta - \UUU^0 \right\|_{H^1 \rightarrow H^{-1}} \leq C \left| \UUU^\delta - \UUU^0\right|_{H^{-1}}.
\end{equation}
See e.g. \cite[Lemma 2.1]{Dr4}. Recall that $\UUU^\delta$ is related to  $\UUU$ via \eqref{eq:5s}. The function $\UUU$ is periodic in the first variable and compactly supported in the second variable. We write a Fourier decomposition of $\UUU$:
\begin{equation}
\UUU(t,\tau) = \sum_{m \in \Z} b_m(t) e^{2i
\pi m \tau}, \ \ \ \ b_m(t) \de \int_0^1 e^{-2i
\pi m \tau'} \UUU(t,\tau') d\tau'.
\end{equation}
Because of \eqref{eq:5s} and \eqref{eq:5t},
\begin{equation}
\UUU^\delta(t) - \UUU^0(t) = \sum_{m \neq 0} b_m(t) e^{2i
\pi m t/\delta}.
\end{equation}
In other words, $\UUU^\delta - \UUU^0$ has a highly oscillatory structure. The coefficients $b_m$ are smooth functions of $t$. Their Sobolev norms decay rapidly since $\UUU$ depends smoothly on $\tau$.  We can then conclude as in the proof of \cite[Lemma 3.1]{Dr3}
\end{proof}

The function $\UUU^0$ is an effective potential that arises as the homogenized limit of $\UUU^\delta$. It appears in the Dirac operator $\Di(\mu)$. Indeed, a computation shows that
\begin{equations}
\Di(\mu)^2 =
\nu_F^2 |k'|^2 D_t^2 + \mu^2 \cdot \nu_F^2 | \ell|^2 + \var_F^2 \kappa^2
+  \var_\star \matrice{0 & -\nu_\star k' \\ \ove{\nu_\star k'} & 0} (D_t \kappa).
\end{equations}
Because of \eqref{eq:2w}, we deduce that
\begin{equation}\label{eq:2x}
\Di(\mu)^2 = \nu_F^2 |k'|^2 D_t^2 + \mu^2 \cdot \nu_F^2 | \ell|^2 + \var_F^2 + \UUU^0.
\end{equation}
We will apply this identity in the next section.

\subsection{A cyclicity argument} The next result is stated abstractly. It relies on the cyclicity principle.

\begin{lem}\label{lem:1l} Let $A, B, C, D,E$ be bounded operators: 
\begin{equations}
A : H^1\big(\R,\C^2\big) \rightarrow L^2[\zeta], \ \ \ B : L^2[\zeta] \rightarrow L^2\big(\R,\C^2\big), 
\\ 
C : L^2\big(\R,\C^2\big) \rightarrow L^2[\zeta], \ \ \ \ D : L^2\big(\R,\C^2\big) \rightarrow H^1\big(\R,\C^2\big), 
\\
E : L^2\big(\R,\C^2\big) \rightarrow L^2(\R,\C^2\big).
\end{equations}
Assume that for some $M \geq 1$,
\begin{itemize}
\item[(a)] The operator norms of $A, B, C, D, E$  are bounded by $M$.
\item[(b)] The operator $\Id + D E D : L^2(\R, \C^2) \rightarrow L^2(\R,\C^2)$ is invertible and
\begin{equation}
\left\| (\Id + DED)^{-1} \right\|_{L^2(\R,\C^2)} \leq M.
\end{equation}
\item[(c)] The following estimate holds: 
\begin{equation}
\epsi \de \big\| D (BC-E) D\big\|_{L^2(\R,\C^2)} \leq \dfrac{1}{2M}.
\end{equation}
\end{itemize}
Then the operator $\Id + C D^2 B : L^2[\zeta] \rightarrow L^2[\zeta]$ is invertible;
\begin{equations}\label{eq:5v}
\left\| \big(\Id + C D^2 B\big)^{-1} \right\|_{L^2[\zeta]} \leq 3M^5;  \ \ \ \ \text{ and }
\\
\left\| AD^2B \cdot (\Id + C D^2 B)^{-1} -  AD \cdot (\Id + DED)^{-1} \cdot DB \right\|_{L^2[\zeta]} \leq  2M^6 \epsi.
\end{equations}
\end{lem}

\begin{proof} Below we use $L^2$ and $H^1$ to denote $L^2(\R,\C^2)$ and $H^1(\R,\C^2)$. 

1. Recall that $\Id + C D^2 B = \Id + CD \cdot DB : L^2[\zeta] \rightarrow L^2[\zeta]$ is invertible if and only if $\Id + DB \cdot CD : L^2 \rightarrow L^2$ is invertible. In this case, the inverses are related via
\begin{equation}\label{eq:2y}
\big(\Id + C D^2 B\big)^{-1} =  \Id - CD \big(\Id + D B \cdot C D\big)^{-1} DB.
\end{equation}
Because of (b), $\Id + DED$ is invertible and
\begin{equations}\label{eq:1h}
\Id + DB \cdot CD = \Id + DED + D(BC-E)D
\\ 
= (\Id + DED) \cdot \left(\Id + (\Id + DED)^{-1} \cdot D(BC-E)D \right).
\end{equations}
Because of both (b) and (c),
\begin{equation}
\left\| (\Id + DED)^{-1} \cdot D(BC-E)D \right\|_{L^2} \leq \dfrac{1}{2}.
\end{equation}
This implies that $\Id + (\Id + DED)^{-1} \cdot D(BC-E)D$ is invertible by a Neumann series; the inverse has operator norm controlled by $2$. Thanks to \eqref{eq:1h}, $\Id + DB CD$ is invertible and the inverse has norm controlled by $2M$. Hence $\Id + C D^2 B$ is invertible. Thanks to \eqref{eq:2y} and (a), 
\begin{equation}
\left\| (\Id + C D^2 B)^{-1}\right\|_{L^2[\zeta]} \leq 1 + M^2 \cdot 2M \cdot M^2 \leq 3M^5.
\end{equation}
This proves the first estimate of \eqref{eq:5v}.

2. Observe that
\begin{equation}
(\Id + D BC D)^{-1} - (\Id + D E D)^{-1} = (\Id + D E D)^{-1}  \cdot D(E-BC)D \cdot (\Id + D BC D)^{-1}.
\end{equation}
Because of the bounds proved in Step 1 and of (c),
\begin{equation}\label{eq:5w}
\left\| (\Id + D BC D)^{-1} - (\Id + D E D)^{-1} \right\|_{L^2} \leq 2M^2 \epsi.
\end{equation}

We write
\begin{equations}
AD^2B \cdot \big(\Id + C D^2 B\big)^{-1} = AD^2B \cdot \left( \Id - CD \big(\Id + D BC D\big)^{-1} DB \right)
\\
= AD \cdot DB - AD \cdot D B CD \big(\Id + D BC D\big)^{-1} \cdot DB 
\\
= AD \cdot \left( \Id - D B CD \big(\Id + D BC D \big)^{-1} \right) \cdot DB = AD \cdot \big(\Id + D BC D \big)^{-1} \cdot DB.
\end{equations}
The operator norms of $AD : L^2 \rightarrow L^2[\zeta]$ and $DB :  L^2[\zeta] \rightarrow H^1$ are each bounded by $M^2$ because of (a). We deduce from \eqref{eq:5w} that
\begin{equation}\label{eq:1e}
\left\| AD^2B \cdot \big(\Id + C D^2 B\big)^{-1} - AD \cdot \big(\Id + D E D \big)^{-1} \cdot DB \right\|_{L^2[\zeta]} \leq 2M^6 \epsi.
\end{equation}
This proves the second estimate of \eqref{eq:5w}, hence completes the proof of the lemma.\end{proof}

We would like to apply Lemma \ref{lem:1l} with the choices:
\begin{equations}
A \de \delta^{1/2} \cdot  \matrice{ \phi_1 \\ \phi_2}^\top  e^{-i  \mu \delta \blr{\ell,x} }  \Pi^* \cdot  \UU_\delta \big( \Di(\mu)+z \big), \ \ \ \  
B \de \dfrac{1}{\delta^{1/2}} \cdot \UU_\delta^{-1} \Pi  e^{i  \mu \delta \blr{\ell,x} } \ove{\matrice{\phi_1 \\ \phi_2}}
\\
C \de  \delta^{1/2} \var_\star \left(2(D_t\kappa)_\delta \cdot  \matrice{ k'\cdot D_x \phi_1 \\ k'\cdot D_x \phi_2}^\top + (\kappa_\delta^2-1) W \matrice{ \phi_1 \\ \phi_2}^\top \right) 
 \cdot  e^{-i  \mu \delta \blr{\ell,x} }  \Pi^*  
 \cdot \ \UU_\delta
  \bst
\end{equations}
\begin{equations}\label{eq:2u}
  D = \left(\Di_+(\mu)^2-z^2\right)^{-1/2} = R_0(\mu,z)^{1/2}, \ \ \ E = \UUU^0.
\end{equations}
These operators are manufactured so that
\begin{equation}\label{eq:2v}
\QQ_\delta(\mu,z) = \dfrac{1}{\delta} AD^2B, \ \ \ \ \KK_\delta(\mu,z) = CD^2B,
\end{equation}
see the formula of Lemma \ref{lem:1d}. Recall that $\UUU^\delta, \UUU^0$ were defined in Lemma \ref{lem:1n} and observe that $BC = \UUU^\delta \rightarrow E = \UU^0$ (for the operator norm $H^1 \rightarrow H^{-1}$). This provides the favorable setting needed for the use of  the cyclicity argument (Lemma \ref{lem:1l}).

The definition of $D$ requires some precisions. Let $\vp(\w) = (\w^2-z^2)^{-1/2}$, where the squareroot is holomorphic on $\C \setminus (-\infty,0]$. If $|z| < \sqrt{\var_F^2 +  \mu^2\cdot \nu_F^2 |\ell|^2}$ and
\begin{equations}
\w \in \Sigma_{L^2}\big(\Di_+(\mu)\big) = \R \setminus \Big[ - \sqrt{\var_F^2 +  \mu^2\cdot \nu_F^2 |\ell|^2}, \sqrt{\var_F^2 +  \mu^2\cdot \nu_F^2 |\ell|^2} \Big].
\end{equations}
then $\Re\big(\w^2 - z^2 \big) > 0$. Hence $\vp$ is well defined on the spectrum of $\Di_+(\mu)$. This allows to define $D = \vp\big( \Di_+(\mu) \big)$ using the spectral theorem.

\begin{lem}\label{lem:1m} Fix $\epsilon_1 > 0, \ \tmu \in \R$. There exists $\delta_0 > 0$ such that if
\begin{equations}\label{eq:1g}
\delta \in (0,\delta_0), \ \ \mu \in (-\tmu,\tmu), \\ 
z^2 \in \Dd\big(0,\var_F^2 + \mu^2 \cdot \nu_F^2 | \ell|^2\big), \ \ \dist\left( \Sigma_{L^2}\big(\Di(\mu)^2 \big),z^2 \right) \geq \epsilon_1^2
\end{equations} 
then $\big(\Id + D E D \big)^{-1}$ and $\big(\Id + C D^2 B\big)^{-1}$ are invertible on $L^2[\zeta]$. Moreover,
\begin{equations}\label{eq:6s}
 AD^2B \cdot \big(\Id + C D^2 B\big)^{-1} = AD \cdot \big(\Id + D E D \big)^{-1} \cdot DB + \OO_{L^2[\zeta]}(\delta),
 \\
 \big(\Id + C D^2 B\big)^{-1} = \OO_{L^2[\zeta]}(1).
\end{equations} 
\end{lem}

\begin{proof} Below we use $L^2$ and $H^1$ to denote $L^2(\R,\C^2)$ and $H^1(\R,\C^2)$. The equation \eqref{eq:6s} is a consequence of Lemma \ref{lem:1l}, assuming that the assumptions (a), (b) and (c) hold with a constant $M$ independent of $\delta, \mu, z$ satisfying \eqref{eq:1g}.

1. We first verify (a). We observe that that the only singular dependence of $A, B, C$ and $E$ is in $\delta$. It arises only in the operators $\delta^{1/2} \UU_\delta$ and $\delta^{-1/2} \UU_\delta^{-1}$, which are both isometries on $L^2$. In addition,
\begin{equation}
\dist\left( \Sigma_{L^2}\big(\Di(\mu)^2 \big),z^2 \right) \geq \epsilon_1^2 \ \ \Rightarrow  \ \ \dist\left( \Sigma_{L^2}\big(\Di_+(\mu)^2 \big),z^2 \right) \geq \epsilon_1^2.
\end{equation}
Therefore $D$ is controlled by $\epsi_1^{-2}$; and (a) holds 
independently of $\delta, \mu, z$ satisfying \eqref{eq:1g}.

2. From the definition \eqref{eq:2u} of $D$, $D$ is invertible. Therefore we can write
\begin{equation}
\Id + D E D = D \big( D^{-2}+E \big)  D.
\end{equation}
Moreover, thanks to \eqref{eq:2x},
\begin{equation}\label{eq:1x}
D^{-2}+E = \Di(\mu)^2-z^2.
\end{equation}
When $z$ satisfies the condition of \eqref{eq:1g}, the operator $\Di(\mu)^2-z^2$ is invertible. This comes with the bound:
\begin{equation}
\left\| \left(\Di(\mu)^2-z^2\right)^{-1} \right\|_{L^2} \leq \dfrac{1}{\epsilon_1^2}.
\end{equation}
This is independent of $\delta$: (b) holds.

3. The operator $D$ maps $L^2$ to $H^1$ and $H^{-1}$ to $L^2$, with uniformly bounded norm in $\mu, z$ satisfying \eqref{eq:1g}. Therefore (c) holds -- possibly after shrinking $\delta_0$ -- if
\begin{equation}\label{eq:1m}
\| BC - E \|_{H^1 \rightarrow H^{-1}} = O(\delta).
\end{equation}
We observe that $BC = \UUU^\delta$ and recall that $E = \UUU^0$. Therefore \eqref{eq:1m} reduces to the quantitative estimate 
\eqref{eq:1l} proved in Lemma \ref{lem:1n}.

4. Because of Steps 1, 2 and 3, we can apply Lemma \ref{lem:1l}. It yields Lemma \ref{lem:1m}.\end{proof}

According to this lemma, when \eqref{eq:1g} holds, $\Id + \KK_\delta(\mu,z)$ is invertible. Hence
\begin{equation}
\QQ_\delta(\mu,z) \cdot \big(\Id + \KK_\delta(\mu,z)\big)^{-1}
\end{equation}
is well-defined. Thanks to \eqref{eq:2v}, 
\begin{equations}
\QQ_\delta(\mu,z) \cdot \big(\Id + \KK_\delta(\mu,z)\big)^{-1} = \dfrac{1}{\delta} AD \cdot \big(\Id + D E D \big)^{-1} \cdot DB + \OO_{L^2[\zeta]}(1)
\\
= \dfrac{1}{\delta} AD \cdot D^{-1} \left( D^{-2}+E \right)^{-1} D^{-1} \cdot DB + \OO_{L^2[\zeta]}(1)
= \dfrac{1}{\delta} A \cdot \left( D^{-2}+E \right)^{-1} \cdot B + \OO_{L^2[\zeta]}(1).
\end{equations}
We now plug in the formula \eqref{eq:2u} for $A,B,C,D,E$, and we use the relation \eqref{eq:1x}. This yields:
\begin{equations}
\QQ_\delta(\mu,z) \cdot \big(\Id + \KK_\delta(\mu,z)\big)^{-1} 
\\
= \dfrac{1}{\delta} \cdot  \matrice{ \phi_1 \\ \phi_2}^\top  e^{-i  \mu \delta \blr{\ell,x} }  \Pi^*  \UU_\delta  \cdot (\Di(\mu)+z) \cdot \left(\Di(\mu)^2-z^2 \right)^{-1} \cdot 
 \UU_\delta^{-1} \Pi  e^{i  \mu \delta \blr{\ell,x} } \ove{\matrice{\phi_1 \\ \phi_2}} + \OO_{L^2[\zeta]}(1)
\end{equations}
\begin{equations}\label{eq:3s}
 = \dfrac{1}{\delta} \cdot  \matrice{ \phi_1 \\ \phi_2}^\top  e^{-i  \mu \delta \blr{\ell,x} } \cdot \Pi^*   \UU_\delta \cdot \left(\Di(\mu)-z\right)^{-1} \cdot 
 \UU_\delta^{-1} \Pi \cdot e^{i  \mu \delta \blr{\ell,x} } \ove{\matrice{\phi_1 \\ \phi_2}} + \OO_{L^2[\zeta]}(1).
\end{equations}
We are now ready to prove Theorem \ref{thm:2}.

\begin{proof}[Proof of Theorem \ref{thm:2}] 1. Fix $\epsilon > 0$ and $\mu_\sharp > 0$. Fix $z \in \C$ satisfying 
\begin{equation}\label{eq:5u}
z \in \Dd\Big(0,\sqrt{\var_F^2 + \mu^2\cdot \nu_F^2 | \ell|^2} -\dfrac{\epsilon}{3} \Big); \ \ \ \ \dist\Big( \Sigma_{L^2}\big(\Di(\mu)^2 \big), z^2 \Big) \geq \dfrac{\epsilon^2}{9}.
\end{equation}
Note that this does not quite correspond to the assumptions of Theorem \ref{thm:2}. Instead it is a stronger form of the assumptions of Lemma \ref{lem:1m} with $\epsilon_1 = \epsilon/3$. The equation \eqref{eq:5u} implies that $\Id + \KK_\delta(\mu,z)$ is invertible. Apply Lemma \ref{lem:1d} with
\begin{equation}
\te = 1 - \dfrac{\epsilon}{3\sqrt{\var_F^2 + \mu^2 \cdot \nu_F^2 | \ell|^2}}.
\end{equation}
It implies that
\begin{equation}
\Id + \KKK_\delta(\zeta,\lambda) = \Id + \KK_\delta(\mu,z) + \OO_{L^2[\zeta]}(\delta^{2/3}).
\end{equation} 
Hence -- after possibly shrinking $\delta_0$ -- the operator $\Id + \KKK_\delta(\zeta,\lambda)$ is invertible. The inverses of $\Id + \KKK_\delta(\zeta,\lambda)$ and $\Id + \KK_\delta(\mu,z)$ are related via
\begin{equation}
\big(\Id + \KKK_\delta(\zeta,\lambda)\big)^{-1} =  \big(\Id + \KK_\delta(\mu,z)\big)^{-1} + \OO_{L^2[\zeta]}\left(\delta^{2/3}\right),
\end{equation}
because $\big(\Id + \KK_\delta(\mu,z)\big)^{-1} = (\Id + CD^2B)^{-1}$ is uniformly bounded by Lemma \ref{lem:1m}. It follows that under \eqref{eq:5u}, $\PP_\delta[\zeta] - \lambda$ is invertible and
\begin{equation}
\big( \PP_\delta[\zeta] - \lambda \big)^{-1} = \QQQ_\delta(\zeta,\lambda) \cdot \big(\Id + \KKK_\delta(\zeta,\lambda)\big)^{-1}.
\end{equation}

2. Observe that $\QQ_\delta(\mu,z) = \OO_{L^2[\zeta]}(\delta^{-1})$: this comes the relation between  $\QQQ_\delta(\zeta,\lambda)$ and  $\QQ_\delta(\mu,z)$ provided by Lemma \ref{lem:1d}. We deduce that  $\PP_\delta[\zeta] - \lambda$ is invertible and
\begin{equations}
\big( \PP_\delta[\zeta] - \lambda \big)^{-1} = \QQ_\delta(\mu,z) \cdot \big( \Id + \KK_\delta(\mu,z) \big)^{-1} + \OO_{L^2[\zeta]}\left(\delta^{-1/3}\right).
\end{equations}
Thanks to \eqref{eq:3s}, this simplifies to $\big( \PP_\delta[\zeta] - \lambda \big)^{-1} =$
\begin{equations}\label{eq:3q}
 \dfrac{1}{\delta} \cdot  \matrice{ \phi_1 \\ \phi_2}^\top  e^{-i  \mu \delta \blr{\ell,x} }  \Pi^* \UU_\delta \cdot \left(\Di(\mu)-z\right)^{-1} \cdot 
 \UU_\delta^{-1} \Pi  e^{i  \mu \delta \blr{\ell,x} } \ove{\matrice{\phi_1 \\ \phi_2}} + \OO_{L^2[\zeta]}\left(\delta^{-1/3}\right).
\end{equations} 

\vspace*{-.5cm}

\begin{center}
\begin{figure}
\caption{The top blue area represents the domain of validity of \eqref{eq:3q} provided by Step 1 and 2. The bottom blue area represents the domain of validity of \eqref{eq:3q} as specified by Theorem \ref{thm:2}. In Step 3 we prove that \eqref{eq:3q} holds near $\var = -\var_0^\mu$, at the price of increasing $\epsilon/3$ to $\epsilon$.
}\label{fig:11}
{\includegraphics{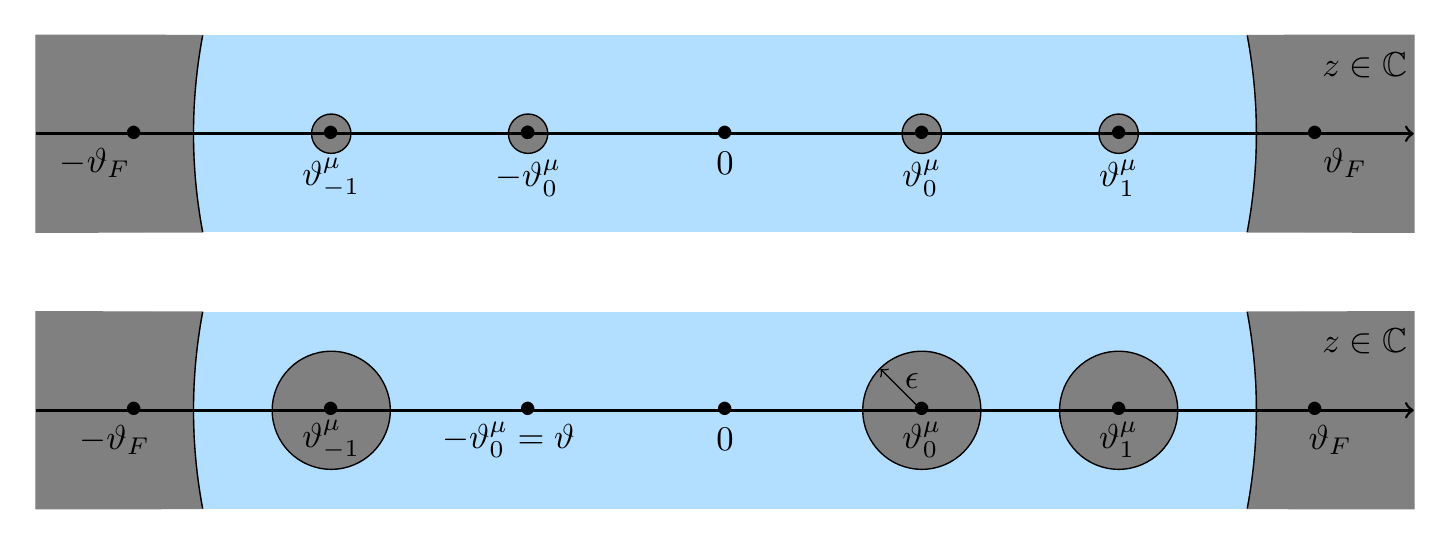}}
\end{figure}
\end{center}

3. The estimate \eqref{eq:3q} is valid as long as $z$ satisfies \eqref{eq:5u}. There is a subtlety here: \eqref{eq:5u} does not quite correspond to the assumption of Theorem \ref{thm:2}. To conclude the proof, we must show that \eqref{eq:5u} is unnecessarily strong. In other words, we assume in these final steps that
\begin{equation}
z \in \Dd\Big(0,\sqrt{\var_F^2 + \mu^2\cdot \nu_F^2 | \ell|^2} -\epsilon \Big), \  \
 \dist\Big( \Sigma_{L^2}\big(\Di(\mu) \big), z \Big) \geq \epsilon; \ \ \ \ \dist\Big( \Sigma_{L^2}\big(\Di(\mu)^2 \big), z^2 \Big) < \dfrac{\epsilon^2}{9}.
\end{equation}
The third condition implies
\begin{equation}
\dist\big(\Sigma_{L^2}\big(\Di(\mu)\big),z\big) < \dfrac{\epsilon}{3} \ \text{ or } \ \dist\big(\Sigma_{L^2}\big(-\Di(\mu)\big),z\big) < \dfrac{\epsilon}{3}.
\end{equation} 
From the second condition, we deduce that $\dist\big(\Sigma_{L^2}\big(-\Di(\mu)\big),z\big) < \epsilon/3$. The spectra of $\Di(\mu)$ and $-\Di(\mu)$ differ by at most one eigenvalue: 
\begin{equation}\label{eq:6v}
\Sigma_{L^2}\big(-\Di(\mu)\big) \setminus \Sigma_{L^2}\left(\Di(\mu)\right) \subset \{\var\}, \ \ \ \var \de - \mu \cdot \nu_F |\ell| \cdot \sgn(\var_\star),
\end{equation}
see Lemma \ref{lem:1a}. Hence, $z$ must belong to $\Dd(\var,\epsilon/3)$. 

4. Because of Step 3, the proof of Theorem \ref{thm:2} is complete if we can show that \eqref{eq:3q} holds when
\begin{equation}
z \in \Dd\Big(0,\sqrt{\var_F^2 + \mu^2\cdot \nu_F^2 | \ell|^2} -\epsilon \Big), \  \
 \dist\Big( \Sigma_{L^2}\big(\Di(\mu) \big), z \Big) \geq \epsilon; \ \ \ \ z \in \Dd\left( \var, \dfrac{\epsilon}{3} \right).
\end{equation}
Fix $s \in \p\Dd(\var,\epsilon/3)$. Then, $|z-s| < 2\epsilon/3$. This implies that
\begin{equation}
s \in \Dd\Big(0,\sqrt{\var_F^2 + \mu^2\cdot \nu_F^2 | \ell|^2} -\dfrac{\epsilon}{3} \Big), \  \
 \dist\Big( \Sigma_{L^2}\big(\Di(\mu) \big), s \Big) \geq \dfrac{\epsilon}{3}, \ \ |\var-s| = \dfrac{\epsilon}{3}.
\end{equation}
Because of \eqref{eq:6v}, $s$ satisfies
\begin{equation}
\dist\Big( \Sigma_{L^2}\big(\Di(\mu) \big), s \Big) \geq \dfrac{\epsilon}{3}, \ \ \dist\Big( \Sigma_{L^2}\big(-\Di(\mu) \big), s \Big) = \dfrac{\epsilon}{3} \ \ \Rightarrow \ \ \dist\big( \Sigma_{L^2}\big(\Di(\mu)^2 \big), s \Big) \geq \dfrac{\epsilon^2}{9}.
\end{equation}
In particular, $s$ satisfies \eqref{eq:5u}.

Therefore Steps 1 and 2 apply to $s \in \p \Dd\left(\var,\epsilon/3\right)$. They yield $\big( \PP_\delta[\zeta] - E_\star - \delta s \big)^{-1} =$
\begin{equations}\label{eq:6k}
 \dfrac{1}{\delta} \cdot  \matrice{ \phi_1 \\ \phi_2}^\top  e^{-i  \mu \delta \blr{\ell,x} }  \Pi^* \UU_\delta \cdot \left(\Di(\mu)-s\right)^{-1} \cdot 
 \UU_\delta^{-1} \Pi  e^{i  \mu \delta \blr{\ell,x} } \ove{\matrice{\phi_1 \\ \phi_2}} + \OO_{L^2[\zeta]}\left(\delta^{-1/3}\right).
\end{equations} 
Note that $\left(\Di(\mu)-s\right)^{-1}$ has no poles in the disk $\Dd(\var,\epsilon/3)$: otherwise $z$ could not be at distance at least $\epsilon$ from $\Sigma_{L^2}\big(\Di(\mu) \big)$. Thus, integrating \eqref{eq:6k} over the circle $\p \Dd(\var,\epsilon/3)$,
\begin{equation}\label{eq:6l}
\dfrac{1}{2\pi i} \oint_{\p \Dd(\var,\epsilon/3)} \big( \PP_\delta[\zeta] - E_\star-\delta s \big)^{-1} ds  = \OO_{L^2[\zeta]}\left(\delta^{-1/3}\right).
\end{equation}
We substitute $\lambda = E_\star+\delta s$ in \eqref{eq:6l} to get
\begin{equation}\label{eq:6j}
\dfrac{1}{2\pi i} \oint_{\p \Dd(E_\star+\delta\var,\epsilon \delta/3)} \big( \PP_\delta[\zeta] - \lambda \big)^{-1} d\lambda  = \OO_{L^2[\zeta]}\left(\delta^{2/3}\right).
\end{equation}
The equation \eqref{eq:6j} implies that $\big( \PP_\delta[\zeta] - \lambda \big)^{-1}$ cannot have a pole in $\Dd(E_\star + \var\delta,\epsilon \delta/3)$. Indeed, since $\PP_\delta[\zeta]$ is selfadjoint, the non-zero residues of its resolvent are non-zero projectors, hence have $L^2[\zeta]$-operator norm at least equal to $1$.

We deduce that $s \mapsto \big( \PP_\delta[\zeta] - E_\star-\delta s \big)^{-1}$ is holomorphic in the disk $\Dd(\var,\epsilon/3)$. So is the leading term in \eqref{eq:6k}. Their difference is bounded by $\OO_{L^2[\zeta]}(\delta^{-1/3})$ on the boundary of the disk. By the maximal principle, this difference is $\OO_{L^2[\zeta]}(\delta^{-1/3})$ also inside the disk. This shows that \eqref{eq:6k} holds when $s$ is in the disk $\Dd(\var,\epsilon/3)$. Equivalently \eqref{eq:3q} holds when $z \in \Dd(\var,\epsilon/3)$. This completes the proof of Theorem \ref{thm:2}.
\end{proof}



\section{A topological perspective}\label{sec:8}

\subsection{The role of $\var_\star^A$ and $\var_\star^B$ in the spectral flow}\label{sec:7.1} Assume that $P_0$ has Dirac points $(\xi_\star^A,E_\star)$ and $(\xi_\star^B,E_\star)$ -- where $\xi_\star^A$ and $\xi_\star^B$ were defined in \eqref{eq:3y}. 
Following Definition \ref{def:1}, these Dirac points are associated to Dirac eigenbasis $(\phi_1^A,\phi_2^A)$ and $(\phi_1^B,\phi_2^B)$: 
\begin{equation}\label{eq:1s}
\phi_1^J \in L^2_{\xi_\star^J,\tau}, \ \ \phi_2^J \in L^2_{\xi_\star^J,\ove{\tau}}, \ \ J=A,B; \ \  \text{ and  } \ \var_\star^J = \blr{\phi_1^J, W\phi_1^J}_{L^2_{\xi_\star^J}}.
\end{equation}
We recall that $\var_\star^J$ does not depend on the choice of Dirac eigenbasis satisfying \eqref{eq:1s}. The next result is a key identity -- see also \cite[\S7.1]{LWZ}.

\begin{lem}\label{lem:1o} The identity $\var_\star^A+\var_\star^B = 0$ holds. 
\end{lem}

\begin{proof} 1. We claim that $\II \phi_1^A \in L^2_{\xi_\star^B,\tau}$. Thanks to \eqref{eq:3y},
\begin{equation}
-\xi_\star^A = -\dfrac{2\pi}{3}(2k_1+k_2) = \dfrac{2\pi}{3}(k_1+2k_2) = \xi_\star^B \mod 2\pi\Lambda^*.
\end{equation}
Because $\phi_1^A \in L^2_{\xi_\star^A,\tau}$,
\begin{equations}
(\II \phi_1^A)(x+w) = \phi_1^A(-x-w) = e^{-i\lr{\xi_\star^A,w}} (\II\phi_1^A)(x) = e^{i\lr{\xi_\star^B,w}} (\II\phi_1^A)(x), 
\\
(\RR \II \phi_1^A)(x) = \phi_1^A(-Rx) = \tau \phi_1^A(-x) = \tau (\II \phi_1^A)(x).
\end{equations}
It follows that $\II \phi_1^A \in L^2_{\xi_\star^B,\tau}$ -- as claimed. The same calculation shows that $\II \phi_2^A \in L^2_{\xi_\star^B,\otau}$. 

The operator $P_0$ is $\II$-invariant. Therefore, $\II \phi_1^A$ and $\II \phi_2^A$ form an orthonormal basis of $\ker_{L^2_{\xi_\star}}\hspace*{-1.5mm}\big(P_0(\xi_\star^B)-E_\star\big)$; and 
$(\II \phi_1^A,\II \phi_2^A)$ is a Dirac eigenbasis for $(\xi_\star^B,E_\star)$. 

2. Because $W$ is odd and $\var_\star^B$ does not depend on the choice of Dirac eigenbasis,
\begin{equation}
\var_\star^B = \blr{ \II \phi_1^A, W \II \phi_1^A }_{L^2_{\xi_\star^B}} = - \blr{\phi_1^A, W \phi_1^A}_{L^2_{\xi_\star^A}} = -\var_\star^A.
\end{equation}
This completes the proof. \end{proof}

Recall the assumption (H4): for every $\zeta \notin \{ 2\pi/3, \ 4\pi/3  \} \mod 2\pi \Z$ and $\tau, \ \tau' \in \R$,
\begin{equation}
\lambda_{0,j_\star}(\zeta k + \tau k') < \lambda_{0,j_\star+1}(\zeta k + \tau' k').
\end{equation}

\begin{lem}\label{lem:1t} Assume $\operatorname{(H1)}$ -- $\operatorname{(H4)}$ hold for both $\xi_\star^A$ and $\xi_\star^B$. There exists a function
$E \in C^0\big(\R/(2\pi \Z), \R\big)$ with $E(\zeta_\star^A) = E(\zeta_\star^B) = E_\star$ and such that
\begin{equation}\label{eq:6f}
\forall \zeta \in \R, \ \  E(\zeta) \notin \Sigma_{L^2[\zeta],\ess} \big( \PP_\delta[\zeta] \big).
\end{equation}
Moreover, there exist $\mub > 0$ and $\delta_0 > 0$ such that if
\begin{equation}
\delta \in(0,\delta_0), \ \ \zeta \in [0,2\pi], \ \  |\zeta - 2\pi/3| \geq \mub \delta, \ \ |\zeta-4\pi/3| \geq \mub \delta,
\end{equation}
then the operator $\PP_\delta[\zeta]$ has no spectrum in $[E(\zeta)-\delta,E(\zeta)+\delta]$. 
\end{lem}

\begin{proof} 1.  Set $r(\zeta) = \dist\big(\zeta,\{2\pi/3, 4\pi/3\}\big)$. We first show that there exists $a > 0$ such that for $\zeta \in [0,2\pi]$,
\begin{equation}\label{eq:6b}
\inf_{\tau, \tau' \in \R} \big(\lambda_{0,j_\star+1}(\zeta k + \tau' k') - \lambda_{0,j_\star}(\zeta k + \tau k') \big)  \geq 4a \cdot r(\zeta).
\end{equation}
Otherwise, we can find $\zeta_n \in [0,2\pi], \tau_n, \tau'_n \in \R$ such that 
\begin{equation}\label{eq:6a}
\lambda_{0,j_\star+1}(\zeta_n k + \tau'_n k') - \lambda_{0,j_\star}(\zeta_n k + \tau_n k') \big)  \leq \dfrac{r(\zeta_n)}{n} = \dfrac{1}{n} \cdot \dist\big(\zeta_n, \{2\pi/3,4\pi/3\} \big).
\end{equation}
Using periodicity of the dispersion curves, we can assume that $\tau_n, \tau'_n$ both live in $[0,2\pi]$. In particular we can pass to converging subsequences: there exist $\zeta_\infty, \tau_\infty$ and $\tau'_\infty$ with
\begin{equation}\label{eq:3d}
\lambda_{0,j_\star}(\zeta_\infty k + \tau_\infty k') = 
\lambda_{0,j_\star+1}(\zeta_\infty k + \tau'_\infty k').
\end{equation}
Because of (H4), $\zeta_\infty \in \{2\pi/3,4\pi/3\} = \{\zeta_\star^A,\zeta_\star^B\} \mod 2\pi$. In the proof of Lemma \ref{lem:1x}, we showed that
\begin{equation}
\zeta_\star \in \{2\pi/3,4\pi/3\}, \ \ \tau, \ \tau' \in \R \ \ \Rightarrow \ \ 
\lambda_{0,j_\star}(\zeta_\star k + \tau k') \leq E_\star, \ \  \lambda_{0,j_\star+1}(\zeta_\infty k + \tau' k') \geq E_\star.
\end{equation}
Thanks to \eqref{eq:3d}, we deduce that $\lambda_{0,j_\star+1}(\zeta_\infty k + \tau'_\infty k') = E_\star = \lambda_{0,j_\star}(\zeta_\infty k + \tau_\infty k')$. The no-fold condition implies that $\zeta_\infty k + \tau_\infty k' = \zeta_\infty k + \tau'_\infty k' = \xi_\star$, where $\xi_\star \in \{\xi_\star^A,\xi_\star^B\}$ is a Dirac point momentum. In particular,
$\zeta_n k + \tau'_n k'$ and $\zeta_n k + \tau_n k'$ both converge to $\xi_\star$. We deduce that for $n$ sufficiently large,
\begin{equation}
\lambda_{0,j_\star+1}(\zeta_n k + \tau'_n k') - \lambda_{0,j_\star}(\zeta_n k + \tau_n k') \geq \nu_F \big| \zeta_n k + \tau'_n k'-\xi_\star \big| \geq \nu_F |k'| \cdot r(\zeta_n),
\end{equation}
because $\lr{\xi_\star,v} \in \{2\pi/3,4\pi/3\}$.
This contradicts \eqref{eq:6a}. We deduce that \eqref{eq:6b} holds for some $a>0$. Without loss of generalities, we assume below that $a < \nu_F |\ell|$.

2. Define 
\begin{equation}
E(\zeta) \de 2a \cdot r(\zeta) + \sup_{\tau \in \R} \lambda_{0,j_\star}(\zeta k+\tau k').
\end{equation}
This is a continuous, $2\pi$-periodic function.  Observe that  for every $\xi \in \zeta k + \R k'$,
\begin{equation}\label{eq:6c}
\lambda_{0,j_\star}(\xi) \leq E(\zeta) - 2a \cdot r(\zeta) \leq E(\zeta) + 2a \cdot r(\zeta) \leq \lambda_{0,j_\star+1}(\xi).
\end{equation}
Assume that $a \cdot r(\zeta) \geq \delta$ and that $\lambda \in [E(\zeta)-\delta,E(\zeta)+\delta]$. Since the dispersion surfaces are labeled in increasing order, we deduce that
\begin{equation}
\xi \in \zeta k + \R k' \ \ \Rightarrow \ \ \dist\big( \Sigma_{L^2_\xi}\big( P_0(\xi) \big), \lambda \big) \geq a \cdot r(\zeta).
\end{equation}
The reconstruction formula \eqref{eq:3z} and the spectral theorem yield
\begin{equation}\label{eq:6d}
a \cdot r(\zeta) \geq \delta , \  \ \lambda \in [E(\zeta)-\delta,E(\zeta)+\delta] \ \ \Rightarrow \ \ 
\left\| \big( P_0[\zeta] - \lambda \big)^{-1} \right\|_{L^2[\zeta]} \leq \dfrac{1}{a \cdot r(\zeta)}.
\end{equation}

3. We now observe that under the assumptions of \eqref{eq:6d},
\begin{equation}\label{eq:6e}
\PP_\delta[\zeta] - \lambda  = (P_0[\zeta] - \lambda) \cdot \left( \Id + \delta \cdot \big(P_0[\zeta] - \lambda\big)^{-1} \cdot   \kappa_\delta W  \right).
\end{equation}
Because of \eqref{eq:6d} and since $\kappa, W$ are in $L^\infty$, there exist $\delta_0 > 0$ and $\mub > 0$ with
\begin{equation}
\delta \in (0,\delta_0), \ \ \zeta \in [0,2\pi], \ \ r(\zeta) \geq \mub \delta \ \ \Rightarrow \ \ 
\left\|\delta \cdot \big(P_0[\zeta] - \lambda\big)^{-1} \cdot   \kappa_\delta W \right\|_{L^2[\zeta]} \leq \dfrac{1}{2}.
\end{equation}
In particular, the second factor in the RHS of \eqref{eq:6e} is invertible via a Neumann series. We deduce that $\PP_\delta[\zeta] - \lambda$ is invertible. This implies that $\PP_\delta[\zeta]$ has no spectrum in $[E(\zeta)-\delta,E(\zeta)+\delta]$, as long as $r(\zeta) \geq \mub \delta$.

4. It remains to show that $E(\zeta)$ is not in the essential spectrum of $\PP_\delta[\zeta]$, independently of $\zeta$. Because of Step 3, this holds for every $\zeta$ such that $r(\zeta) \geq \mub \delta$. Fix $\zeta$ such that $r(\zeta) < \mub \delta$. Let $\xi_\star$ be a Dirac point closest to $\zeta k + \R k'$: the distance between $\xi_\star$ and the line $\zeta k + \R k'$ is $r(\zeta) |\ell|$. Because of \eqref{eq:6c},
\begin{equation}
\lambda_{0,j_\star}\big(\zeta k + \tau k' \big) + 2a \cdot r(\zeta) \leq E(\zeta) \leq \lambda_{0,j_\star+1}\big( \zeta k + \tau k' \big) - 2a \cdot r(\zeta).
\end{equation}
Since $\xi_\star$ is a Dirac point, we get
\begin{equation}
E_\star - \big(\nu_F|\ell| - 2a\big) \cdot r(\zeta)  + O\left(r(\zeta)^2\right) \leq E(\zeta) \leq E_\star + \big(\nu_F|\ell| + 2a\big) \cdot r(\zeta)  + O\left(r(\zeta)^2\right).
\end{equation}
Hence, for $\delta$ sufficiently small,
\begin{equation}
E(\zeta) \in \big[ E_\star - \big(\nu_F|\ell|- \big) \cdot r(\zeta), E_\star + \big(\nu_F|\ell|- a\big) \cdot r(\zeta) \big].
\end{equation}
Fix $\te \in (0,1)$ such that $\nu_F|\ell|- a = \te \nu_F |\ell|$; $\te$ exists because $a \in (0,\nu_F |\ell|)$. Then 
\begin{equation}
E(\zeta) \in \Dd \left(E_\star, \te \sqrt{\var_F^2 \delta^2 +  r(\zeta)^2 \cdot \nu_F^2 |\ell|^2} \right).
\end{equation}
Apply Theorem \ref{thm:1} with $\mu_\sharp > \mu_\flat$: for $\delta$ sufficiently small and $|\zeta - \zeta_\star| < \mu_\sharp \delta$,  
$E(\zeta) \notin \Sigma_{L^2[\zeta],\ess}\big( P_{\pm\delta}(\zeta) \big)$. This implies that $E(\zeta)$ is not in the essential spectrum of $\PP_\delta[\zeta]$ as long as $r(\zeta) < \mu_\flat \delta$, which concludes the proof.\end{proof}

Lemma \ref{lem:1t} allows to define the spectral flow of the family $\zeta \mapsto \PP_\delta[\zeta]$ as $\zeta$ runs from $0$ to $2\pi$: it is the signed number of eigenvalues of $\PP_\delta[\zeta]$ that cross the curve $\zeta \mapsto E(\zeta)$ (with downwards crossings counted positively). Because $\PP_\delta[\zeta]$ depends periodically on $\zeta$, the spectral flow of $\PP_\delta$ is a topological invariant. We refer to \cite[\S4]{Wa} for an introduction to spectral flow. We are now ready to prove Corollary \ref{cor:4}.

\begin{proof}[Proof of Corollary \ref{cor:4}] We split $[0,2\pi]$ in three parts: $[0,2\pi] = I_A \cup I_B \cup I_0$ with
\begin{equation}
I_J \de \left[\zeta_\star^J - \mub \delta , \zeta_\star^J+\mub \delta \right], \ J = A,B; \ \ I_0 \de [0,2\pi] \setminus \big( I_A \cup I_B \big),
\end{equation}
where we identified $\zeta_\star^J$ with their reduction modulo $2\pi \Z$.
The spectral flow of $\zeta \in I_0 \mapsto \PP_\delta[\zeta]$ through $E_\star$ vanishes because of Lemma \ref{lem:1t}. 

In order to compute the spectral flow of $\zeta \in I_J \mapsto \PP_\delta[\zeta]$ through $E_\star$, we fix $\tmu > \mub, \ \tvar > \var_N$ and we apply Corollary \ref{cor:2}. This result allows to precisely count the number $N_\pm^J$ of eigenvalues of $\PP_\delta[\zeta_\star^J \pm \mub \delta]$ in the set 
\begin{equation}
\EE \de \left[ E_\star - \delta \sqrt{\tvar^2 + \mub^2 \cdot \nu_F^2 | \ell|^2}, E_\star \right]
\end{equation}
in terms of the number of eigenvalues $2N+1$ of the Dirac operator $\Di(\mu)$. Thanks to Lemma \ref{lem:1a}, we find:
\begin{equation}
N_-^J = N+1, \ \ N_+^J = N \text{ if } \var_\star^J > 0; \ \ \ \ N_-^J = N, \ \ N_+^J = N+1 \text{ if } \var_\star^J < 0.
\end{equation}
In particular, the spectral flow of $\zeta \in I_J \mapsto \PP_\delta[\zeta]$ through $E_\star$ is $N_+^J - N_-^J = -\sgn(\var_\star^J)$ -- see e.g. \cite[\S4.1]{Wa}. Since $\var_\star^A$ and $\var_\star^B$ have opposite sign, the spectral flow of the whole family $\zeta \in [0,2\pi] \mapsto \PP_\delta[\zeta]$ vanishes.\end{proof}

\subsection{Magnetic perturbations of honeycomb Schr\"odinger operators}\label{sec:7.2} Let $V$ be a honeycomb potential and $\Aa \in C^\infty\big(\R^2,\R^2\big)$ be $\Lambda$-periodic, odd and real-valued. Set 
\begin{equation}
\tPP_\delta \de -\Delta + V + \delta \cdot \kappa_\delta \cdot \tW, \ \ \ \ \tW \de \Aa \cdot D_x + D_x \cdot \Aa.
\end{equation}
This operator is a non-local perturbation of $P_0 = -\Delta+V$, where $\delta \cdot \Aa$ plays the role of a perturbing magnetic field. We introduce similarly to $\PP_\delta[\zeta]$ the operator $\tPP_\delta[\zeta]$ formally equal to $\tPP_\delta$ but acting on $L^2[\zeta]$. 

We first state a simple analog of Lemma \ref{lem:1h} regarding the coefficient $\var_\star$:

\begin{lem}\label{lem:1p} Let $(\xi_\star,E_\star)$ be a Dirac point of $P_0$ with Dirac eigenbasis $(\phi_1,\phi_2)$ -- see Definition \ref{def:1}. Then $\blr{\phi_1,\tW\phi_2 }_{L^2_{\xi_\star}} = \blr{\phi_2,\tW\phi_1}_{L^2_{\xi_\star}} = 0$. Furthermore,
\begin{equation}
\te_\star \de \blr{\phi_1,\tW\phi_1}_{L^2_{\xi_\star}} = -\blr{\phi_2,\tW\phi_2}_{L^2_{\xi_\star}}.
\end{equation}
\end{lem}

See Appendix \ref{app:2} or \cite[Proposition 5.1]{LWZ} for the proof. Below we state Corollary \ref{cor:6}, which condenses the analog of Theorem \ref{thm:2} and Corollary \ref{cor:2} for the magnetic operator $\tPP_\delta[\zeta]$. We shall assume
\begin{enumerate}
\item[(H3')] The non-degeneracy condition $\te_\star \neq 0$ holds.
\end{enumerate}
We need the operator
\begin{equation}
\tDi(\mu) \de \matrice{0 & \nu_\star k' \\ \ove{\nu_\star k'} & 0} D_t + \mu \matrice{0 & \nu_\star \ell \\ \ove{\nu_\star \ell} & 0} + \te_\star \matrice{1 & 0 \\ 0 & -1} \kappa.
\end{equation}
We denote by $\{\te_j^\mu\}_{j=-n}^n$ its eigenvalues. They are all simple -- see Lemma \ref{lem:1a} -- and lie in $(-\te_F,\te_F)$, where $\te_F = |\te_\star|$.

\begin{cor}\label{cor:6} Assume that $\operatorname{(H1)}$, $\operatorname{(H2)}$ and $\operatorname{(H3')}$ hold and fix $\te_\sharp \in (\te_N,\te_\star)$ and $\tmu > 0$. There exists $\delta_0 > 0$ such that for 
\begin{equation}
\delta \in (0,\delta_0), \ \  \mu \in (- \tmu,\tmu), \ \ \zeta = \zeta_\star + \delta \mu,
\end{equation}
the operator $\Pp_\delta[\zeta]$ has exactly $2n+1$ eigenvalues $\{\lambda_{\delta,j}^\zeta\}_{j \in [-n,n]}$ in 
\begin{equation}
\left[E_\star - \delta \sqrt{\te_\sharp^2 +  \mu^2\cdot \nu_F^2 | \ell|^2}, \ E_\star + \delta\sqrt{\te_\sharp^2 + \mu^2\cdot \nu_F^2 | \ell|^2 } \ \right].
\end{equation}

\noindent These eigenvalues are simple. Furthermore, for each $j \in [-N,N]$, the eigenpairs $(\lambda_{\delta,j}^\zeta,v_{\delta,j}^\zeta)$ admit full expansions in powers of $\delta$:
\begin{equations}
\lambda_{\delta,j}^\zeta = E_\star +\te_j^\mu \cdot \delta + b_2^\mu \cdot  \delta^2   + \dots + b_M^\mu \cdot \delta^M + O\left( \delta^{M+1} \right), 
\\
v_{\delta,j}^\zeta(x) = e^{i(\zeta-\zeta_\star) \blr{\ell,x}} \Big( g_0^\mu\big(x,\delta \blr{k',x} \hspace{-.5mm}\big) + \dots + \delta^M \cdot g_M\big(x,\delta \blr{k',x}\hspace{-.5mm}\big) \Big) + o_{H^k}\left( \delta^M \right).
\end{equations}
In the above expansions:
\begin{itemize}
\item $M$ and $k$ are any integer; $H^k$ is the $k$-th order Sobolev space.
\item $\te_j^\mu$ is the $j$-th eigenvalue of $\tDi(\mu)$.
\item The terms $b_m^\mu \in \R$, $g_m^\mu \in X$ are recursively constructed via multiscale analysis.
\item The leading order term $g_0^\mu$ satisfies
\begin{equation}
g_0^\mu(x,t) = \beta_1^\mu(t) \phi_1(x) + \beta_2^\mu(t) \phi_2(x), \ \ \ \ \left(\tDi(\mu) - \te_j^\mu\right) \matrice{\beta_1^\mu \\ \beta_2^\mu} = 0.
\end{equation}
\end{itemize}
\end{cor}

The proof is identical to that of Theorem \ref{thm:2} and Corollary \ref{cor:2}; we do not reproduce it here. Let $\te_\star^J$ be the coefficient $\te_\star$ associated to the Dirac point $(\xi_\star^J,E_\star)$. The main difference between $\PP_\delta[\zeta]$ and $\tPP_\delta[\zeta]$ lies in the next identity -- see also \cite[\S7.1]{LWZ}.

\begin{lem} The identity $\te_\star^A=\te_\star^B$ holds. 
\end{lem}

\begin{proof} Because of Step 1 in the proof of Lemma \ref{lem:1o}, $(\II \phi_1^A,\II \phi_2^A)$ is a Dirac eigenbasis for $(\xi_\star^B,E_\star)$. Since $\te_\star^B$ does not depend on the choice of Dirac eigenbasis and $\tW$ commutes with $\II$,
\begin{equation}
\te_\star^B = \bblr{ \II \phi_1^A, \tW \II \phi_1^A }_{L^2_{\xi_\star^B}} = \bblr{\phi_1^A, \tW \phi_1^A}_{L^2_{\xi_\star^A}} = \te_\star^A.
\end{equation}
This completes the proof. \end{proof}

Corollary \ref{cor:5} has the same proof as Corollary \ref{cor:4}. We find that the spectral flow of $\tPP_\delta$ in the $j_\star$-th gap as $\zeta$ runs from $0$ to $2\pi$ is equal to 
\begin{equation}
-\sgn\big(\te_\star^A\big) - \sgn\big(\te_\star^B\big) = -2 \cdot \sgn(\te_\star).
\end{equation}


\appendix

\section{}

\subsection{Proof of some identities}\label{app:2} We prove the identities relating the Dirac eigenbasis and $W$. Similar proofs arise in \cite{FLTW2,FLTW3,LWZ}.

\begin{proof}[Proof of Lemma \ref{lem:1i}] Below we use $\blr{\cdot,\cdot}$ instead of $\lr{\cdot,\cdot}_{L^2_{\xi_\star}}$ to simplify notations.

1. We first analyze the (2-vector) $\blr{\phi_1,D_x\phi_1}$. We observe that $\blr{\phi_1,D_x\phi_1} \in \R^2$ because $D_x$ is selfadjoint. Since $\phi_1 \in L^2_{\xi_\star,\tau}$, 
\begin{equation}
\blr{\phi_1, D_x \phi_1} = \blr{\RR \phi_1, \RR D_x \phi_1} = \blr{\tau \phi_1, \RR D_x \RR^{-1} \cdot \tau \phi_1} = \blr{\phi_1, \big( \RR D_x \RR^{-1} \big) \cdot \phi_1}.
\end{equation}
As $\RR D_x \RR^{-1} = R^{-1}D_x$, we conclude that $\blr{\phi_1,D_x\phi_1}$ is either $0$ or an eigenvector of $R$. Since the latter cannot be real, we conclude $\blr{\phi_1,D_x\phi_1} = 0$. The same argument applies to $\blr{\phi_2,D_x\phi_2}$.

2. We now analyze $\blr{\phi_1,D_x\phi_2}$. Since $\phi_1 \in L^2_{\xi_\star,\tau}$ and $\phi_2 \in L^2_{\xi_\star,\otau}$ 
\begin{equation}
\blr{\phi_1, D_x \phi_2} = \blr{\RR \phi_1, \RR D_x \phi_2} = \blr{\tau \phi_1, \RR D_x \RR^{-1} \cdot \otau \phi_2} = \otau^2 \blr{\phi_1, \big( \RR D_x \RR^{-1} \big) \cdot \phi_2}.
\end{equation}
As $\RR D_x \RR^{-1} = R^{-1}D_x$ and $\otau^2 = \tau$ we deduce 
$R\blr{\phi_1, D_x \phi_2} = \tau \blr{\phi_1, D_x \phi_2}$. 
This yields $\blr{\phi_1, D_x \phi_2} \in \ker_{\C^2}(R-\tau)$. This eigenspace is $\C \cdot [1,i]^\top$; thus there exists $\nu_\star \in \C$ with
\begin{equation}
2\blr{\phi_1, D_x \phi_2} = \nu_\star \cdot \matrice{1 \\ i}.
\end{equation}
If we identify the point $\eta = (\eta_1,\eta_2) \in \R^2$ with $\eta_1+i\eta_2 \in \C$, then 
\begin{equation}
2\blr{\phi_1, (\eta \cdot D_x) \phi_2} = 2\blr{\phi_1, (\eta_1 D_{x_1} + \eta_2 D_{x_2}) \phi_2} = \nu_\star \eta_1 + i \nu_\star \eta_2 = \nu_\star \eta.
\end{equation}
Above $\nu_\star \eta$ denotes the multiplication of $\nu_\star$ with $\eta = \eta_1+i\eta_2$. Taking the complex conjugate of this identity and observing that $\eta \cdot D_x$ is a selfadjoint operator, we get
\begin{equation}
2\blr{\phi_2, (\eta \cdot D_x) \phi_1} = \ove{\nu_\star \eta}.
\end{equation}

3. It remains to show that $|\nu_\star| = \nu_F$. Fix $\eta \in \R^2$ with $|\eta|=1$. Because of perturbation theory of eigenvalues, the operator $P_0(\xi_\star + t\eta)$ has precisely two eigenvalues near $E_\star$ when $t$ is sufficiently small -- see \cite[\S VII1.3, Theorem 1.8]{Ka}. Because $(\xi_\star, E_\star)$ is a Dirac point of $P_0$, they are
\begin{equation}\label{eq:2g}
E_\star \pm \nu_F t + O\left(t^2\right).
\end{equation}

Let $\xi = \xi_\star+t\eta$.
We want to construct approximate eigenvectors of $P_0(\xi)$. Let $a, b \in \C^2$, $\mu \in \R$ and $v \in H^2_{\xi_\star}$, with $v = O_{H^2_{\xi_\star}}(1)$ uniformly in $t$. Then
\begin{equations}\label{eq:5j}
e^{-it\blr{\eta,x}} \big(P_0  - E_\star + \mu t\big) e^{it\blr{\eta,x}} \cdot (a \phi_1 + b \phi_2 + tv)  \\
 = \big((D_x+t\eta)^2 + V - E_\star + \mu t\big)(a \phi_1 + b \phi_2 + tv)
\\
= t(P_0-E_\star)v + t (2\eta \cdot D_x + \mu) (a \phi_1 + b \phi_2) + O_{L^2_{\xi_\star}}\left(t^2\right). 
\end{equations}
We now construct $v$ such that
\begin{equation}\label{eq:5i}
(P_0-E_\star)v + (2\eta \cdot D_x + \mu) (a \phi_1 + b \phi_2) = 0.
\end{equation}
 This equation admits a solution if and only if $(2\eta \cdot D_x + \mu) (a \phi_1 + b \phi_2)$ is orthogonal to $\phi_1$ and $\phi_2$. This solvability condition is equivalent to
\begin{equation}\label{eq:2h}
\systeme{\blr{\phi_1,(2\eta \cdot D_x + \mu) (a \phi_1 + b \phi_2)}  = 0 
\\
\blr{\phi_2,(2\eta \cdot D_x + \mu) (a \phi_1 + b \phi_2)}  = 0} \ \Leftrightarrow \systeme{ \nu_\star \eta \cdot b + \mu a = 0  \\
\ove{\nu_\star \eta} \cdot b + \mu a = 0}.
\end{equation}
A non-trivial solution of \eqref{eq:2h} exists if and only if
\begin{equation}\label{eq:2i}
\Det \matrice{\mu & \nu_\star \eta \\ \ove{\nu_\star \eta} & \mu} = 0 \ \ \Leftrightarrow \ \ |\nu_\star \eta|^2 = |\nu_\star|^2 = \mu^2.
\end{equation}
Therefore, when $\mu = |\nu_\star|$, we can construct $(a,b) \neq (0,0)$ satisfying \eqref{eq:2h} for $\mu = \pm |\nu_\star|$. With this choice,  \eqref{eq:5i} admits a solution $v$. It follows from \eqref{eq:5j} that
\begin{equation}
\big(P_0(\xi)  - E_\star + |\nu_\star| t\big) \cdot e^{it\blr{\eta,x}} (a \phi_1 + b \phi_2 + tv) = O\left(t^2\right).
\end{equation}
In other words, we constructed a $O(t^2)$-accurate quasimode for $P_0(\xi)$, with energy $E_\star + |\nu_\star| t$. A general principle -- see e.g. \cite[Lemma 3.1]{DFW} -- implies that $P_0(\xi)$ has an eigenvalue at $E_\star - |\nu_\star| t + O(t^2)$. Because of \eqref{eq:2g}, this eigenvalue must be $E_\star - \nu_F t + O(t^2)$. This implies $|\nu_\star| = \nu_F$ and completes the proof.\end{proof}

\begin{proof}[Proof of Lemma \ref{lem:1h}] Below we use $\blr{\cdot,\cdot}$ instead of $\lr{\cdot,\cdot}_{L^2_{\xi_\star}}$ to simplify notations. We start by proving the first identity. Since $\II$ is an isometry and $\II \phi_2 = \ove{\phi_1}$, 
\begin{equation}
\blr{\phi_2,W\phi_1} = \blr{\II \phi_2, \II W \II \phi_1} = -\blr{\overline{\phi_1},W\overline{\phi_2}} = - \blr{\phi_2,W\phi_1}.
\end{equation}
This implies $\blr{\phi_2,W\phi_1} = 0$. Using that $W$ is real-valued, $\blr{\phi_1,W\phi_2} = 0$ as well.
We prove now the second identity: for the same reasons as above,
\begin{equation}
\blr{\phi_1,W\phi_1} = \blr{\II \phi_1, \II W \II \phi_1} = -\blr{\overline{\phi_2},W\overline{\phi_2}} = -\blr{\phi_2,W\phi_2}.
\end{equation} 
This completes the proof. \end{proof}

\begin{proof}[Proof of Lemma \ref{lem:1p}] Below we use $\blr{\cdot,\cdot}$ instead of $\lr{\cdot,\cdot}_{L^2_{\xi_\star}}$ to simplify notations. We start by proving the first identity. Since $\II$ is an isometry from $L^2_{\xi_\star^A}$ to $L^2_{\xi_\star^B}$ and $\II \phi_2 = \ove{\phi_1}$, $\II \tW = \tW \II$,
\begin{equation}
\blr{\phi_2,\tW\phi_1} = \blr{\II \phi_2, \II \tW \II \phi_1} = \blr{\overline{\phi_1}, \tW\overline{\phi_2}}.
\end{equation}
Moreover, $\ove{\tW} = -\tW$ because $\Aa$ is real-valued and $D_x = \frac{1}{i} \nabla$. Therefore,
\begin{equation}
\blr{\phi_2,\tW\phi_1} = -\blr{\overline{\phi_1}, \overline{\tW\phi_2}} = -\blr{\tW\phi_2, \phi_1} = -\blr{\phi_2, \tW\phi_1}.
\end{equation}
We used in the last equality the selfadjointness of $\tW$. We deduce $\blr{\phi_2,\tW\phi_1} = 0$. Similarly, $\blr{\phi_1,\tW\phi_2} = 0$.

We prove now the second identity: for the same reasons as above,
\begin{equation}
\blr{\phi_1,\tW\phi_1} = \blr{\II \phi_1, \II \tW \phi_1} = -\blr{\overline{\phi_2},\tW\overline{\phi_2}} = -\blr{\ove{\phi_2},\ove{\tW\phi_2}} = -\blr{\phi_2,\tW\phi_2}.
\end{equation} 
This completes the proof.
\end{proof}

\subsection{Spectrum of the Dirac operator}\label{app:1}

\begin{proof}[Proof of Lemma \ref{lem:1a}] 1. Introduce the matrices:
\begin{equations}
\bmu = \dfrac{1}{\nu_F |k'|}\matrice{0 & \nu_\star k' \\ \ove{\nu_\star k'} & 0}, \ \ \bmd = \dfrac{1}{\nu_F |\ell|}\matrice{0 & \nu_\star \ell \\ \ove{\nu_\star \ell'} & 0}, \ \ \bmt = \matrice{1 & 0 \\ 0 & -1}.
\end{equations}
Note that $\bmj^2 = \Id$. Moreover, the matrices $\bmj$ anticommute: $\bmj \bmk + \bmk \bmj = 0$ when $j \neq k$. Indeed,  $\bmu \bmd + \bmd \bmu$ equals
\begin{equation}
\dfrac{1}{\nu_F |k'| \cdot \nu_F |\ell|}\matrice{\nu_\star k' \cdot \ove{\nu_\star \ell} + \nu_\star \ell \cdot \ove{\nu_\star k'} & 0 \\ 0 & \nu_\star k' \cdot \ove{\nu_\star \ell} + \nu_\star \ell \cdot \ove{\nu_\star k'}} = \dfrac{2\Re(k' \ove{\ell})}{|\ell k'|} = 0,
\end{equation}
because $\Re(k' \ove{\ell}) = \blr{k',\ell} = 0$. 
With these notations, 
\begin{equation}
\Di(\mu) =  \nu_F |k'| \bmu D_t +  \mu \cdot \nu_F |\ell| \bmd + \var_\star \bmt \kappa = \Di_\star + \mu \cdot \nu_F |\ell| \bmd.
\end{equation}

2. The formula for the essential spectrum is derived by looking at those of the asymptotic operators:
\begin{equation}
\Di_\pm(\mu) \de \nu_F |k'| \bmu D_t +  \mu \cdot \nu_F |\ell| \bmd \pm \var_\star \bmt.
\end{equation}
These are Fourier multipliers. Their essential spectrum corresponds to the possible eigenvalues of their symbol as the Fourier parameter runs through $\R$. We find
\begin{equation}
\Sigma_{L^2,\ess}\big(\Di_\pm(\mu) \big) =
\R \setminus \left[-\sqrt{\var_F^2 + \mu^2\cdot \nu_F^2 |\ell|^2 }, \sqrt{\var_F^2 + \mu^2 \cdot \nu_F^2 |\ell|^2} \right].
\end{equation}

3. We start by studying the bifurcation of the zero mode of $\Di_\star = \Di(0)$. This mode satisfies the equation $\Di(0) u = 0$ or equivalently
\begin{equation}
\big(\nu_F |k'|\p_t + \var_\star i \bmu \bmt \kappa \big) u =0.
\end{equation}
The matrix $i\bmu\bmt$ has eigenvalues $\pm 1$. Let $u_0$ be an eigenvector of $i\bmu\bmt$ associated with the eigenvalue $\sgn(\var_\star)$ and set
\begin{equation}
u(t) = u_0 \cdot \exp\left( - \dfrac{\var_F}{\nu_F |k'|}\int_0^t \kappa(s) ds \right).
\end{equation} 
A direct calculation shows that $u$ is an eigenvector of $\Di(0)$.

We claim that $\bmd u_0 = \sgn(\var_\star) u_0$. Since $i\bmu\bmt u_0 = \sgn(\var_\star) u_0$,
\begin{equation}
i\bmd \bmu \bmt u_0 = \sgn(\var_\star) \bmd u_0, \ \ \ \ i\bmd \bmu \bmt = \dfrac{i}{|k'\ell|}  \matrice{ \ell \ove{k'} & 0 \\ 0 & -\ove{\ell} k' }.
\end{equation}
Recall that $\Re(\ell \ove{k'})=0$ because $\ell$ and $k'$ are orthogonal. Therefore $-\ove{\ell} k' = \ell \ove{k'}$,  and we deduce that
\begin{equations}
\sgn(\var_\star) \bmd u_0 = -\dfrac{i}{|k'\ell|}  k' \ove{\ell}  u_0 
\ \  \Leftrightarrow \ \ \bmd u_0 = \sgn \left(\Im\left( k'\ove{\ell} \right) \right) \cdot \sgn(\var_\star) u_0.
\end{equations}
We recall that $k' = -a_2k_1 + a_1 k_2$, $k = b_2 k_1-b_1 k_2$, $a_2b_1 - b_2 a_1  = 1$ -- see \S\ref{sec:6.1}. Hence:
\begin{equation}
\Im\left( k'\ove{\ell} \right) = \Det[k,k'] = (a_2b_1 - b_2 a_1) \cdot \Det[k_1,k_2] = 1 >0.
\end{equation}
We deduce that $\bmd u_0 = \sgn(\var_\star) u_0$ and $\bmd u = \sgn(\var_\star)  u$.
 
We recall that $\Di(\mu) = \Di_\star+\mu \cdot \nu_F|\ell| \bmd$, $\Di_\star u = 0$ and obtain
\begin{equations}
\Di(\mu)u = \mu \cdot \nu_F |\ell|\cdot \sgn(\var_\star) u.
\end{equations}
This shows that $\mu \cdot \nu_F|\ell| \cdot \sgn(\var_\star)$ is an eigenvalue of $\Di(\mu)$.

4. Let $\var_j > 0$ be an eigenvalue of $\Di_\star$. Since $\bmd \Di_\star = -\Di_\star \bmd$, we deduce that $-\var_j$ is also eigenvalue of $\Di_\star$. The respective eigenvectors are denoted $f_+,f_-$ and are related via $\bmd f_+ = f_-$.  We look for an eigenpair $(E,a_+f_++a_-f_-)$ of $\Di(\mu) = \Di_\star+\mu \cdot \nu_F |\ell| \bmd$: it suffices to solve the equation
\begin{equations}
(\Di_\star + \mu \nu_F |\ell| \bmd) \sum_\pm a_\pm f_\pm = E \sum_\pm a_\pm f_\pm
\\
\Leftrightarrow \  \sum_\pm \pm \var_j a_\pm f_\pm + \mu \cdot \nu_F| \ell| a_\pm f_\mp = E \sum_\pm a_\pm f_\pm \ \Leftrightarrow \ (\var_j \bst + \mu \cdot \nu_F | \ell| \bsu) \matrice{a_+ \\a_-} = E a.
\end{equations}
This is equivalent to $(E,a)$ being an eigenpair of $\var_j \bst + \mu \cdot \nu_F | \ell| \bsu$. We conclude that $E = \pm \sqrt{\var_j^2+ \mu^2 \cdot \nu_F^2 | \ell|^2}$ are both eigenvalues of $\Di(\mu)$.

5. So far we only showed that the eigenvalues of $\Di_\star$ induces eigenvalues of $\Di(\mu)$. We must prove the converse statement. Without loss of generalities, $\mu \neq 0$. We first deal with eigenvalues of $\Di(\mu)$ which \textit{apparently} do not bifurcate from the zero mode of $\Di_\star$. That is, we assume first that $(E,f)$ is an eigenpair of $\Di(\mu) = \Di_\star+\mu \cdot  \nu_F| \ell| \bmd$ with $E \neq \sgn(\var_\star) \cdot \nu_F |\ell| \mu$. 

We first claim that $f$ and $g = \bmd f$ are linearly independent. Otherwise, we would have $f=\bmd f$ or $f=-\bmd f$ because $\bmd = \Id$. This would imply  respectively in the first and second case
\begin{equation}\label{eq:3u}
\Di_\star f = (E-\mu \cdot \nu_F |\ell|)f \  \text{ or } \ \Di_\star f = (E+\mu \cdot \nu_F |\ell|)f.
\end{equation}
In particular, $f$ is an eigenvector of $\Di_\star$ with $\bmd f$ and $f$ colinear. Because of Step 3 it must be a zero mode of $\Di_\star$. Because of Step 2 we must have $\bmd f = \sgn(\var_\star) f$. Going back to \eqref{eq:3u}, $E=\sgn(\var_\star)\cdot \nu_F |\ell|$, which contradicts our assumption.

We now look for an eigenpair of $\Di_\star$ in the form $(\var_j,af+bg)$. We get the equation:
\begin{equations}
\Di_\star (af+bg) = \var_j (af+bg) \  \Leftrightarrow \ a (Ef-\mu \cdot \nu_F |\ell| g) + b (\mu \cdot \nu_F |\ell| f - E g) = \var_j (af+bg)
\\
  \Leftrightarrow \ \big( E\bsu + i\mu \cdot \nu_F |\ell|\bsd \big) \matrice{a \\ b} = \var_j \matrice{a \\ b}.
\end{equations}
Hence, $\var$ is an eigenvalue of $E\bsu + i\mu \cdot \nu_F |\ell|\bsd$; equivalently, $\var_j = \pm \sqrt{E^2- \mu^2 \cdot \nu_F^2 |\ell|^2}$. 

6. To conclude we deal with the case of an eigenpair $(E,f)$ of $\Di(\mu)$ with $E =  \mu \cdot \nu_F |\ell| \cdot \sgn(\var_\star)$ -- i.e. when $E$ seemingly bifurcates from the zero mode of $\Di_\star$. 

We claim that $f$ and $g = \bmd f$ are colinear. Otherwise, following the last part of Step 5, we would be able to construct $[a,b]^\top$ eigenvector of $\sgn(\var_\star) \bsu + i \bsd$ such that $af+bg$ is an eigenvector of $\Di_\star$. The matrix $\sgn(\var_\star) \bsu + i \bsd$ has only one eigenvector, which is either $[0,1]^\top$ or $[1,0]^\top$. Therefore either $f$ or $g$ -- but not both -- are eigenvector of $\Di_\star$. This implies that $f$ or $g$ is a zero eigenvector of $\Di_\star$. In particular, $f$ and $\bmd f$ (or $g$ and $\bmd g$) are colinear -- which is a contradiction.

It follows that deduce that $f = \bmd f$ or $f = -\bmd f$. If $\bmd f = \sgn(\var_\star) f$, we are done. In the other case, we deduce existence of an eigenpair $(f,2\mu \cdot \nu_F |\ell| \cdot \sgn(\var_\star))$ of $\Di_\star$. This would require $f$ and $\bmd f$ to be colinear, which is impossible. This completes the proof of the converse statement.

7. The argument presented in Step 5 and 6 shows that the eigenvalues of $\Di(\mu)$ and $\Di_\star$ have the same multiplicity. \cite[Appendix C]{DFW} shows that $\Di_\star$ has only simple eigenvalues. This completes the proof.
\end{proof}

\subsection{A calculation}\label{app:3}

\begin{proof}[Proof of Lemma \ref{lem:1d}] 1. From Theorem \ref{thm:1}, when \eqref{eq:3l} is satisfied,
\begin{equation}
(P_\delta[\zeta]-\lambda)^{-1} \pm (P_{-\delta}[\zeta]-\lambda)^{-1} = S_\delta(\mu,z) \pm S_{-\delta}(\mu,z) + \OO_{L^2[\zeta]}\left(\delta^{-1/3}\right).
\end{equation}
A calculation yields $S_\delta(\mu,z) \pm S_{-\delta}(\mu,z) =$
\begin{equations}
\dfrac{1}{  \delta} \cdot  \matrice{ \phi_1 \\ \phi_2}^\top  e^{-i  \mu \delta \blr{\ell,x} }  \Pi^* \cdot  \UU_\delta \Big( \big( \Di_+(\mu) - z \big)^{-1} \pm \big( \Di_-(\mu) - z \big)^{-1} \Big) \UU_\delta^{-1} \cdot \Pi  e^{i  \mu \delta \blr{\ell,x} } \ove{\matrice{\phi_1 \\ \phi_2}}.
\end{equations}
We now compute the resolvent computation $\big( \Di_+(\mu) - z \big)^{-1} \pm \big( \Di_-(\mu) - z \big)^{-1}$. We have:
\begin{equations}
\big( \Di_+(\mu) - z \big)^{-1} +  \big( \Di_-(\mu) - z \big)^{-1} = 2\matrice{ z &   \nu_\star k'D_t+\mu \nu_\star \ell  \\    \ove{\nu_\star k'}D_t+\mu\ove{ \nu_\star\ell} 
  & z} R_0(\mu,z), 
\\
\big( \Di_+(\mu) - z \big)^{-1} -  \big( \Di_-(\mu) - z \big)^{-1} = 2\var_\star\matrice{1 & 0 \\ 0 & -1} R_0(\mu,z)= 2\var_\star\bst R_0(\mu,z).
\end{equations}
Above we recall that $R_0(\mu,z) = \big( \nu_F^2 |k'|^2 D_t^2 + \mu^2 \cdot \nu_F^2|\ell|^2 + \var_F^2 - z^2 \big)^{-1}$. This implies that $S_\delta(\mu,z) + S_{-\delta}(\mu,z)$ equals
\begin{equations}
\dfrac{2}{  \delta} \cdot  \matrice{ \phi_1 \\ \phi_2}^\top  e^{-i  \mu \delta \blr{\ell,x} }  \Pi^* \cdot  \UU_\delta \matrice{z & \nu_\star k'D_t + \mu \nu_\star\ell \\ \ove{\nu_\star k'} D_t + \mu\ove{\nu_\star \ell} & z} R_0(\mu,z) \UU_\delta^{-1} \cdot \Pi  e^{i  \mu \delta \blr{\ell,x} } \ove{\matrice{\phi_1 \\ \phi_2}}; 
\\
\text{and } \ 
S_\delta(\mu,z) - S_{-\delta}(\mu,z) = \dfrac{2}{  \delta} \cdot  \matrice{ \phi_1 \\ \phi_2}^\top  e^{-i  \mu \delta \blr{\ell,x} }  \Pi^* \cdot  \UU_\delta \var_\star \bst R_0(\mu,z) \UU_\delta^{-1} \cdot \Pi  e^{i  \mu \delta \blr{\ell,x} } \ove{\matrice{\phi_1 \\ \phi_2}}.
\end{equations}
We similarly obtain $(k' \cdot D_x) \big(S_\delta(\mu,z) - S_{-\delta}(\mu,z)\big)$
\begin{equations}
 = \dfrac{2}{ \delta} \cdot  \matrice{ (k' \cdot D_x) \phi_1 \\ (k' \cdot D_x) \phi_2}^\top  e^{-i  \mu \delta \blr{\ell,x} }  \Pi^* \cdot  \UU_\delta \var_\star \bst R_0(\mu,z) \UU_\delta^{-1} \cdot \Pi  e^{i  \mu \delta \blr{\ell,x} } \ove{\matrice{\phi_1 \\ \phi_2}}.
\end{equations}

2. From the definition of $\KKK_\delta[\zeta](z)$, we see that
\begin{equations}
\KKK_\delta[\zeta](z) = \dfrac{1}{2} \left( [-\Delta,\kappa_\delta] + \delta(\kappa_\delta^2-1) W \right)\left((P_\delta[\zeta]-\lambda)^{-1} - (P_{-\delta}[\zeta]-\lambda)^{-1}\right)
\\
= \dfrac{1}{2} \left( 2(D_t \kappa)_\delta \cdot (k' \cdot D_x) + \delta(\kappa_\delta^2-1) W \right)\left(S_\delta(\mu,z) - S_{-\delta}(\mu,z)\right) + \OO_{L^2[\zeta]}\left(\delta^{2/3}\right).
\end{equations}
Thanks to Step 1, the leading order term is $\KK_\delta(\mu,z) \de $
\begin{equations}
\var_\star \left(2(D_t\kappa)_\delta \cdot  \matrice{ k'\cdot D_x \phi_1 \\ k'\cdot D_x \phi_2}^\top + (\kappa_\delta^2-1) W \matrice{ \phi_1 \\ \phi_2}^\top \right) 
 \cdot  e^{-i  \mu \delta \blr{\ell,x} }  \Pi^*  
 \\ 
 \cdot \ \UU_\delta
  \bst R_0(\mu,z) \UU_\delta^{-1} \cdot \Pi  e^{i  \mu \delta \blr{\ell,x} } \ove{\matrice{\phi_1 \\ \phi_2}}.
\end{equations}

3. Because of the definition \eqref{eq:2e} and of Theorem \ref{thm:1},
\begin{equations}
\QQQ_\delta(\zeta,\lambda) = \dfrac{1}{2} \cdot \Big( \big(P_\delta[\zeta]-\lambda\big)^{-1} + \big(P_{-\delta}[\zeta]-\lambda \big)^{-1} \Big) + \dfrac{\kappa_\delta}{2} \cdot \Big( \big( P_\delta[\zeta]-\lambda \big)^{-1} - \big( P_{-\delta}[\zeta]-\lambda \big)^{-1} \Big)
\\
= \dfrac{1}{2} \Big(S_\delta(\mu,z) + S_{-\delta}(\mu,z)\Big) + \dfrac{\kappa_\delta}{2} \cdot \Big(S_\delta(\mu,z) - S_{-\delta}(\mu,z)\Big) + \OO_{L^2[\zeta]}\left(\delta^{-1/3}\right).
\end{equations}
Thanks to the first step the leading order term is $\QQ_\delta(\mu,z) \de $
\begin{equations}
 \dfrac{1}{  \delta} \cdot  \matrice{ \phi_1 \\ \phi_2}^\top  e^{-i  \mu \delta \blr{\ell,x} }  \Pi^* \cdot  \UU_\delta \cdot \matrice{z & \nu_\star k'D_t + \mu \nu_\star\ell \\ \ove{\nu_\star k'} D_t + \mu\ove{\nu_\star \ell} & z} R_0(\mu,z) \cdot \UU_\delta^{-1} \cdot \Pi  e^{i  \mu \delta \blr{\ell,x} } \ove{\matrice{\phi_1 \\ \phi_2}} 
\\
+ 
\kappa_\delta \cdot \dfrac{1}{  \delta} \cdot  \matrice{ \phi_1 \\ \phi_2}^\top  e^{-i  \mu \delta \blr{\ell,x} }  \Pi^* \cdot  \UU_\delta \cdot \var_\star \bst R_0(\mu,z) \cdot  \UU_\delta^{-1} \cdot \Pi  e^{i  \mu \delta \blr{\ell,x} } \ove{\matrice{\phi_1 \\ \phi_2}}.
\end{equations}
A key identity is $\kappa_\delta \Pi^* \UU_\delta = \Pi^* \UU_\delta \kappa$. Therefore, we deduce that
$\QQ_\delta(\mu,z) =$
\begin{equations}
 \dfrac{1}{  \delta} \cdot  \matrice{ \phi_1 \\ \phi_2}^\top  e^{-i  \mu \delta \blr{\ell,x} }  \Pi^* \cdot  \UU_\delta \cdot \matrice{\var_\star \kappa + z & \nu_\star k'D_t + \mu \nu_\star\ell \\ \ove{\nu_\star k'} D_t + \mu\ove{\nu_\star \ell} & - \var_\star \kappa + z}  R_0(\mu,z) \cdot \UU_\delta^{-1} \cdot \Pi  e^{i  \mu \delta \blr{\ell,x} } \ove{\matrice{\phi_1 \\ \phi_2}} .
\end{equations}
The operator
\begin{equations}
\Di(\mu)+z =\matrice{\var_\star \kappa + z & \nu_\star k'D_t + \mu \nu_\star\ell \\ \ove{\nu_\star k'} D_t + \mu\ove{\nu_\star \ell} & - \var_\star \kappa + z}
\end{equations}
emerges and we end up with
\begin{equations}
\QQ_\delta(\mu,z) = \dfrac{1}{  \delta} \cdot  \matrice{ \phi_1 \\ \phi_2}^\top  e^{-i  \mu \delta \blr{\ell,x} }  \Pi^* \cdot  \UU_\delta (\Di(\mu)+z) \cdot R_0(\mu,z) \UU_\delta^{-1} \cdot \Pi  e^{i  \mu \delta \blr{\ell,x} } \ove{\matrice{\phi_1 \\ \phi_2}}.
\end{equations}
This completes the proof. \end{proof}

\end{document}